\documentclass{amsart}
\usepackage{amsfonts}
\usepackage{amsmath}
\usepackage{setspace}
\usepackage{amssymb}
\usepackage[all]{xy}
\usepackage{color}
\usepackage{soul}
\usepackage{mathrsfs}
\usepackage{colortbl}
\usepackage{tikz}
\usetikzlibrary{arrows.meta, decorations.pathmorphing, backgrounds, positioning, fit, petri, patterns}
\usepackage[margin=1in]{geometry}
\usepackage[hidelinks]{hyperref}

\newcommand{\Z}{{\mathbb{Z}}}

\newcommand{\R}{{\mathbb{R}}}

\newcommand{\e}{\varepsilon}
\newcommand{\pushout}{\ar@{}[ul(0.35)]|-{\ulcorner}}
\newcommand{\pullback}{\ar@{}[dr(0.35)]|-{\lrcorner}}

\newcommand{\add}{\mathop{\text{add}}}
\DeclareMathOperator{\Rep}{Rep}
\DeclareMathOperator{\rep}{rep}
\newcommand{\repAR}{\rep_k(A_\R)}

\newcommand{\diag}{\begin{tikzpicture}\draw (0,0) -- (0.3,0); \filldraw (0,0) circle[radius=.3mm]; \filldraw (.3,0) circle[radius=.3mm]; \end{tikzpicture}}

\newcommand{\ReppwfAR}{\Rep_k^{\rm{pwf}}(A_\R)}

\newcommand{\MM}{\mathop{\boldsymbol{\Gamma}}}
\newcommand{\DbAR}{{\mathcal{D}^b(A_\R)}}
\newcommand{\CAR}{{\mathcal{C}(A_\R)}}

\newcommand{\KsplitCAR}{{K_0^{\text{split}}(\CAR)}}

\DeclareMathOperator{\Hom}{Hom}

\newcommand{\supp}{\mathop{\text{supp}}}
\newcommand{\Ext}{\mathop{\text{Ext}}}

\newcommand{\Ind}{\text{Ind}}

\newtheorem{lemma}{Lemma}[subsection]
\newtheorem{proposition}[lemma]{Proposition}
\newtheorem{theorem}[lemma]{Theorem}
\newtheorem{cor}[lemma]{Corollary}

\numberwithin{figure}{subsection}
\numberwithin{table}{subsection}

\theoremstyle{definition}
\newtheorem{definition}[lemma]{Definition}
\newtheorem{remark}[lemma]{Remark}
\newtheorem{example}[lemma]{Example}
\newtheorem{notation}[lemma]{Notation}

\newtheorem{construction}[lemma]{Construction}

\title[Continuous Quivers of Type $A$ (III)]{Continuous Quivers of Type $A$ (III) Embeddings of Cluster Theories}
\author{Kiyoshi Igusa}
\address{Brandeis University in Waltham, MA, USA}
\email{igusa@brandeis.edu}

\author{Job~D.~Rock}
\address{Ghent University in Ghent, Beligum}
\email{jobdrock@gmail.com}

\author{Gordana Todorov}
\address{Northeastern University in Boston, MA, USA}
\email{g.todorov@northeastern.edu}

\date{13 October 2020}

\begin{document}


\maketitle

\begin{abstract}
We continue the work started in parts (I) and (II).
In this part we classify which continuous type $A$ quivers are derived equivalent and introduce the new continuous cluster category with $\mathbf E$-clusters, which are a generalization of clusters.
In the middle we provide a rigorous connection between the previous construction of the continuous cluster category and the new construction.
We conclude with the introduction of a cluster theory, generalizing the notion of a cluster structure.
Using this new notion, we demonstrate how one embeds known type $A$ cluster theories into the new $\mathbf E$-cluster theory in a way compatible with mutation.
This is part (III) in a series of work that will conclude with a continuous generalization of mutation for cluster theories.
\end{abstract}


\section{Introduction}
\subsection{History}
Cluster algebras were introduced by Fomin and Zelevinksy in \cite{ClusterAlgebrasI, ClusterAlgebrasII, ClusterAlgebrasIII, ClusterAlgebrasIV}.
One application is in particle physics to study scattering diagrams (see work of Golden, Goncharov, Spradlin, Vergud, and Volovicha in \cite{scatteringdiagrams}).
Two different categorifications of the cluster structure in a cluster algebra were introduced independently in \cite{BMRRT} and \cite{CalderoChapatonSchiffler2006}.
The first and third authors introduced a continuous version of a cluster category in \cite{IgusaTodorov1}.

In Part (I) \cite{IgusaRockTodorov} the authors defined continuous quivers of type $A$, generalizing quivers of type $A$, and proved results about decomposition of pointwise finite-dimensional representations and the category of finitely generated representations.
In Part (II) \cite{Rock} the second author introduced the Auslander--Reiten space, a continuous analog to the Auslander--Reiten quiver, and proved results relating the AR-space to extensions in the representation category and distinguished triangles in the derived category.
This is Part (III) of this series; the goal is to form a continuous generalization of clusters and mutations.

\subsection{Contributions}
We begin this work with classifying the derived categories of representations of continuous quivers of type $A$.
In the finite case, the derived categories for any orientation of an $A_n$ quiver are all equivalent \cite{Happel}.
This is not true for the continuum.
\begin{theorem}(Theorem \ref{thm:derived classification})
Let $A_\R$ and $A'_\R$ be possibly different orientations of continuous type $A$ quivers.
Then  $\DbAR$ and $\mathcal{D}^b(A'_\R)$ are equivalent as triangulated categories if and only if one of the following holds:
\begin{enumerate}
\item both $A_\R$ and $A'_\R$ have finitely many sinks and sources,
\item the sinks and sources of $A_\R$ and $A'_\R$ are each bounded on exactly one side, or
\item the sinks and sources of $A_\R$ and $A'_\R$ are unbounded on both sides.
\end{enumerate}
\end{theorem}

We then define the new continuous cluster category $\CAR$ as the orbit category of the doubling of $\DbAR$ via almost-shift (same method as in \cite{IgusaTodorov1}).
We define $\mathbf E$-clusters (Definition \ref{def:E-cluster}) in $\CAR$ and prove that if an element is $\mathbf E$-mutable then the choice is unique and yields another $\mathbf E$-cluster.
In an $\mathbf E$-cluster $T$ there may exist an indecomposable $X$ such that $X$ is not $\mathbf E$-mutable in $T$ but is $\mathbf E$-mutable in some other $\mathbf E$-cluster $T'$.
Thus we will not call an indecomposable frozen if it is not $\mathbf E$-mutable in some particular $\mathbf E$-cluster.
\begin{theorem}(Theorem \ref{thm:mutation theorem})
Let $T$ be a $\mathbf E$-cluster and $V\in T$ be $\mathbf E$-mutable with choice $W$.
Then $(T\setminus \{V\})\cup\{W\}$ is a $\mathbf E$-cluster and any other choice $W'$ for $V$ is isomorphic to $W$.
\end{theorem}

The new construction is not unrelated to the previous construction in \cite{IgusaTodorov1}.
In fact, the new continuous cluster category and the previous continuous cluster category are related by a localization of the derived category $\DbAR$ when $A_\R$ has finitely-many sinks and sources.
In the previous construction, the category playing the role of the derived category was denoted $\mathcal D_\pi$.
\begin{theorem}(Theorems \ref{thm:triangulated equivalent} and \ref{thm:localized cluster category})

Assume $A_\R$ has finitely-many sinks and sources.
Then there exist triangulated localizations $\DbAR\to \DbAR[\mathcal M^{-1}]$ and $\CAR\to \CAR[\mathcal N^{-1}]$ and triangulated equivalences of categories $G:\DbAR[{\mathcal M}^{-1}]\to \mathcal D_\pi$ and $H:\CAR[\mathcal N^{-1}]\to \mathcal C_\pi$.
\end{theorem}

The key difference between these new $\mathbf E$-clusters and existing cluster structures (including the previous construction) is the weakening of the requirement for mutable elements.
Instead of \emph{there exists a unique choice} we require \emph{there exists none or one}.
This leads to our final contribution: that of cluster theories (Definition \ref{def:cluster theory}) which generalize cluster structures.
In this definition a cluster theory on a skeletally-small Krull-Schmidt additive category $\mathcal C$ is a groupoid $\mathscr T_{\mathbf P}(\mathcal C)$ induced by a pairwise compatibility condition $\mathbf P$ on $\Ind(\mathcal C)$.
In this language $\mathbf E$-clusters form the $\mathbf E$-cluster theory of $\CAR$ (Example \ref{xmp:IRT model}).
With this new definition we define a rigorous description of an embedding of cluster theories (Definition \ref{def:cluster embedding}).

{~}

\begin{theorem}\label{thm:thmD}(Theorems \ref{thm:finite embedding}, \ref{thm:infinite embedding}, \ref{thm:continuous embedding})
{~}
\begin{itemize}
\item For any $A_n$ quiver there is an embedding of cluster theories $\mathscr T_{\mathbf N_n}(\mathcal C(A_n))\to \mathscr T_{\mathbf E}(\CAR)$.
\item For the straight $A_\infty$ quiver there is an embedding of cluster theories $\mathscr T_{\mathbf N_\infty}(\mathcal C(A_\infty))\to \mathscr T_{\mathbf E}(\CAR)$.
\item There is an embedding of cluster theories from the previous construction's cluster theory in \cite{IgusaTodorov1} to our new construction: $\mathscr T_{\mathbf N_\R}(\mathcal C_\pi)\to \mathscr T_{\mathbf E}(\CAR)$.
\end{itemize}
\end{theorem}
\noindent Theorem \ref{thm:thmD} uses the cluster theories from \cite{CalderoChapatonSchiffler2006}, \cite{HolmJorgensen}, and \cite{IgusaTodorov1}, respectively.

\subsection{Future Work}
The final part of this series addresses the continuous generalization of the mutation and the embedding of cluster theories from the groupoid of $A_\infty$ clusters using the completed infinity-gon (introduced in \cite{BaurGraz}).
In particular, continuous mutation allows for the encoding of the transfinite mutation introduced in the same paper along side new mutations of $\mathbf E$-clusters.
These embeddings and continuous structure should be useful as intuition, if not also as machinery, for a continuous cluster algebra that can handle the embeddings of all existing type $A$ cluster algebras.
The final part of this series also addresses generalizations of the geometric models of type $A$ cluster theories.

\subsection{Acknowledgements}
The authors would like to thank Ralf Schiffler for creating the Cluster Algebra Summer School in 2017 where the idea for this project was conceived.
The second author would like to thank Eric Hanson for helpful discussions.

The majority of this work was completed while the second author was a graduate student at Brandeis University and they would like to thank the university for its hospitality.

\section{Parts (I) and (II)}
In this section we recall the requisite information from parts (I) and (II) of this series.

\subsection{Continuous Quivers of Type $A$ and Their Representations}
In this subsection we recall relevant definitions and theorems about continuous quivers of type $A$ and their representations from Part (I) of this series \cite{IgusaRockTodorov}.

Fix a field $k$.
We begin with definitions of a continuous quivers of type $A$ and its representations.
However, we first use a picture to illustrate the concept.
\begin{displaymath}\begin{tikzpicture}
\draw[dotted] (-.5,-.5) -- (.1,.1);
\draw[dotted] (4.3,-.7) -- (4.8,-.2);
\draw[thick] (0,0) -- (1,1) -- (2.3,-0.3) -- (2.8, 0.2) -- (4, -1) -- (4.4, -0.6);
\draw[thick,->] (1,1) -- (.5,.5);
\draw[thick,->] (1,1) -- (1.65,0.35);
\draw[thick,->] (2.8,0.2) -- (2.55,-0.05);
\draw[thick,->] (2.8,0.2) -- (3.4,-0.4);
\draw[thick,->] (4.4,-0.6) -- (4.2,-0.8);
\filldraw (1,1) circle[radius=.5mm];
\filldraw (2.3,-.3) circle[radius=.5mm];
\filldraw (2.8,.2) circle[radius=.5mm];
\filldraw (4,-1) circle[radius=.5mm];
\draw (1,1) node [anchor=south] {$s_{2n-1}$};
\draw (2.3,-0.3) node[anchor=north] {$s_{2n}$};
\draw (2.8,0.2) node[anchor=south] {$s_{2n+1}$};
\draw (4,-1) node[anchor=north] {$s_{2n+2}$};
\end{tikzpicture}\end{displaymath}

\begin{definition}\label{def:AR}
A \ul{quiver of continuous type $A$}, denoted by $A_\R$, is a triple $(\R,S,\preceq)$, where:
\begin{enumerate}
\item 
\begin{enumerate}
\item$S\subset \R$ is a discrete subset, possibly empty, with no accumulation points.
\item Order on $S\cup\{\pm\infty\}$ is induced by the order of $\R$, and $-\infty<s<+\infty$ for $\forall s\in S$.
\item Elements of $S\cup\{\pm\infty\}$ are indexed by a subset of $\Z\cup\{\pm\infty\}$ so that $s_n$ denotes the element of 
$S\cup\{\pm\infty\}$ with index $n$. The indexing must adhere to the following two conditions:
\begin{itemize}
\item[i1] There exists $s_0\in S\cup\{\pm\infty\}$.
\item[i2] If $m\leq n\in\Z\cup\{\pm\infty\}$ and $s_m,s_n\in S\cup\{\pm\infty\}$ then for all $p\in\Z\cup\{\pm\infty\}$ such that $m\leq p \leq n$ the element $s_p$ is in $S\cup\{\pm\infty\}$.
\end{itemize}
\end{enumerate}
\item New partial order $\preceq$ on $\R$, which we call  the \ul{orientation} of $A_\R$, is defined as:
\begin{itemize}
\item[p1\ ] The $\preceq$ order between consecutive elements of $S\cup\{\pm\infty\}$ does not change.
\item[p2\ ] Order reverses at each element of $S$.
\item[p3\ ] If $n$ is even $s_n$ is a sink.
\item[p3'] If $n$ is odd $s_n$ is a source.
\end{itemize}
\end{enumerate}
\end{definition}

\begin{definition}\label{def:representation}
Let $A_\R=(\R,S\preceq)$ be a continuous quiver of type $A$.
A \ul{representation} $V$ of $A_\R$ is the following data:
\begin{itemize}
\item A vector space $V(x)$ for {each} $x\in \R$.
\item For every pair $y\preceq x$ in $A_\R$ a linear map $V(x,y):V(x)\to V(y)$ such that if $z\preceq y \preceq x$ then $V(x,z)=V(y,z)\circ V(x,y)$.
\end{itemize}
We say $V$ is \ul{pointwise finite-dimensional} if $\dim V(x) < \infty$ for all $x\in\R$.
\end{definition}

\begin{definition}\label{note:interval indecomposable}
Let $A_\R$ be a continuous quiver of type $A$ and $I\subset \R$ be an interval.
We denote by $M_I$ the representation of $A_\R$ where
\begin{align*}
M_I(x) &= \left\{ \begin{array}{ll} k & x\in I \\ 0 & \text{otherwise} \end{array} \right. &
M_I(x,y) &= \left\{ \begin{array}{ll} 1_k & y\preceq x\in I \\ 0 & \text{otherwise.} \end{array} \right.
\end{align*}
We call $M_I$ an \ul{interval indecomposable}.
\end{definition}

\begin{notation}\label{note:intervals}
Let $a<b\in\R\cup\{\pm \infty\}$.
By the notation $|a,b|$ we mean an interval subset of $\R$ whose endpoints are $a$ and $b$.
The $|$'s indicate that $a$ and $b$ may or may not be in the interval.
In practice this is (i) clear from context, (ii) does not matter in its context, or (iii) intentionally left as an unknown.
For example, we may write $M_{|a,b|}$ to refer to one of four possible interval indecomposables.
There is one exception: if $a$ or $b$ is $-\infty$ or $+\infty$, respectively, then the $|$ always means $($ or $)$, respectively.
\end{notation}

There are essentially three important results that we will use from \cite{IgusaRockTodorov}.
The first is about the structure of a pointwise finite-dimensional representation of a continuous quiver of type $A$.
It is important to note the last statement of the theorem recovers a result by Botnan and Crawley-Boevey in \cite{BotnanCrawleyBoevey}.

\begin{theorem}(Theorems 2.3.2 and 2.4.13 in \cite{IgusaRockTodorov})
Let $A_\R$ be a continuous quiver of type $A$.
For any interval $I\subset \R$, the representation $M_I$ of $A_\R$ is indecomposable.
Any indecomposable pointwise finite-dimensional representation of $A_\R$ is isomorphic to $M_I$ for some interval $I$.
Furthermore, for any indecomposable representations $V$ and $W$ of $A_\R$, $V\cong W$ if and only if $\supp V = \supp W$.
Finally, any pointwise finite-dimensional representation $V$ of $A_\R$ is the direct sum of interval indecomposables.
\end{theorem}

\begin{definition}\label{def:pwf reps}
Let $A_\R$ be a continuous quiver of type $A$.
By $\ReppwfAR$ we denote the category of pointwise finite-dimensional representations of $A_\R$.
That is, for any representation $V$ in $\ReppwfAR$ and $x\in \R$, the $k$-vector space $V(x)$ is finite-dimensional.
\end{definition}
In Section \ref{sec:new continuous cluster category} we will need the following classification of indecomposable pointwise finite-dimensional projective representations.

\begin{theorem}\label{thm:projs}(Theorem 2.1.6 and Remark 2.4.16 in \cite{IgusaRockTodorov})
Let $P$ be a projective indecomposable in $\ReppwfAR$.
Then there exists $a\in \R\cup\{\pm\infty\}$ such that $P$ is isomorphic to one of the following indecomposables: $P_a$, $P_{(a}$, or $P_{a)}$:
\begin{align*}
P_a(x) &= \left\{\begin{array}{ll} k & x \preceq a \\ 0 & \text{otherwise} \end{array}\right. &
P_a(x,y) &= \left\{\begin{array}{ll} 1_k & y\preceq x \preceq a \\ 0 & \text{otherwise} \end{array}\right. \\
P_{(a}(x) &= \left\{\begin{array}{ll}  k & x\preceq a\text{ and } x>a\text{ in }\R  \\ 0 & \text{otherwise} \end{array}\right. &
P_{(a}(x,y) &= \left\{\begin{array}{ll} 1_k & y\preceq x \preceq a \text{ and } x,y>a \\ 0 & \text{otherwise} \end{array}\right. \\
P_{a)}(x) &= \left\{\begin{array}{ll}  k & x\preceq a\text{ and }x<a\text{ in }\R  \\ 0 & \text{otherwise} \end{array}\right. &
P_{a)}(x,y) &= \left\{\begin{array}{ll} 1_k & y\preceq x \preceq a \text{ and }x,y<a \\ 0 & \text{otherwise} \end{array}\right.
\end{align*}
\end{theorem}

\begin{definition}\label{def:fg reps}
Let $A_\R$ be a continuous quiver of type $A$.
By $\repAR$ we denote the full subcategory of $\ReppwfAR$ whose objects are finitely generated by the indecomposable projectives in Theorem \ref{thm:projs}.
In particular, the indecomposable projectives in both categories are the same.
\end{definition}

\begin{theorem}\label{thm:little rep}(Theorem 3.0.1 in \cite{IgusaRockTodorov})
Let $A_\R$ be a continuous quiver of type $A$. The following hold.
\begin{enumerate}
\item For any pair of indecomposable representations $M_I$ and $M_J$ in $\repAR$ \begin{displaymath} \Hom(M_I,M_J)\cong k \text{ or }\Hom(M_I,M_J)=0. \end{displaymath}
\item The category $\repAR$ is abelian.
\item The category $\repAR$ is Krull-Schmidt but not artinian.
\item The global dimension of $\repAR$ is 1.
\item For any indecomposables $M_I$ and $M_J$ in $\repAR$, either $\Ext^1(M_I,M_J)\cong k$ or $\Ext^1(M_I,M_J)=0$.
\item The category $\repAR$ has some, but not all, Auslander--Reiten sequences (fully classified in \cite[Table 3.1.3]{Rock}).
\end{enumerate}
\end{theorem}

\subsection{The AR-space of $\repAR$ and $\DbAR$}
 In this subsection we recall the necessary definitions and theorems from Part (II) of this series, \cite{Rock}.
We start with $\lambda$ functions, which are used to construct the Auslander--Reiten space, or AR-space, of $\repAR$ and $\DbAR$, the latter of which is constructed in the usual way out of an abelian category.

\begin{definition}\label{def:lambda function}
Let $z\in\R$. Then $z=2n\pi + w$ for $n\in\Z$ and $0\leq w \leq 2\pi$.
So let the function $\lambda:\R\to\R$ be given by
\begin{displaymath}
\lambda(z)=\lambda(2n\pi + w) = \left\{ \begin{array}{ll} w-\frac{\pi}{2} & 0\leq w \leq \pi \\ -w + \frac{3\pi}{2} & \pi\leq w \leq 2\pi. \end{array} \right.
\end{displaymath}
A \ul{$\lambda$ function} is a function $\R\to[-\frac{\pi}{2},\frac{\pi}{2}]$ defined by $x\mapsto \lambda(x-\kappa)$ for a fixed $\kappa\in[-\pi,\pi]$.
\end{definition}

Recall $\MM:\text{Ind}(\repAR)\to\R\times[-\frac{\pi}{2},\frac{\pi}{2}]$ is a function from the isomorphism classes of indecomposables in $\repAR$ to the real plane.
The map $\MM^b$ is the natural extension of $\MM$ to all the indecomposables of $\DbAR$ where  $\MM^b V[1] = (x+\pi,-y)$ for each indecomposable $V$ with $\MM^bV =(x,y)$.
These functions are used to define the continuous generalization of the Auslander--Reiten quiver, called the Auslander--Reiten space, of $\repAR$ and $\DbAR$ (\cite[Definitions 4.1.9 and 5.2.5]{Rock}).

Recall each (isomorphism class of) indecomposables in $\repAR$ or $\DbAR$ have a position (\cite[Definition 4.1.2]{Rock}) 1, 2, 3, or 4 in the AR-space of $\repAR$ or $\DbAR$, respectively.
The positions are to be thought of as occupying the four corners of a diamond:
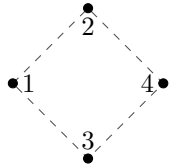
\begin{displaymath}\begin{tikzpicture}
\draw[dashed, draw opacity=.7] (0,0) -- (1,1) -- (2,0) -- (1,-1) -- (0,0);
\draw (0,0) node[anchor=west] {1};
\draw (1,1) node[anchor=north] {2};
\draw (1,-1) node[anchor=south]{3};
\draw (2,0) node[anchor=east]{4};
\filldraw[fill=black] (0,0) circle [radius=0.6mm];
\filldraw[fill=black] (1,1) circle [radius=0.6mm];
\filldraw[fill=black] (1,-1) circle [radius=0.6mm];
\filldraw[fill=black] (2,0) circle [radius=0.6mm];
\end{tikzpicture}\end{displaymath}
In particular, two indecomposables $V$ and $W$ in $\repAR$ or $\DbAR$ are isomorphic if and only if their position and image under $\MM$ or $\MM^b$, respectively, are the same.

\begin{example}\label{xmp:AR space of DbAR}
Recall from \cite{Rock} that when $A_\R$ has the straight descending orientation the following is part of the AR-space:
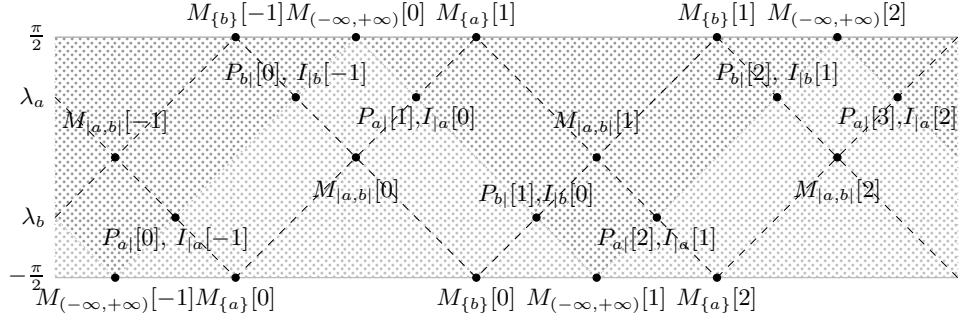
\begin{displaymath} \begin{tikzpicture}[scale=.8]
\draw[draw=white!60!black] (-5,2) -- (10,2);
\draw[draw=white!70!black] (-5,-2) -- (10,-2);
\draw[draw=white!70!black, dotted, pattern=crosshatch dots, pattern color=white!75!black] (-5,-1) -- (-5,-2) -- (-4,-2) -- (-5,-1);
\draw[draw=white!70!black, dotted, pattern=crosshatch dots, pattern color=white!60!black] (-5,2) -- (-5,-1) -- (-4,-2) -- (0, 2) -- (-5,2);
\draw[draw=white!70!black, dotted, pattern=crosshatch dots, pattern color=white!75!black] (-4,-2) -- (0,2) -- (4,-2) -- (-4,-2);
\draw[draw=white!70!black, dotted, pattern=crosshatch dots, pattern color=white!60!black] (0,2) -- (4,-2) -- (8,2) -- (0,2);
\draw[draw=white!70!black, dotted, pattern=crosshatch dots, pattern color=white!75!black] (4,-2) -- (8,2) -- (10,0) -- (10,-2) -- (4,-2);
\draw[draw=white!70!black, dotted, pattern=crosshatch dots, pattern color=white!60!black] (8,2) -- (10,0) -- (10,2) -- (8,2);
\draw[dashed] (-5,-1) -- (-2,2) -- (2,-2) -- (6,2) -- (10,-2);
\draw[dashed] (-5, 1) -- (-2,-2) -- (2,2) -- (6,-2) -- (10,2);
\draw[fill=black] (0,2) circle[radius=0.6mm];
\draw (0,2) node [anchor=south] {\small $M_{(-\infty,+\infty)}[0]$};
\draw[fill=black] (8,2) circle[radius=0.6mm];
\draw (8,2) node [anchor=south] {\small $M_{(-\infty,+\infty)}[2]$};
\draw[fill=black] (-4,-2) circle[radius=0.6mm];
\draw (4,-2) node [anchor=north] {\small $M_{(-\infty,+\infty)}[1]$};
\draw[fill=black] (4,-2) circle[radius=0.6mm];
\draw (-4,-2) node [anchor=north] {\small $M_{(-\infty,+\infty)}[-1]$};
\draw[fill=black] (0,0) circle[radius=0.6mm];
\draw (0,-0.2) node [anchor=north] {\small $M_{|a,b|}[0]$};
\draw[fill=black] (-4,0) circle[radius=0.6mm];
\draw (-4,0.2) node [anchor=south] {\small $M_{|a,b|}[-1]$};
\draw[fill=black] (4,0) circle[radius=0.6mm];
\draw (4,0.2) node [anchor=south] {\small $M_{|a,b|}[1]$};
\draw[fill=black] (8,0) circle[radius=0.6mm];
\draw (8,-0.2) node [anchor=north] {\small $M_{|a,b|}[2]$};
\draw[fill=black] (-1,1) circle[radius=0.6mm];
\draw (-1,1) node [anchor=south] {\small $P_{b|}[0]$, $I_{|b}[-1]$};
\draw[fill=black] (-3,-1) circle[radius=0.6mm];
\draw (-3,-1) node [anchor=north] {\small $P_{a|}[0]$, $I_{|a}[-1]$};
\draw[fill=black] (1,1) circle[radius=0.6mm];
\draw (1,1) node [anchor=north] {\small $P_{a|}[1]$,$I_{|a}[0]$};
\draw[fill=black] (3,-1) circle[radius=0.6mm];
\draw (3,-1) node [anchor=south] {\small $P_{b|}[1]$,$I_{|b}[0]$};
\draw[fill=black] (7,1) circle[radius=0.6mm];
\draw (7,1) node [anchor=south] {\small $P_{b|}[2]$, $I_{|b}[1]$};
\draw[fill=black] (5,-1) circle[radius=0.6mm];
\draw (5,-1) node [anchor=north] {\small $P_{a|}[2]$,$I_{|a}[1]$};
\draw[fill=black] (9,1) circle[radius=0.6mm];
\draw (9,1) node [anchor=north] {\small $P_{a|}[3]$,$I_{|a}[2]$};
\draw[fill=black] (-2,2) circle[radius=0.6mm];
\draw (-2,2) node [anchor=south] {\small $M_{\{b\}}[-1]$};
\draw[fill=black] (-2,-2) circle[radius=0.6mm];
\draw (-2,-2) node [anchor=north] {\small $M_{\{a\}}[0]$};
\draw[fill=black] (2,-2) circle[radius=0.6mm];
\draw (2,-2) node [anchor=north] {\small $M_{\{b\}}[0]$};
\draw[fill=black] (2,2) circle[radius=0.6mm];
\draw (2,2) node [anchor=south] {\small $M_{\{a\}}[1]$};
\draw[fill=black] (6,2) circle[radius=0.6mm];
\draw (6,2) node [anchor=south] {\small $M_{\{b\}}[1]$};
\draw[fill=black] (6,-2) circle[radius=0.6mm];
\draw (6,-2) node [anchor=north] {\small $M_{\{a\}}[2]$};
\draw (-5,-1) node [anchor=east] {\small $\lambda_b$};
\draw (-5,1) node [anchor=east] {\small $\lambda_a$};
\draw(-5,-2) node [anchor=east] {\small $-\frac{\pi}{2}$};
\draw(-5,2) node [anchor=east] {\small $\frac{\pi}{2}$};
\end{tikzpicture} \end{displaymath}
\end{example}

In \cite{Rock}, an extra generalized metric is defined; this allows for the description of lines and their slopes in the AR-space of $\repAR$ and $\DbAR$.
In particular this also allows for rectangles and almost-complete rectangles in the AR-space of these categories as well.
This results in the following theorems that will be used in Section \ref{sec:derived equivalence}.
The ``good slopes'' can be thought of as analogous to a $45^\circ$ angle and by ``scaling'' we mean ``scaling of morphisms''.

\begin{theorem}\label{thm:extensions are rectangles}(Theorem 4.3.11 in \cite{Rock})
Let $V=M_{|a,b|}$ and $W=M_{|c,d|}$ be indecomposables in $\repAR$ such that $V\not\cong W$.
Then there is a nontrivial extension $V\hookrightarrow E\twoheadrightarrow W$ if and only if there exists a rectangle or almost complete rectangle whose corners are the indecomposables in the sequence with $V$ as the left-most corner and $W$ as the right-most corner.
\begin{itemize}
\item If the rectangle is complete $E$ is a direct sum of two indecomposables.
\item If the rectangle is almost complete $E$ is indecomposable.
\end{itemize}
Furthermore, there is a bijection
\begin{displaymath}\begin{tikzpicture}
\draw (0,.2) node {$\left\{\begin{array}{c}\text{rectangles and almost complete rectangles with} \\ \text{``good'' slopes of sides in the AR-space of }\repAR\end{array}\right\}$};
\draw [<->, thick] (0,-.2) -- (0,-1.2);
\draw (0,-.7) node[anchor=west] {$\cong$};
\draw (0,-1.6) node {$\left\{\begin{array}{c}\text{nontrivial extensions of indecomposables by indecomposables}\\ \text{up to scaling and isomorphisms}\end{array}\right\}$};
\end{tikzpicture}\end{displaymath}
\end{theorem}

In the following theorem, we consider a triangle and any of its rotations to be distinct for the purposes of the statement of the theorem. We also say ``nontrivial triangle'' to mean a distinguished triangle that is not of the form $(A\to A\to 0\to)$, $(A\to 0 \to A[1]\to)$, or $(0\to A\to A\to)$. 

\begin{theorem}\label{thm:triangles are rectangles}(Theorem 5.2.10 in \cite{Rock})
Let $V=M_{|a,b|}[m]$ and $W=M_{|c,d|}[n]$ be indecomposables in $\DbAR$ such that $V\not\cong W$.
Then there is a nontrivial distinguished triangle $V\to U\to W\to$ if and only if there exists a rectangle or almost complete rectangle in the AR-space of $\DbAR$ whose corners are the indecomposables in the triangle with $V$ as the left-most corner and $W$ as the right-most corner.

\begin{itemize}
\item If the rectangle is complete $E$ is a direct sum of two indecomposables.
\item If the rectangle is almost complete $E$ is indecomposable.
\end{itemize}

Furthermore, there is a bijection
\begin{displaymath}\begin{tikzpicture}
\draw (0,.2) node {$\left\{\begin{array}{c}\text{rectangles and almost complete rectangles with} \\ \text{``good'' slopes of sides in the AR-space of }\DbAR\end{array}\right\}$};
\draw [<->, thick] (0,-.2) -- (0,-1.2);
\draw (0,-.7) node[anchor=west] {$\cong$};
\draw (0,-1.6) node {$\left\{\begin{array}{c}\text{nontrivial triangles with first and third term indecomposable}\\ \text{up to scaling and isomorphisms}\end{array}\right\}$};
\end{tikzpicture}\end{displaymath}
\end{theorem}

\section{Derived Equivalence}\label{sec:derived equivalence}
In this section we classify the derived categories of finitely generated representations of quivers of continuous type $A$.

\subsection{The Octahedral Axiom}
We recall the octahedral axiom of a triangulated category, which we use explicitly in Lemma \ref{lem:geometric cones 1}.
In particular we use \cite[Proposition 1.4.6]{Neeman} in Neeman's book, which the author proves is equivalent to the octahedral axiom.
Suppose the following are distinguished triangles:
\begin{displaymath}\xymatrix@R=2ex@C=8ex{
U\ar[r]^-f & V\ar[r]^-{g'} & W' \ar[r]^-{h'} & U[1]\\
V\ar[r]^-g & W\ar[r]^-{h} & U'\ar[r]^-{f'} &  V[1].
}\end{displaymath}
Then the following distinguished triangles also exist:
\begin{displaymath}\xymatrix@R=2ex@C=8ex{
U\ar[r]^-{g\circ f}  & W \ar[r]^-i & V' \ar[r]^-{j} & U[1] \\
W' \ar[r]^-{i'} & V' \ar[r]^-{j'} & U' \ar[r]^-{g'[1]\circ f'} &  W'[1],
}\end{displaymath}
such that $h=j'\circ i$ and $h'=j\circ i'$ (note the mismatched primes).
This is often drawn as the lower and upper part of an octahedron.
Another way of stating the axiom is that given the ``lower cap'' on the left below the ``upper cap'' on the right exists as well:
\begin{displaymath}\xymatrix{
U' \ar[dd]_-{[1]} \ar[dr]^-{[1]} & & W \ar[ll] & & U'\ar[dd]_-{[1]} & & W\ar[ll] \ar[dl] \\
& V \ar[ur] \ar[dl] & & & & V' \ar[ul] \ar[dr]_-{[1]} & \\
W' \ar[rr]_-{[1]} & & U\ar[uu] \ar[ul]& & W' \ar[rr]_-{[1]} \ar[ur] & & U.\ar[uu]
}\end{displaymath}
By an immediate result in \cite{Neeman}, we also have the following distinguished triangles:
\begin{displaymath}\xymatrix@R=3ex@C=16ex{
V \ar[r]^-{\left[\begin{array}{l}g\\ g'\end{array}\right]}  & W\oplus W' \ar[r]^-{\left[\begin{array}{ll}i & -i'\end{array}\right]} & V' \ar[r]^-{f[1]\circ j=f'\circ j'} & V[1]. \\
V' \ar[r]_-{\left[\begin{array}{l} j \\ j'\end{array}\right]} & U[1] \oplus U' \ar[r]_-{\left[\begin{array}{ll} f[1] & -f'\end{array}\right]} & V[1] \ar[r]_-{(i\circ g)[1] = (i'\circ g')[1]} & V'[1].
}\end{displaymath}

We will need the following well-known facts, also from \cite{Neeman}, albeit in a slightly different form.
Thus, we will state it as a proposition and provide a concise proof.
\begin{proposition}\label{prop:easy triangle splice}
Let $\mathcal D$ be a triangulated category and let
\begin{displaymath}\xymatrix@R=2ex@C=8ex{
V\ar[r]^-f & W_1 \ar[r]^-{p_1} & U_1 \ar[r]^-{q_1} & V[1]
}\end{displaymath}
be a distinguished triangle where $f$ is nontrivial.
Let $h:W_1\to W_2$ such that $h$ and $h\circ f$ are nontrivial. Then
\begin{displaymath}\xymatrix@C=10ex{
V \ar[r]^-{\left[\begin{array}{cc} f \\ h\circ f \end{array}\right]}& W_1\oplus W_2 \ar[r]^-{\left[\begin{array}{lr}p_1 & 0 \\ h & -1 \end{array}\right]} &  U_1\oplus W_2\ar[r]^-{\left[\begin{array}{cc}q_1 & 0\end{array}\right]} & V[1] 
} \tag{1} \end{displaymath} exists as a distinguished triangle in $\mathcal D$.
Dually, given the following distinguished triangle:
\begin{displaymath}\xymatrix@R=2ex@C=8ex{
V_2 \ar[r]^-g & W \ar[r]^-{p_2} & U_2 \ar[r]^-{q_2} & V_2[2],
}\end{displaymath}
and morphism $h:V_1\to V_2$ where $g$, $h$, and $g\circ h$ are nontrivial there exists
\begin{displaymath}\xymatrix@C=10ex{
V_1\oplus V_2 \ar[r]^-{\left[\begin{array}{cc} g\circ h & g \end{array}\right]} & W \ar[r]^-{\left[\begin{array}{c} 0\\ p_2\end{array}\right]} & V_1[1]\oplus U_2 \ar[r]^-{\left[\begin{array}{ll} 1 & 0 \\ h & q_2 \end{array}\right]} & (V_1\oplus V_2)[1] 
} \tag{2} \end{displaymath}
as a distinguished triangle in $\mathcal D$.
\end{proposition}
\begin{proof}
We prove that (1) is a distinguished triangle.
The proof of (2) is similar.
Set $g=h\circ f$.
We start with the distinguished triangle from the statement and the following distinguished triangles:
\begin{displaymath}\xymatrix@R=2ex@C=8ex{
W_1 \ar[r]^-h & W_2 \ar[r] & E \ar[r] & W_1[1] \\
V \ar[r]^-g & W_2 \ar[r]^-{p_2} & U_2 \ar[r]^-{q_2} & V[1].
}\end{displaymath}
By the octahedral axiom this yields the distinguished triangles
\begin{displaymath}\xymatrix@R=2ex@C=8ex{
U_1 \ar[r]^-r & U_2 \ar[r]^-s & E \ar[r]^-t & U_1[1] \\
W_1 \ar[r]^-{[p_1\, h]^t} & U_1\oplus W_2 \ar[r]^-{[r\, {-}p_2]} & U_2 \ar[r]^-{q_3} & W_1[1].
}\end{displaymath}
We start again with the triangles
\begin{displaymath}\xymatrix@R=2ex@C=8ex{
U_2[-1] \ar[r]^-{q_2[-1]} & V \ar[r]^-{g} & W_2 \ar[r]^-{p_2} & U_2 \\
V \ar[r]^-f & W_1 \ar[r]^-{p_1} & U_1\ar[r]^-{q_1}\ar[r] & V[1] \\
U_2[-1] \ar[r]^-{q_3[-1]}\ar[r] & W_1 \ar[r]^-{[p_1\, h]^t} & U_1\oplus W_2 \ar[r]^-{[r\, {-}p_2]} & U_2.
}\end{displaymath}
Noting that $[p_1\, h]^t \circ f = [0\, 1]^t \circ g$ we see $q_3=f[1]\circ q_2$.
So, we may apply the octahedral axiom again to obtain the desired triangle (1).
\end{proof}

We specifically desire to use Proposition \ref{prop:easy triangle splice} as the following corollary.
\begin{cor}\label{cor:easy triangle splice}
Let $\mathcal D$ be a triangulated category.
Consider the distinguished triangles from Proposition \ref{prop:easy triangle splice}:
\begin{displaymath}\xymatrix@R=2ex@C=8ex{
V\ar[r]^-f & W_1 \ar[r]^-{p_1} & U_1 \ar[r]^-{q_1} & V[1] \\
V\ar[r]^-{h\circ f} & W_2 \ar[r]^-{p_2} & U_2 \ar[r]^-{q_2} & V[1].
}\end{displaymath}
where $h:W_1\to W_2$, $f$, and $h\circ f$ are nontrivial.
If $U_2$ is not isomorphic to summands of $V$, $W_1$, $W_2$, or $U_1$ then it does not appear as a summand in any distinguished triangle of the form
\begin{displaymath}\xymatrix@R=2ex@C=8ex{
V \ar[r]^-{\left[\begin{array}{cc} f \\ h\circ f \end{array}\right]} & W_1\oplus W_2 \ar[r] & U\ar[r] & V[1].}\tag{*}
\end{displaymath}
\end{cor}
\begin{proof}
The distinguished triangle $*$ will be isomorphic to distinguished triangle (1) in Proposition \ref{prop:easy triangle splice}.
Since $U_2$ does not appear in (1) as a summand it will not appear in $*$ as a summand.
\end{proof}

\begin{proposition}\label{prop:one way hom}
Let $V$ and $W$ be indecomposables in $\DbAR$ such that $V\not\cong W$.
Then $\Hom(V,W)\cong k$ implies $\Hom(W,V)=0$.
\end{proposition}
\begin{proof}
This follows directly from \cite[Proposition 4.4.2, Lemma 5.2.9]{Rock}.\end{proof}

\begin{construction}\label{con:nice summing}
Consider a finite direct sum $W=\bigoplus W_i$ of indecomposables in $\DbAR$.
Using Proposition \ref{prop:one way hom} we can consider a subset $X_W$ of summands of $W$ determined by
\begin{itemize}
\item If $W_i,W_j\in X_W$ and $i\neq j$ then $\Hom(W_i,W_j)=0$.
\item If $W_i\in X_W$ and there exists $W_j$ such that $\Hom(W_j,W_i)\cong k$ then $W_i\cong W_j$.
\end{itemize}
(Notice $X_W$ may not be unique.)
Let $Y_W$ be the multisubset of summands of $W$ such that $\bigoplus_{X_W\amalg Y_W} W_i \cong W$.
Below is an example depicted in the AR-space of $\DbAR$; members of $X_W$ are filled in and members of $Y_W$ are not.
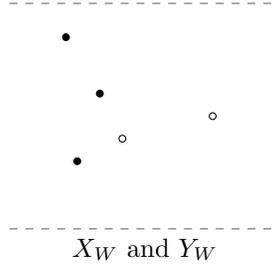
\begin{displaymath}\begin{tikzpicture}[scale=3]
\draw[dashed, white!40!black] (0.2,1) -- (1.4,1);
\draw[dashed, white!40!black] (0.2,0) -- (1.4,0);
\filldraw[fill=black, draw=black] (0.5,0.3) circle[radius=0.15mm];
\filldraw[fill=black, draw=black] (0.6,0.6) circle[radius=0.15mm];
\filldraw[fill=white, draw=black] (0.7,0.4) circle[radius=0.15mm];
\filldraw[fill=black, draw=black] (0.45,0.85) circle[radius=0.15mm];
\filldraw[fill=white, draw=black] (1.1,0.5) circle[radius=0.15mm];
\filldraw (0.8,0) node[anchor=north] {$X_W$ and $Y_W$};
\end{tikzpicture}\end{displaymath}

Now consider a finite sum $V=\bigoplus V_i$ of indecomposables in $\DbAR$.
Consider the subset $_VX$ of summands of $V$ determined by
\begin{itemize}
\item If $V_i, V_j\in _VX$ and $i\neq j$ then $\Hom(V_i,V_j)=0$.
\item If $V_i\in {_VX}$ and there exists $W_j$ such that $\Hom(V_i,V_j)\cong k$ then $V_i\cong V_j$.
\end{itemize}
(Notice $_VX$ may not be unique.)
Let $_VY$ be the multisubset of summands of $V$ such that $\bigoplus_{_VX\amalg _VY} V_i \cong V$.
\end{construction}

\begin{lemma}\label{lem:easy splice first}
Let $f:V\to W$ be a morphism in $\DbAR$.
Suppose $V$ is indecomposable and for each summand of $W$ the composition $f_i: V\stackrel{f}{\to} W\stackrel{\pi}{\to} W_i$ is nonzero.
Choose a set $X_W$ and $Y_W$ as in Construction \ref{con:nice summing}.
Denoting by $\pi$ the projection $W\twoheadrightarrow \bigoplus_{Y_W} W_i$, if
\begin{displaymath}\xymatrix@C=8ex{
V\ar[r] ^-{\left[ f_{i}\right]_{W_i\in X_W}} & \bigoplus_{X_W}W_i \ar[r]^-g & E \ar[r]^-h &  V[1]
}\end{displaymath}
is a distinguished triangle in $\DbAR$ so is
\begin{displaymath}\xymatrix@C=8ex{
V \ar[r]^-f & W\ar[r]^-{\left[\begin{array}{l} g \\ \pi \end{array}\right]} & E\oplus \left(\bigoplus_{Y_W} W_i\right) \ar[r]^-{\left[\begin{array}{ll} h & 0 \end{array}\right]}& V[1].
}\tag{1}\end{displaymath}
Dually, suppose instead $W$ is indecomposable, each $f_i: V_i\to W$ is nonzero, and we've chosen $_VX$ and $_VY$ from Construction \ref{con:nice summing}.
Denoting by $\iota$ the inclusion $\bigoplus_{_VY} V_i\hookrightarrow V$, if
\begin{displaymath}\xymatrix@C=8ex{
\bigoplus_{_VX} V_i  \ar[r]^-{\left[ f_i \right]_{V_i\in {_VX}} } & W \ar[r]^-g & E\ar[r]^-h  & (\bigoplus_{_VX} V_i)[1]
}\end{displaymath}
is a distinguished triangle in $\DbAR$ so is
\begin{displaymath}\xymatrix@C=8ex{
V \ar[r]^-f & W\ar[r]^-{\left[\begin{array}{l} 0 \\ g\end{array}\right]} & \left(\left(\bigoplus_{_VY} V_i\right)[1]\right) \oplus E\ar[r]^-{\left[\begin{array}{ll} \iota & h\end{array}\right]} & V[1].
}\tag{2}\end{displaymath}
\end{lemma}
\begin{proof}
We'll prove (1) since the proof of (2) is similar.
Since $\Hom(A,B)\cong k$ or ${=}0$ for all indecomposables $A$ and $B$ in $\DbAR$, if, for a third indecomposable $C$ as well,
\begin{displaymath} \Hom(A,B)\cong \Hom(A,C)\cong \Hom(B,C)\cong k \end{displaymath}
then given any pair of morphisms $f:A\to B$ and $g:A\to C$ there exists $h:B\to C$ such that $g=hf$.
We may then apply Corollary \ref{cor:easy triangle splice} and obtain (1).
\end{proof}

\subsection{Triangles and the Geometry of the AR-space of $\DbAR$}
In this subsection we show how the geometry of the AR-space of $\DbAR$ is closely tied to the distinguished triangles in $\DbAR$.
We will use these connections in Section \ref{sec:derived equivalence}.
\begin{definition}\label{def:iota}
For each object $V$ in $\DbAR$, $V\cong \bigoplus_{i=1}^\ell M_{|a_i,b_i|}[n_i]$ for intervals $|a_i,b_i|$ and $n_i\in\Z$.
Reindex the $|a_i,b_i|$'s such that the following hold.
\begin{itemize}
\item if $n_i < n_j$ then $i<j$,
\item if $n_i=n_j$ and $a_i<a_j$ then $i< j$,
\item if $n_i=n_j$, $a_i=a_j$, and $b_i<b_j$ then $i<j$,
\item if $n_i=n_j$, $a_i=a_j$, $b_i=b_j$, $a_i\in|a_i,b_i|$, and $a_j\notin |a_j,b_j|$ then $i<j$, and
\item if $n_i=n_j$, $a_i=a_j$, $b_i=b_j$, $a_i,a_j\in |a_i,b_i|$ or $a_i,a_j\notin|a_i,b_i|$, $b_i\notin|a_i,b_i|$, and $b_j\in |a_j,b_j|$ then $i<j$.
\end{itemize}
This ordering determines a unique object $\iota V$, \textbf{not just up to isomorphism}, such that
\begin{displaymath} V\cong \underbrace{(((\cdots(M_{|a_1,b_1|}[n_1]\oplus M_{|a_2,b_2|}[n_2]) \oplus \cdots )\oplus M_{|a_{\ell-1},b_{\ell-1}|}[n_{\ell-1}]) \oplus M_{|a_\ell,b_\ell|}[n_\ell])}_{\iota V}.\end{displaymath}
Fix an isomorphism $\iota_V:V\to \iota V$  (where $\iota V=\bigoplus_{i=1}^\ell M_{|a_i,b_i|}[n_i]$ as we've described).
Note that in some cases $\iota_V$ will be the identity.

Let $f:V\to W$ be a morphism in $\DbAR$.
Then there exists a unique morphism $\iota(f):\iota(V)\to \iota(W)$ that makes the following diagram commute:
\begin{displaymath}\xymatrix{V \ar[d]_-{\iota_V} \ar[r]^-f & W \ar[d]^-{\iota_W} \\ \iota(V) \ar@{-->}[r]_-{\exists ! \iota(f)} & \iota(W).} \end{displaymath}.
\end{definition}

\begin{remark}\label{rmk:iota}
It is straightforward to check that $\iota:\DbAR\to \DbAR$ is a triangulated equivalence of categories as the image of $\iota$ is a skeleton of $\DbAR$.
\end{remark}

\begin{definition}\label{def:translation}
Let $V\cong \bigoplus V_i$ be an object in $\DbAR$ where each $V_i$ is indecomposable and let $\{(x_i,y_i)\}=\{\MM^b V_i\}$.
Let $r\in \R$ and set $z_i=x_i-r$ for each $V_i$.
If each $(z_i,y_i)$ is in the image of $\MM^b$ then there exists $M_{|a_i,b_i|}[n_i]$ for each $V_i$ such that $\MM^b M_{|a_i,b_i|}[n_i]=(z_i,y_i)$ and the position of $M_{|a_i,b_i|}[n_i]$ is the same as the position of $V_i$.
Then we write $T_r V$ to mean $\bigoplus M_{|a_i,b_i|}[n_i]$, indexed and parenthesized in the same way as in Definition \ref{def:iota}.
(If $r=0$ then $T_r V=\iota (V)$.)

Let $f:M_{|a,b|}[m]\to M_{|c,d|}[n]$ be a morphism in $\DbAR$ such that $T_r M_{|a,b|}[m]$ and $T_r M_{|c,d|}[n]$ are defined.
By \cite[Proposition 5.2.8 and Lemma 5.2.9]{Rock},
\begin{displaymath}\Hom_{\DbAR}(M_{|a,b|}[m], M_{|c,d|}[n])\cong \Hom_{\DbAR}(T_r M_{|a,b|}[m], T_r M_{|c,d|}[n]).\end{displaymath}
If $f$ is nonzero then both Hom sets are $k$.
Then $f$ is a scalar in $\Hom_{\DbAR}(M_{|a,b|}[m],M_{|c,d|}[n])$.
We define $T_r f$ to be the same scalar in $\Hom_{\DbAR}(T_r M_{|a,b|}[m], T_r M_{|c,d|}[n])$.

For arbitrary map $f:V\to W$ in $\DbAR$, $\iota(f)$ is a direct sum of morphisms $\hat{f}_{i,j}:M_{|a_i,b_i|}[m_i]\to M_{|c_j,d_j|}[n_j]$.
Let $f:V\to W$ be a morphism in $\DbAR$ and $r\in\R$ such that both $T_r V$ and $T_r W$ are defined.
Then we define $T_r f$ to be $\bigoplus_{\hat{f}_{i,j}} T_r \hat{f}_{i,j}$.
\end{definition}

\begin{remark}\label{rmk:translation}
Note that aside from our choice of $r$, $T_r$ does not depend on any additional choices beyond those for $\iota$ in Definition \ref{def:iota}.
\end{remark}

\begin{definition}\label{def:by geometry and maps}
Let $V\cong \bigoplus V_i$ and $W\cong \bigoplus W_j$ be objects in $\DbAR$ where each $V_i$ and $W_j$ are indecomposable.
Let $f:V\to W$ be a morphism.
Consider the distinguished triangle $V\stackrel{f}{\to} W\stackrel{g}{\to} U\stackrel{h}{\to}$.
We say $U$ is \ul{determined by geometry and $f$} if for any $r\in\R$ such that $T_r V$ and $T_r W$ are defined, there exists a distinguished triangle $T_r V\stackrel{T_r f}{\to} T_r W\to T_r U\to $ in $\DbAR$.
\end{definition}

A visual example of the proof technique to the following lemma is exhibited in Example \ref{xmp:constructing the cone}, stated afterwards.

\begin{lemma}\label{lem:geometric cones 1}
Let $V$ be indecomposable and $W$ an object in $\DbAR$.
Suppose $f:V\to W$ is a nonzero morphism and consider the distinguished triangle $V\stackrel{f}{\to} W\to U\to$.
Then $U$ is determined by geometry and $f$.
The conclusion holds if instead $W$ is indecomposable and $V$ is some object.

Furthermore:
\begin{enumerate}
\item If $V$ and $W$ together have $n$ indecomposable summands then $U$ has at most $n$ indecomposable summands.
\item If $V$ is indecomposable, for each indecomposable summand $U_j$ of $U$ there is an indecomposable summand $W_i$ of $W$ such that the line segment with endpoints $W_i$ and $U_j$ has slope $\pm(1,1)$.
\item If $W$ is indecomposable, for each indecomposable summand $U_i$ of $U$ there is an indecomposable summand $V_j$ of $V$ such that the line segment with endpoints $U_i[-1]$ and $V_j$ has slope $\pm(1,1)$.
\end{enumerate}
\end{lemma}
\begin{proof}
\textbf{Setup and Base Case.}
We prove the statement where $V$ is indecomposable as the proof when $W$ is indecomposable is similar.
In particular we'll prove (1) and (2) in the enumerated list (as (3) is part of the case when $W$ is the indecomposable).
The proof will be by induction on the number of summands of $W$.
Our base case is 1; i.e.\ $W$ is also indecomposable.
The base case follows from Theorem \ref{thm:triangles are rectangles}.

\textbf{Induction Setup and Trivial Case.}
For induction assume statements (1) and (2) hold when $V$ is indecomposable and $W$ has $n$ or fewer indecomposable summands.
Suppose $W\cong \bigoplus_{i=1}^n W_i$ where each $W_i$ is indecomposable.
For each summand $W_i$, let $f_i:V\to W_i$ be such that $f=[f_1\, \cdots\, f_n]^t$.

Let $W_{n+1}$ be an additional indecomposable and $f_{n+1}:V\to W_{n+1}$ a morphism.
If $f_{n+1}=0$ we are done as we obtain the distinguished triangle
\begin{displaymath}\xymatrix@R=2ex@C=10ex{ V\ar[r]^-{\left[\begin{array}{l} f\\0\end{array}\right]} & W \oplus W_{n+1} \ar[r]^-{\left[\begin{array}{ll} g & 0 \\ 0 & 1 \end{array}\right]} & U\oplus W_{n+1} \ar[r]^-{\left[\begin{array}{ll} h & 0 \end{array}\right]} & V[1].}\end{displaymath}

$\textbf{Hom}\mathbf{\boldsymbol{(}W\boldsymbol{,}W_{n\boldsymbol{+}1}\boldsymbol{)}\boldsymbol{\neq} 0} \textbf{ or }
\textbf{Hom}\mathbf{\boldsymbol{(}W_{n\boldsymbol{+}1}\boldsymbol{,}W\boldsymbol{)}\boldsymbol{\neq} 0}$\textbf{.}
Assume $f_{n+1}\neq 0$.
If $\Hom(W,W_{n+1})$ is not 0 such that $f_{n+1}$ factors through $f$ then again the lemma follows using Lemma \ref{lem:easy splice first}.
If $\Hom(W_{n+1},W)$ $\neq 0$ then there is at least one summand $W_j$ of $W$ such that $\Hom(W_{n+1},W_i) \neq 0$.
In this case, if $f_i=0$ then we can reverse the roles of $W_{n+1}$ and $W_i$ and use the induction hypothesis where $f_{n+1}=0$.
If $f_i\neq 0$ then our induction hypothesis holds for $V\to (\bigoplus_1^{j-1} W_i)\oplus (\bigoplus_{j+1}^{n+1}W_i)$ and we apply Lemma \ref{lem:easy splice first} again.

If $\Hom(W,W_{n+1})\neq 0$ but $f_{n+1}$ does not factor through $W$ via $f$ then for each summand $W_i$ of $W$ such that $\Hom(W_i,W_{n+1})\neq 0$ we know $f_i=0$.
Then we can reverse the roles of one such $W_i$ and $W_{n+1}$ and use the induction hypothesis where $f_{n+1}=0$.

$\textbf{Hom}\mathbf{\boldsymbol{(}W\boldsymbol{,}W_{n\boldsymbol{+}1}\boldsymbol{)}\boldsymbol{=}0\boldsymbol{=}
\textbf{Hom}\boldsymbol{(}W_{n\boldsymbol{+}1}\boldsymbol{,}W\boldsymbol{)}}$\textbf{.}
Suppose $f_{n+1}\neq 0$ and $\Hom(W,W_{n+1})$ and $\Hom(W_{n+1},W)$ are 0.
Choose an $X_W$ (Construction \ref{con:nice summing}).
If $X_W$ has fewer than $n$ elements then we are done since $X_W\cup \{W_{n+1}\}$ has $n$ or fewer elements and we may then apply the induction hypothesis.
So suppose $X_W$ contains each summand of $W$.

Note $\Hom(W,W_{n+1})=0$ but $f_i\neq 0$ for all $1\leq i\leq n+1$.
Thus the $W_i$, for all $1\leq i \leq n+1$, are  totally ordered by $y$-coordinate and position (2 is greater than 1 and 4 which are greater than 3).
Reindex the $W_i$s such that $W_i < W_{i+1}$ in the total order and replace $W$ with $\bigoplus_{i=1}^n W_i$ in the new index.
We have the following distinguished triangles:
\begin{displaymath}\xymatrix@R=2ex@C=8ex{
U'[-1] \ar[r]^-{h'[-1]} & V\ar[r]^-{f_{n+1}} & W_{n+1} \ar[r]^-{g'} & U' \\
V \ar[r]^-f & W\ar[r]^-g & U \ar[r]^-h & V[1]\\
U'[-1] \ar[r]^-{f\circ h'[-1]} & W \ar[r]^-p & E\ar[r]^-q & U'.
}\end{displaymath}
By our base case $U'$ has at most two indecomposable summands and by the rest of Theorem \ref{thm:triangles are rectangles} the slopes of the line segments from $W_{n+1}$ to each summand is $\pm(1,1)$.
By \cite[Lemma 5.2.9]{Rock} $W_n$ can map to at most one of the indecomposable summands of $U'$
and that is the only possible summand that can map to $W_n[1]$.

Thus $W$ can map to at most one indecomposable summand of $U'$.
Then by induction $E$ is determined by geometry and $f\circ h'[-1]$.
By the formulation of the octahedral axiom we've stated, we obtain the distinguished triangles
\begin{displaymath}\xymatrix@C=10ex{
W_{n+1}\ar[r]^-r & E \ar[r]^-s & U \ar[r]^-t & W_{n+1}[1] \\
V\ar[r]_-{\left[\begin{array}{l} f \\ f_{n+1} \end{array}\right]} & W\oplus W_{n+1} \ar[r]_-{\left[\begin{array}{ll} p & -r\end{array}\right]} & E \ar[r]_-{h\circ s =h' \circ q} & V[1].
}\end{displaymath}
Thus $E$ is determined by geometry and $[ f\, f_{n+1}]^t$.
Furthermore, the cone has the desired number of indecomposable summands.
\end{proof}

\begin{example}\label{xmp:constructing the cone}
The practical technique for doing this may be somewhat opaque to the reader at first.
As an example, one might have $f:V\to\bigoplus_{i=1}^4 W_i$.
According to the proposition the cone $E$ should be the direct sum of no more than 5 indecomposables, the slopes between select indecomposables should be as described, and $E$ should be determined by geometry and $f$.
We start with $V\to W_i$ for each $i$, then using the technique in the proof, splice together the triangles one by one.
At each step $i>1$, the new cone $E^i$ is given by using most of the old cone $E^{i-1}$ with a slight change to account for the new indecomposable.
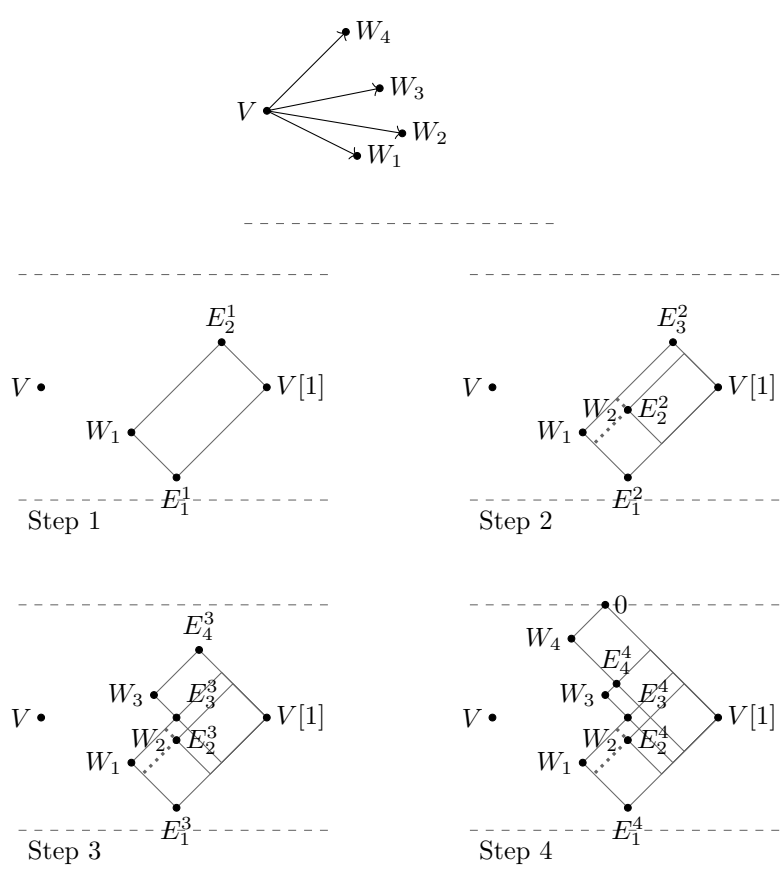
\begin{displaymath}\begin{tikzpicture}[scale=3]
\draw[dashed, white!40!black] (0,1) -- (1.4,1);
\draw[dashed, white!40!black] (0,0) -- (1.4,0);
\draw(0.1,0.5) node[anchor=east] {{$V$}};
\draw[fill=black, black] (0.1,0.5) circle[radius=0.15mm];
\draw (0.5,0.3) node [anchor=west] {{$W_1$}};
\draw[fill=black, black] (0.5,0.3) circle[radius=0.15mm];
\draw (0.6,0.6) node [anchor=west] {{$W_3$}};
\draw[fill=black, black] (0.6,0.6) circle[radius=0.15mm];
\draw (0.7,0.4) node[anchor=west] {{$W_2$}};
\draw[fill=black, black] (0.7,0.4) circle[radius=0.15mm];
\draw (0.45,0.85) node[anchor=west] {{$W_4$}};
\draw[fill=black, black] (0.45,0.85) circle[radius=0.15mm];
\draw[->, black] (0.1,0.5) -- (0.5,0.3);
\draw[->, black] (0.1,0.5) -- (0.6,0.6);
\draw[->, black] (0.1,0.5) -- (0.7,0.4);
\draw[->, black] (0.1,0.5) -- (0.45,0.85);
\end{tikzpicture}\end{displaymath}

\begin{displaymath}\begin{tikzpicture}[scale=3]
\draw[dashed, white!40!black] (0,1) -- (1.4,1);
\draw[dashed, white!40!black] (0,0) -- (1.4,0);
\draw[white!40!black] (0.5,0.3) -- (0.9,0.7) -- (1.1,0.5) -- (0.7,0.1) -- (0.5,0.3);
\draw (0.1,0.5) node[anchor=east] {{$V$}};
\draw[fill=black, black] (0.1,0.5) circle[radius=0.15mm];
\draw (0.5,0.3) node [anchor=east] {{$W_1$}};
\draw[fill=black, black] (0.5,0.3) circle[radius=0.15mm];
\draw (1.1,0.5) node [anchor=west] {{$V[1]$}};
\draw[fill=black, black] (1.1,0.5) circle[radius=0.15mm];
\draw (0.9,0.7) node [anchor=south] {{$E^1_2$}};
\draw[fill=black, black] (0.9,0.7) circle[radius=0.15mm];
\draw (0.7,0.1) node [anchor=north] {{$E^1_1$}};
\draw[fill=black, black] (0.7,0.1) circle[radius=0.15mm];
\draw (0,-0.1) node[anchor=west] {{Step 1}};

\draw[dashed, white!40!black] (2,1) -- (3.4,1);
\draw[dashed, white!40!black] (2,0) -- (3.4,0);
\draw[white!40!black] (2.5,0.3) -- (2.9,0.7) -- (3.1,0.5) -- (2.7,0.1) -- (2.5,0.3);
\draw[white!40!black] (2.7,0.4) -- (2.95,0.65) -- (3.1,0.5) -- (2.85,0.25) -- (2.7,0.4);
\draw[white!40!black, very thick, dotted] (2.65,0.45) -- (2.7,0.4) -- (2.55,0.25);
\draw (2.1,0.5) node[anchor=east] {{$V$}};
\draw[fill=black, black] (2.1,0.5) circle[radius=0.15mm];
\draw (2.5,0.3) node [anchor=east] {{$W_1$}};
\draw[fill=black, black] (2.5,0.3) circle[radius=0.15mm];
\draw (2.7,0.4) node[anchor=east] {{$W_2$}};
\draw[fill=black, black] (2.7,0.4) circle[radius=0.15mm];
\draw (3.1,0.5) node [anchor=west] {{$V[1]$}};
\draw[fill=black, black] (3.1,0.5) circle[radius=0.15mm];
\draw (2.9,0.7) node [anchor=south] {{$E^2_3$}};
\draw[fill=black, black] (2.9,0.7) circle[radius=0.15mm];
\draw (2.7,0.4) node[anchor=west] {{$E^2_2$}};
\draw (2.7,0.1) node [anchor=north] {{$E^2_1$}};
\draw[fill=black, black] (2.7,0.1) circle[radius=0.15mm];
\draw (2,-0.1) node[anchor=west] {{Step 2}};
\end{tikzpicture}\end{displaymath}

\begin{displaymath}\begin{tikzpicture}[scale=3]
\draw[dashed, white!40!black] (0,1) -- (1.4,1);
\draw[dashed, white!40!black] (0,0) -- (1.4,0);
\draw[white!40!black] (0.5,0.3) -- (0.9,0.7) -- (1.1,0.5) -- (0.7,0.1) -- (0.5,0.3);
\draw[white!40!black] (0.6,0.6) -- (0.8,0.8) -- (1.1,0.5) -- (0.9,0.3) -- (0.6,0.6);
\draw[white!40!black] (0.7,0.4) -- (0.95,0.65) -- (1.1,0.5) -- (0.85,0.25) -- (0.7,0.4);
\draw[white!40!black, very thick, dotted] (0.65,0.45) -- (0.7,0.4) -- (0.55,0.25);
\draw (0.1,0.5) node[anchor=east] {{$V$}};
\draw[fill=black, black] (0.1,0.5) circle[radius=0.15mm];
\draw (0.5,0.3) node [anchor=east] {{$W_1$}};
\draw[fill=black, black] (0.5,0.3) circle[radius=0.15mm];
\draw (0.6,0.6) node [anchor=east] {{$W_3$}};
\draw[fill=black,black] (0.6,0.6) circle[radius=0.15mm];
\draw (0.7,0.4) node[anchor=east] {{$W_2$}};
\draw[fill=black, black] (0.7,0.4) circle[radius=0.15mm];
\draw (1.1,0.5) node [anchor=west] {{$V[1]$}};
\draw[fill=black, black] (1.1,0.5) circle[radius=0.15mm];
\draw (0.8,0.8) node[anchor=south] {{$E^3_4$}};
\draw[fill=black, black] (0.8,0.8) circle[radius=0.15mm];
\draw (0.7,0.5) node[anchor=south west] {{$E^3_3$}};
\draw[fill=black, black] (0.7,0.5) circle[radius=0.15mm];
\draw (0.7,0.4) node[anchor=west] {{$E^3_2$}};
\draw (0.7,0.1) node [anchor=north] {{$E^3_1$}};
\draw[fill=black, black] (0.7,0.1) circle[radius=0.15mm];
\draw (0,-0.1) node[anchor=west] {{Step 3}};

\draw[dashed, white!40!black] (2,1) -- (3.4,1);
\draw[dashed, white!40!black] (2,0) -- (3.4,0);
\draw[white!40!black] (2.5,0.3) -- (2.9,0.7) -- (3.1,0.5) -- (2.7,0.1) -- (2.5,0.3);
\draw[white!40!black] (2.6,0.6) -- (2.8,0.8) -- (3.1,0.5) -- (2.9,0.3) -- (2.6,0.6);
\draw[white!40!black] (2.7,0.4) -- (2.95,0.65) -- (3.1,0.5) -- (2.85,0.25) -- (2.7,0.4);
\draw[white!40!black, very thick, dotted] (2.65,0.45) -- (2.7,0.4) -- (2.55,0.25);
\draw[white!40!black] (2.45,0.85) -- (2.6,1) -- (3.1,0.5) -- (2.95,0.35) -- (2.45,0.85);
\draw (2.1,0.5) node[anchor=east] {{$V$}};
\draw[fill=black, black] (2.1,0.5) circle[radius=0.15mm];
\draw (2.5,0.3) node [anchor=east] {{$W_1$}};
\draw[fill=black, black] (2.5,0.3) circle[radius=0.15mm];
\draw (2.6,0.6) node [anchor=east] {{$W_3$}};
\draw[fill=black, black] (2.6,0.6) circle[radius=0.15mm];
\draw (2.7,0.4) node[anchor=east] {{$W_2$}};
\draw[fill=black, black] (2.7,0.4) circle[radius=0.15mm];
\draw(2.45,0.85) node[anchor=east] {{$W_4$}};
\draw[fill=black, black] (2.45,0.85) circle[radius=0.15mm];
\draw (3.1,0.5) node [anchor=west] {{$V[1]$}};
\draw[fill=black, black] (3.1,0.5) circle[radius=0.15mm];
\draw (2.6,1) node[anchor=west] {{$0$}};
\draw[fill=white, black] (2.6,1) circle[radius=0.15mm];
\draw (2.65,0.65) node[anchor=south] {{$E^4_4$}};
\draw[fill=black, black] (2.65,0.65) circle[radius=0.15mm];
\draw (2.7,0.5) node[anchor=south west] {{$E^4_3$}};
\draw[fill=black, black] (2.7,0.5) circle[radius=0.15mm];
\draw (2.7,0.4) node[anchor=west] {{$E^4_2$}};
\draw (2.7,0.1) node [anchor=north] {{$E^4_1$}};
\draw[fill=black, black] (2.7,0.1) circle[radius=0.15mm];
\draw (2,-0.1) node[anchor=west] {{Step 4}};
\end{tikzpicture}\end{displaymath}
\end{example}

\begin{example}
Another notable example is what happens when we check the cone $E_1\oplus E_2\to V[1]$ from a map of indecomposable $V\to W$.
Since $E_1$ and $E_2$ both create line segments with other endpoint $V$ and slope $\pm(1,1)$, any triangle $E_i\to V\to F_i\to$ has $F_i$ indecomposable.
Furthermore, each $F_i$ will also share an endpoint with $V$.
Of course, algebraically we \emph{must} have that $W[1]$ is the cone of $E_1\oplus E_2\to V[1]$ but this is also correct \emph{geometrically}.
We end up with the following picture:
\begin{displaymath}\begin{tikzpicture}[scale=1.5]
\draw[dashed, white!40!black] (0,2) -- (4,2);
\draw[dashed, white!40!black] (0,0) -- (4,0);
\draw[white!40!black] (0,1.5) -- (1.5,0) -- (3.5,2) -- (4,1.5) -- (2.5,0) -- (0.5,2) -- (0,1.5);
\draw[white!40!black] (1,1) -- (2,2) -- (3,1) -- (2,0) -- (1,1);
\draw[white!40!black] (3.75,1.75) -- (3,1) -- (3.25,0.75);

\draw[fill=black] (0,1.5) circle[radius=.3mm];
\draw[fill=black] (4,1.5) circle[radius=.3mm];
\draw[fill=black] (2,0.5) circle[radius=.3mm];
\draw (0,1.5) node[anchor=east] {$V$};
\draw (2,0.5) node[anchor=south] {$V[1]$};
\draw (4,1.5) node[anchor=west] {$V[2]$};

\draw[fill=black] (1,1) circle[radius=.3mm];
\draw[fill=black] (3,1) circle[radius=.3mm];
\draw (1,1) node[anchor=east] {$W$};
\draw (3,1) node[anchor=east] {$W[1]$};

\draw[fill=black] (1.75,0.25) circle [radius=.3mm];
\draw[fill=black] (3.75,1.75) circle [radius=.3mm];
\draw (1.75,0.25) node [anchor=east] {$E_1$};
\draw (3.75,1.75) node [anchor=west] {$E_1[1]$};

\draw [fill=black] (1.25,1.25) circle [radius=.3mm];
\draw [fill=black] (3.25,0.75) circle [radius=.3mm];
\draw (1.25,1.25) node [anchor=east] {$E_2$};
\draw (3.25,0.75) node [anchor=west] {$E_2[1]$};

\draw [fill=white] (3.5,2) circle [radius=.3mm];
\draw (3.5,2) node [anchor=south] {0};
\draw [fill=white] (2.5,0) circle [radius=.3mm];
\draw (2.5,0) node [anchor=north] {0};
\end{tikzpicture}\end{displaymath}
\end{example}

\begin{theorem}\label{thm:geometric cones}
Let $f:V\to W$ be a morphism in $\DbAR$ and consider the distinguished triangle
\begin{displaymath}\xymatrix@C=6ex{
V\ar[r]^-f & W\ar[r]^-g &  U\ar[r]^-h & V[1].
}\end{displaymath}
Then $U$ is determined by geometry and $f$.
\end{theorem}
\begin{proof}
Since $V$ and $W$ are each finite direct sums of indecomposables we'll prove the statement by induction on the number of indecomposable summands of $V$, using Lemma \ref{lem:geometric cones 1} as the base case.

Assume the lemma holds for all positive integers $\leq n$ and let $V_i$ for $1\leq i\leq n+1$ be indecomposables and let $V=\bigoplus_{i=1}^n V_i$.
Let $f:V\to W$ and $f_{n+1}:V_{n+1}\to W$ be a nontrivial morphisms.
We begin with the following distinguished triangles:
\begin{displaymath}\xymatrix@R=2ex@C=8ex{
V \ar[r]^-f & W \ar[r]^-g & U\ar[r]^-h & V[1] \\
W \ar[r]^-{g_{n+1}} & U' \ar[r]^-{h_{n+1}} &  V_{n+1}[1] \ar[r]^-{f_{n+1}[1]} & W[1] \\
V \ar[r]^-{g_{n+1}\circ f} & U' \ar[r]^-p & E \ar[r]^-q & V[1].
}\end{displaymath}
Applying our octahedral axiom we obtain the following distinguished triangles:
\begin{displaymath}\xymatrix@R=4ex@C=18ex{
U \ar[r]^-{p'} & E\ar[r]^-{q'} & V_{n+1}[1] \ar[r]^-{(g\circ f_{n+1})[1]} &  U[1] \\
E \ar[r]_-{\left[\begin{array}{l} q \\ q' \end{array}\right]} & V[1]\oplus V_{n+1}[1] \ar[r]_-{\left[\begin{array}{ll} f[1] & -f_{n+1}[1] \end{array}\right]} & W[1] \ar[r]_-{(p'\circ g)[1] = (p\circ g_{n+1})[1]} & E[1].
}\end{displaymath}
By the induction hypothesis, $E$ is determined by geometry and $g_{n+1}\circ f$.
Therefore, $E$ is determined by geometry and $[f\, -f_{n+1}]$.
\end{proof}

\subsection{Derived Equivalence}\label{sec:derived equivalence}

In this subsection we prove the first main result of this paper: the bounded derived categories of two continuous type $A$ quivers are triangulated-equivalent if and only if they belong to the same of the following three classes:
\begin{itemize}
\item finitely many sinks and sources,
\item the existence of a minimal or maximal sink/source, but not both, or
\item no minimal or maximal sink/source.
\end{itemize}
Note that these are all disjoint.

\begin{lemma}\label{lem:slope lemma}
Let $A_\R$ and $A'_\R$ be possibly different continuous quivers of type $A$ and let $\mathcal G:\DbAR\to \mathcal D^b(A'_\R)$ be an equivalence of categories.
Let $V$ and $W$ be indecomposables in $\DbAR$ and suppose the slope of the line segment from $V$ to $W$ in the AR-space of $\DbAR$ is $\pm(1,1)$.
Then the slope of the line segment from $\mathcal G V$ to $\mathcal G W$ in the AR-space of $\mathcal D^b (A'_\R)$ is $\pm(1,1)$.
\end{lemma}
\begin{proof}
For contradiction, suppose the conclusion is false.
We know by \cite[Lemma 5.2.9]{Rock} that if the slope of the line segment from $\mathcal G V$ to $\mathcal G W$ in the AR-space of $\mathcal D^b (A'_\R)$ is greater than $(1,1)$ or less than $-(1,1)$ then $\Hom_{\mathcal D^b(A'_\R)}(\mathcal G V,\mathcal G W)=0$.
This contradicts that $\mathcal G$ is an equivalence of categories.

Now suppose the slope is less than $(1,1)$ and greater than $-(1,1)$.
We know $\mathcal G W$ must be close enough to $\mathcal G V$ by \cite[Lemma 5.2.9]{Rock}.
Then we may use Theorem \ref{thm:triangles are rectangles} and construct a distinguished triangle $\mathcal G V \to U \to \mathcal G W\to$.
We may also choose a $U'$ in $\mathcal D^b(A'_\R)$ such that $\Hom(U,U')=\Hom(U',U)=0$ on the line segement from $\mathcal G V$ to $\mathcal G W$.
Then any morphism $f:\mathcal G V\to \mathcal G W$ factors through either $U$ or $U'$ but not both.
This contradicts the fact that any morphism from $V$ to $W$ in $\DbAR$ that factors through some $X$ and $X'$ factors through both and $\Hom(X,X')\cong k$ or $\Hom(X',X)\cong k$.
\end{proof}

\begin{definition}\label{def:representative}
For each isomorphism class of indecomposable objects in $\DbAR$ we call the representation $V$ such that $\iota V=V$ the \ul{representative}.
\end{definition}

\begin{definition}\label{def:the functor}
Let $A_\R$ and $A'_\R$ be possibly different continuous type $A$ quivers.
Suppose also that both $A_\R$ and $A'_\R$ have either (a) finitely many sinks and sources or (b) infinitely many sinks and sources but a minimal or maximal sink/source.
We construct a functor $\mathcal{F}$ from $\DbAR$ to $\mathcal{D}^b(A'_\R)$.

If (b) is true then by the proof of \cite[Proposition 2.4.4]{Rock} there is a $\lambda$ function $\lambda_{*\infty}$ (Definition \ref{def:lambda function}) whose graph in $\R^2$ is disjoint from the image of $\MM^b:\text{Ind}(\DbAR)\to \R^2$.
All other points in $\R\times[-\frac{\pi}{2},\frac{\pi}{2}]$ are in the image of $\MM^b$.
For $A'_\R$ there is a similar $\lambda'_{*\infty}$ disjoint from the image of $(\MM')^b$.

If (b) is true let $T$ be the translation of $\R\times[-\frac{\pi}{2},\frac{\pi}{2}]$ such that the graph of $\lambda_{*\infty}$ is taken to the graph of $\lambda'_{*\infty}$.
If (a) is true instead let $T$ be the identity.

For each indecomposable $V$ in $\DbAR$ with position $i$ define $\mathcal F V$ be the representative indecomposable in $\mathcal{D}^b(A'_\R)$ such that $(\MM')^b \mathcal F V =(T \circ \MM^b) V$ and the position of $\mathcal F$ is $i$.
Since $\DbAR$ is Krull-Schmidt (\cite[Proposition 5.1.2]{Rock}) this induces a mapping on all objects given by
\begin{displaymath}
\mathcal{F} \left( \bigoplus V_i\right) := \bigoplus \mathcal{F} V_i.
\end{displaymath}
with the indexing and parenthesizing from Definition \ref{def:iota}.
The Hom space between any two indecomposables in both $\DbAR$ and $\mathcal{D}^b(A'_\R)$ is $k$ or 0.
For any pair of representatives $V'$ and $W'$ in $\mathcal{D}^b(A'_\R)$ we have can take each nontrivial morphism $f:V'\to W'$ to be a nonzero value of $k$.
Define $\mathcal{F} f := f$ for all morphisms of representative indecomposables $V\stackrel{f}{\to} W$ in $\DbAR$.
For indecomposables $V$ and $W$ in $\DbAR$ with representatives $\iota V$ and $\iota W$, any nontrivial morphism $f:V\to W$ factors as in Definition \ref{def:iota}:
\begin{displaymath}\xymatrix{
V \ar[r]^-{\iota_V} & \iota V\ar[r]^-{\iota f} & \iota W \ar[r]^-{\iota_W^{-1}} & W.
}\end{displaymath}
We set $\mathcal{F} f := \mathcal{F}\iota f$.
For any two objects $\bigoplus^m V_i$ and $\bigoplus^n W_j$ in $\DbAR$ we can extend bilinearly with a basis is given by the ordered summands of $\iota (\bigoplus ^m V_i)$ and $\iota (\bigoplus^n W_j)$.\end{definition}

\begin{lemma}\label{lem:F is nice}
Given $A_\R$ and $A'_\R$ in Definition \ref{def:the functor}, $\mathcal F$ is a well-defined additive functor.
Furthermore, $\mathcal F$ is an equivalence of categories and $\mathcal F(V[1])\cong (\mathcal F V)[1]$.
\end{lemma}
\begin{proof}
By using the translation $T$ in Definition \ref{def:the functor} for any indecomposable $V$ in $\DbAR$, $\mathcal F V$ is well-defined and indeed $\mathcal F$ induces a bijection on indecomposable objects.
Furthermore, since $\DbAR$ and $\mathcal D^b(A'_\R)$ are both Krull-Schmidt, we have a bijection on all objects.

Recall the Hom support of an indecomposable $V$ in $\DbAR$ is completely determined by the coordinates of $\MM^bV$ and the position of $V$.
Thus for any pair of indecomposables $V$ and $W$ in $\DbAR$ we have an isomorphism of vector spaces:
\begin{displaymath} \Hom_{\DbAR}(V,W)\cong \Hom_{\mathcal{D}^b(A'_\R)}(\mathcal{F} V, \mathcal{F} W). \end{displaymath}
Extending bilinearly gives us
\begin{displaymath}
\Hom_{\DbAR}\left(\bigoplus^m V_j,\,  \bigoplus^n W_j\right)\cong \Hom_{\mathcal{D}^b(A'_\R)} \left(\bigoplus^m \mathcal{F} V_i,\,  \bigoplus^n \mathcal{F} W_j\right).
\end{displaymath}
It is clear composition is preserved and that the functor is additive.

Finally let $(x,y)=\MM^b V$ for an indecomposable $V$ in $\DbAR$.
We know $\MM^b V[1] = (x+\pi,-y)$ by definition.
Then if $(T\circ \MM^b)V =(\MM')^b \mathcal F V$ we know 
\begin{displaymath}
(T\circ {\MM}^b)(V[1]) =({\MM}')^b \mathcal F (V[1]) = ({\MM}')^b ((\mathcal F V)[1]). 
\end{displaymath}
\end{proof}

\begin{theorem}\label{thm:derived classification}
Let $A_\R$ and $A'_\R$ be possibly different orientations of continuous type $A$ quivers.
One of the following is true if and only if $\DbAR$ and $\mathcal{D}^b(A'_\R)$ are equivalent as triangulated categories:
\begin{enumerate}
\item both $A_\R$ and $A'_\R$ have finitely many sinks and sources,
\item the sinks and sources of $A_\R$ and $A'_\R$ are each bounded on exactly one side, or
\item the sinks and sources of $A_\R$ and $A'_\R$ are unbounded on both sides.
\end{enumerate}
\end{theorem}
\begin{proof}
If (3) holds then by Proposition \cite[Proposition 3.2.1]{IgusaRockTodorov} the representation categories are equivalent as abelian categories and so the derived categories will also be equivalent as triangulated categories.
If (1) or (2) holds then the categories are already equivalent as additive categories, by Lemma \ref{lem:F is nice}.
By the same Lemma, $\mathcal{F}$ commutes with the shift in each derived category.
Thus it remains to show that $\mathcal{F}$ takes cones in distinguished triangles to cones in distinguished triangles.

Let $U\stackrel{f}{\to} V\stackrel{g}{\to} W \stackrel{h}{\to} $ be distinguished in $\DbAR$.
By Theorem \ref{thm:geometric cones} we see $W$ is determined by geometry and $f$.
Then the same geometry and image of the maps under $\mathcal F$ determine a distinguished triangle $\mathcal{F} U\stackrel{\mathcal F f}{\to} \mathcal{F} V\stackrel{g'}{\to} W'\stackrel{h'}{\to} $ in $\mathcal{D}^b(A'_\R)$.
However, by Definition \ref{def:the functor}, we see that $W'\cong\mathcal{F}W$.
Thus we have
\begin{displaymath}\xymatrix@C=6ex{
\mathcal F V\ar[r]^-{\mathcal F f} & \mathcal F U\ar[r]^-{g'} & \mathcal W' \ar[r]^-{h'} \ar@{<->}[d]^-{\cong} & \mathcal F V[1]  \\
\mathcal F V \ar[r]_-{\mathcal F f} & \mathcal F U \ar[r]_-{\mathcal F g} &  \mathcal F W \ar[r]_-{\mathcal F h} & \mathcal F V[1].
}\end{displaymath}
Therefore up to changing signs on the morphisms between indecomposables that make up $h'$ the bottom row is also a distinguished triangle and so $\mathcal F$ is a triangulated equivalence.

Since every continuous quiver of type $A$ with finitely-many sinks and sources is derived equivalent we will show that classes (2) and (3) are disjoint from (1) using the straight descending orientation.
Afterwards we will show (2) and (3) are also disjoint.

Let $A_\R$ be a continuous quiver of type $A$ with infinitely-many sinks and sources.
Let $A'_\R$ be the straight descending continuous quiver of type $A$.
Let $\mathcal P$ be the set of indecomposable projectives of $\text{rep}_k(A'_\R)$ included as indecomposables in degree 0 in $\mathcal D^b(A'_\R)$.

For contradiction, suppose there is an equivalence of categories $\mathcal G:\mathcal D^b(A'_\R)\to \DbAR$.
By Lemma \ref{lem:slope lemma} we know that for any pair $P,P'\in\mathcal P$, the slope of any line segment from $\mathcal G P$ to $\mathcal G P'$ (switching roles if necessary) is $\pm (1,1)$.
Since $\mathcal G$ is an equivalence we know that if there exists $V$ in $\DbAR$ such that $\Hom_{\DbAR}(\mathcal G P,V)\cong \Hom_{\DbAR}(V,\mathcal G P')\cong k$ for $P,P'\in \mathcal P$ then $V \cong \mathcal G P''$ for some $P''\in \mathcal P$.

Note that $\mathcal P$ actually induces an almost-complete line segment in the AR-space of $\mathcal D^b (A'_\R)$.
Then there does not exist an indecomposable object $V$ in $\mathcal D^b(A'_\R)$ such that (i) $V$ is not isomorphic to an element of $\mathcal P$ and (ii) $\Hom_{\mathcal D^b(A'_\R)}(V,P)\cong k$ for all $P\in \mathcal P$.
We also know that if $\Hom_{\mathcal D^b(A'_\R)}(P_{+\infty},U)\cong \Hom_{\mathcal D^b(A'_\R)}(P_{+\infty}, U')\cong k$ then $\Hom_{\mathcal D^b(A'_\R)}(U,U')\cong k$ or $\Hom_{\mathcal D^b(A'_\R)}(U',U)\cong k$.
Since $\mathcal G$ is an equivalence of categories these properties must be preserved.

This forces the image $\mathcal G\mathcal P$ to induce an almost complete line segment whose existing endpoint $\mathcal G P_{+\infty}$ is on the top or bottom boundary of the AR-space of $\DbAR$.
We describe how we will arrive at a contradiction and include pictures after this paragraph.
Since $A_\R$ has infinitely many sinks and sources we must split the set $\mathcal G \mathcal P$ into (at least) two pieces since there is a missing graph of a $\lambda$ function from the image of $\MM^b$.
The set $\MM^b \mathcal G\mathcal P$ is homeomorphic to the disjoint union of (at least) two intervals.
The set $(\MM')^b \mathcal P$ is homeomorphic to one interval.
We can't reorder the elements in the almost-complete line segment and we can't allow a strict inclusion of $\mathcal P$ into the almost-complete line segment in the AR-space of $\DbAR$.
This finally leads us to the contradiction.
Below we depict two of the issues with mapping the projective line:

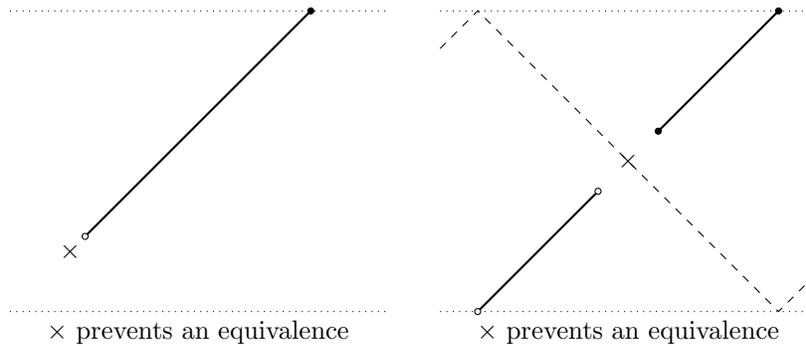
\begin{displaymath}\begin{tikzpicture}
\draw[dotted] (-4,2) -- (1,2);
\draw[dotted] (-4,-2) -- (1,-2);
\draw[thick] (-3,-1) -- (0,2);
\filldraw[fill=white] (-3,-1) circle [radius=.4mm];
\filldraw (0,2) circle [radius=.4mm];
\draw (-3.2,-1.2) node {$\times$};
\draw (-1.5,-2) node [anchor=north] {$\times$ prevents an equivalence};
\end{tikzpicture}
\qquad
\begin{tikzpicture}
\draw[dotted] (-4,2) -- (1,2);
\draw[dotted] (-4,-2) -- (1,-2);
\draw[dashed] (-4,1.5) -- (-3.5,2) -- (0.5,-2) -- (1,-1.5);
\draw (-1.5,0) node {$\times$};
\draw[thick] (-3.5,-2) -- (-1.9,-.4);
\draw[thick] (-1.1,.4) -- (0.5,2);
\filldraw[fill=white] (-3.5,-2) circle[radius=.4mm];
\filldraw[fill=white] (-1.9,-.4) circle[radius=.4mm];
\filldraw (-1.1,.4) circle[radius=.4mm];
\filldraw (0.5,2) circle[radius=.4mm];
\draw (-1.5,-2) node [anchor=north] {$\times$ prevents an equivalence};
\end{tikzpicture}
\end{displaymath}

Now we show that (2) and (3) are also disjoint.
Suppose $A_\R$ has half bounded sinks and sources and $A'_\R$ has unbounded sinks and sources on both sides.
The image in $\R^2$ under $\MM^b$ of an almost complete line segment in the AR-space of $\DbAR$ has at most two path components.
However, one may construct an almost complete line segment in the AR-space of $\mathcal D^b(A'_\R)$ that whose image in $\R^2$ under $(\MM')^b$ has three path components.
Knowing this we apply a similar argument as before and see that an equivalence of categories cannot exist between $\DbAR$ and $\mathcal D^b(A'_\R)$.
\end{proof}

\begin{remark}\label{rem:not all AR triangles}
It follows from Theorem \ref{thm:triangles are rectangles} that the only Auslander--Reiten triangles in $\DbAR$, for some continuous quiver of type $A$, are those of the form
\begin{displaymath}
V_1 \longrightarrow \begin{array}{c} V_2 \\ \oplus \\ V_3 \end{array} \longrightarrow V_4
\end{displaymath}
where each $V_i$ has position $i$ and $\MM^b V_i=\MM^b V_j \in \R\times (-\frac{\pi}{2},\frac{\pi}{2})$ for all $i,j$.
\end{remark}

\subsection{Na\"ive ``Fixes''}
For type $A_n$ quivers, the path algebra of every orientation is derived equivalent to the rest.
However, we see in Theorem \ref{thm:derived classification} that different orientations of continuous type $A$ quivers may yield different derived categories.

One na\"ive approach to ``fixing'' this is to first note that when $A_\R$ has infinitely many sinks and sources that some indecomposables are not finitely-generated.
One may think to include these anyway and then obtains the category of \emph{locally} finitely-generated representations.
These are representations such that the restriction of a module to any finite interval is finitely-generated.

This leads to a problem with the AR-space of the representation category and then the derived category.
The missing $\lambda$ functions in the image of $\MM$ (and $\MM^b$) do get filled but instead of 4 modules at every point the modules that are new modules have only 2 modules per point and thus no Auslander--Reiten sequence.
While this may not appear to be particularly problematic it still fails to unify all the continuous type $A$ quivers' derived categories.

The other possibility is to include $\pm \infty$ in a continuous quiver of type $A$.
However, if a continuous quiver of type $A$ has infinitely-many sinks and sources then the vertex at $\pm\infty$ becomes disjoint from the rest of the quiver.
There can be no path from $\pm \infty$ to a finite value if there are infinitely-many sinks and sources in the way.
So, the finitely-generated representation category just ends up with two additional simples and the locally finitely-generated representation category has the same fate.

In essence, the added complexity of either locally finitely-generated representations or adding $\pm\infty$ to the continuous quivers does not yield any new useful structure.
Futhermore, the trinary classification of the derived categories appears to be inherent to the continuous quivers themselves regardless of these algebraic tricks.

\section{New Continuous Cluster Category}\label{sec:new continuous cluster category}
\subsection{Definition and $g$-vectors}
In this section we will define a new model for the continuous cluster category based on the previous construction by Igusa and Todorov in \cite{IgusaTodorov1}.
In Remark \ref{rem:not all AR triangles} we noted what all the Auslander--Reiten triangles look like.
This also means we don't have left or right Auslander--Reiten triangles for every object.
In fact, those indecomposable objects $V$ such that $\MM^b V$ has $y$-coordinate $\pm \frac{\pi}{2}$ do not belong to \emph{any} Auslander--Reiten triangles.

In \cite{BMRRT} the cluster category associated to a quiver $Q$ is constructed using the orbit category of the bounded derived category of representations via the composition of the shift functor followed by the inverse Auslander--Reiten translation.
The result is a 2-Calabi-Yau orbit category.
We do not have enough Auslander--Reiten triangles to form the (inverse) Auslander--Reiten translation.
However, one may go looking for another possible functor instead; this won't work.

Without all Auslander--Reiten triangles we cannot have a Serre functor.
Without a Serre functor we cannot construct a 2-Calabi-Yau orbit category. 
(See work by Reiten and Van den Bergh \cite{ReitenVandenBergh}.)
Thus, as in \cite{IgusaTodorov1} we will use a functor which is almost-shift.
\begin{definition}\label{def:cluster category}
Let $A_\R$ be a continuous quiver of type $A$ and $\DbAR$ its bounded derived category.
The \ul{continuous $\mathbf E$-cluster category}, denoted $\CAR$, is an orbit category of the doubling of $\DbAR$.
We take the doubling of $\DbAR$ as in \cite{IgusaTodorov1}, which is equivalent to $\DbAR$, and call the functor with respect to which we take the orbit the \ul{almost-shift}.
We denote this orbit category by $\CAR$.
As in \cite{IgusaTodorov1}, $\CAR$ is also a triangulated category.
\end{definition}

\begin{proposition}
Let $V$ and $W$ be indecomposable objects in $\CAR$.
Then \\$\Hom_{\CAR}(V,W)=0$ or $\Hom_{\CAR}(V,W)\cong k$.
\end{proposition}
\begin{proof}
Suppose $\Hom_{\CAR}(V,W)\neq 0$.
We know \begin{displaymath}\Hom_{\CAR}(V,W)\cong \bigoplus_{n\in\Z}\Hom_{\DbAR}(V,W[n])\end{displaymath} and so there exists at least one $n$ such that $\Hom_{\DbAR}(V,W[n])\cong k$.
By \cite[Propositino 4.4.2 and Lemma 5.2.9]{Rock} if $V\not\cong W$ we know that $\Hom_{\DbAR}(V,W[m])=0$ for $m>n$ and $m<n$.
If $V\cong W$ we note that $\Ext^1(U,U)=0$ for all indecomposables $U$ in $\repAR$.
Thus $\Hom_{\CAR}(V,W)\cong \Hom_{\DbAR}(V,W[n])\cong k$.
\end{proof}

Recall that by taking the orbit category, the class of objects in $\CAR$ is the same as that in the doubling of $\DbAR$ even though the isomorphism classes have changed.
Then the following proposition is straight-forward.

\begin{proposition}\label{prop:CAR is K-S}
The orbit category $\CAR$ is Krull-Schmidt.
\end{proposition}

We follow J{\o}rgensen and Yakimov in \cite{JorgensenYakimov} with the following definition.
\begin{definition}\label{def:index}\label{def:g vector}
Denote by $\mathcal P$ the collection of indecomposable objects isomorphic to $P[n]$ in $\CAR$ where $P$ was a projective indecomposable in $\repAR$.
Let $V$ be an indecomposable in $\CAR$.
We define the \ul{g-vector} or \ul{index} of $V$ to be the element $[P_V]-[Q_V]$ in $K_0^{\text{split}}(Add \mathcal P)$ such that $Q_V\to P_V\to V\to$ is a distinguished triangle in $\CAR$ that comes from the projective resolution of $V$.
\end{definition}

\begin{definition}\label{def:euler form}
Let $[A]=\Sigma m_i[A_i]$ and $[B]=\Sigma n_j[B_j]$ be elements of $\KsplitCAR$ where the $A_i$s and $B_j$s are indecomposable.
The Euler bilinear form $\langle [A]\, ,\, [B]\rangle$ is given in the following way.
First, for a pair of indecomposables $A_i$ and $B_j$ in $\CAR$ we define
\begin{displaymath} \langle m_i [A_i]\, ,\, n_j [B_j] \rangle := (m_i\cdot n_j) (\dim\Hom_{\CAR}(A_i,B_j) ).\end{displaymath}
Thus the form is defined to be
\begin{displaymath} \langle [A]\, ,\, [B] \rangle := \sum_i \sum_j \langle [A_i]\, ,\, [B_j] \rangle. \end{displaymath}
Since $\CAR$ is Krull-Schmidt this is always a finite sum and thus well-defined.
\end{definition}

\begin{definition}\label{def:g vector E-compatible}
We say two g-vectors $[P_V]-[Q_V]$ and $[P_W]-[Q_W]$ are \ul{$\mathbf E$-compatible} if
\begin{displaymath}
\langle [P_V]-[Q_V]\, ,\, [P_W]-[Q_W] \rangle \geq 0 \qquad \text{and} \qquad \langle [P_W]-[Q_W]\, ,\, [P_V]-[Q_V] \rangle \geq 0.
\end{displaymath}
\end{definition}

We call this compatibility $\mathbf E$-compatibility to align better with Section \ref{sec:embeddings of cluster theories}, where we introduce the general definition of a cluster theory (Definition \ref{def:cluster theory}).

\begin{proposition}\label{prop:incompatible means extension}
Let $[P_V]-[Q_V]$ and $[P_W]-[Q_W]$ be two $g$-vectors of indecomposables $V$ and $W$ in $\CAR$, both in degree 0.
Consider $V$ and $W$ as images of the composite $\repAR\hookrightarrow \DbAR\to \CAR$.
Then $[P_V]-[Q_V]$ and $[P_W]-[Q_W]$ are not $\mathbf E$-compatible if and only if there is an extension $V\hookrightarrow U \twoheadrightarrow W$ or $W\hookrightarrow U\twoheadrightarrow V$ in $\repAR$.
\end{proposition}
\begin{proof}
Take $P_V$, $P_W$, $Q_V$, and $Q_W$ to be the projectives in $\repAR$ whose image is in the isomorphism classes indicated by the $g$-vectors.
If there is an extension $V\hookrightarrow U\twoheadrightarrow W$ in $\repAR$ then there is a nontrivial morphism $W\to V[1]$ in $\DbAR$.
By Theorem \ref{thm:little rep} there is, up to scaling and isomorphism, a unique extension.
This extension exists because there is a morphism $Q_W \to P_V$ that does not factor through $Q_V\oplus P_W$.

By the proof of \cite[Proposition 3.2.4]{IgusaRockTodorov} we see this means there must be some indecomposable summand of $Q_W$ that maps to at least one indecomposable summand of $P_V$ but does not factor through $Q_V$ or $P_W$.
By Theorem \ref{thm:projs}, $\repAR$ is hereditary and so $Q_W$ is a subrepresentation of $P_W$ and $Q_V$ is a subrepresentation of $P_V$.
Thus \begin{displaymath} \langle [P_V]-[Q_V]\, ,\, [P_W]-[Q_W] \rangle <0. \end{displaymath}

If we start with incompatible $g$-vectors then we reverse the argument and see that, up to symmetry, there is a morphism $Q_W\to P_V$ that does not factor through $Q_V\oplus P_W$.
Thus there is an extension $V\hookrightarrow U\twoheadrightarrow W$ in $\repAR$.
\end{proof}

\begin{proposition}\label{prop:incompatible rectangles}
Let $V[m]$ and $W[n]$ be indecomposable objects in $\CAR$ where $V$ and $W$ are indecomposables in the 0th degree.
Then the $g$-vectors $[P_V]-[Q_V]$ and $[P_W]-[Q_W]$ are not $\mathbf E$-compatible if and only if there is a rectangle or almost complete rectangle in the AR-space of $\repAR$, as a subspace of the AR-space of $\DbAR$, whose sides have slopes $\pm(1,1)$ and whose left and right corners are $V$ and $W$ (not necessarily respectively).
\end{proposition}
\begin{proof}
Without loss of generality suppose $V$ is the left corner and $W$ is the right corner of a rectangle or almost complete rectangle in the AR-space of $\repAR$.
Then by Theorem \ref{thm:triangles are rectangles} there is a distinguished triangle $V\to U\to W\to$.
This corresponds to an extension $V\hookrightarrow U\twoheadrightarrow W$ in $\repAR$ and so by Proposition \ref{prop:incompatible means extension} we know  $[P_V]-[Q_V]$ and $[P_W]-[Q_W]$ are not $\mathbf E$-compatible.

Now suppose  $[P_V]-[Q_V]$ and $[P_W]-[Q_W]$ are not $\mathbf E$-compatible.
Without loss of generality suppose
\begin{displaymath}
\langle [P_W]-[Q_W]\, ,\, [P_V]-[Q_V] \rangle  < 0.
\end{displaymath}
Then in $\repAR$ there is an extension $V\hookrightarrow U\twoheadrightarrow W$.
Thus by Theorem \ref{thm:extensions are rectangles} there is a rectangle or almost complete rectangle in the AR-space of $\repAR$ whose left corner is $V$ and right corner is $W$.
\end{proof}

\begin{example}
Below, the bold lines are the isomorphism classes of objects $P[0]$ and $P[1]$ where $P$ is projective $\repAR$.
The objects $V$ and $W$ are two indecomposables whose $g$-vectors are not $\mathbf E$-compatible and are clearly the left and right corners of a rectangle in the AR-space of $\DbAR$ whose sides have slope $\pm(1,1)$.
The points labeled $P$ and $Q$ with subscripts $Vi$ or $Wj$ are indecomposables of the objects $P_V$, $P_W$, $Q_V$, and $Q_W$.
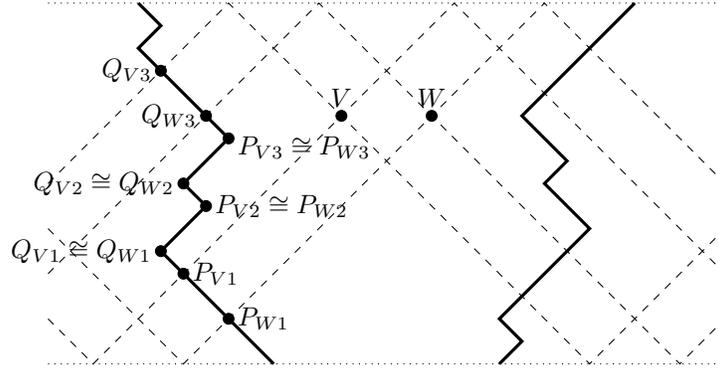
\begin{displaymath}\begin{tikzpicture}[scale=1.2]
\draw[dotted] (-1.5,2) -- (6,2);
\draw[dotted] (-1.5,-2) -- (6,-2);
\draw[very thick](1,-2) --  (-0.25, -0.75) -- (0.25,-0.25) -- (0,0) -- (0.5,0.5) -- (-0.5,1.5) -- (-0.25,1.75) -- (-0.5,2);
\draw[very thick] (5,2) --  (3.75, 0.75) -- (4.25,0.25) -- (4,0) -- (4.5,-0.5) -- (3.5,-1.5) -- (3.75,-1.75) -- (3.5,-2);
\draw[dashed] (-1.5,0) -- (.5,2) -- (4.5,-2) -- (6,-.5);
\draw[dashed] (-1.5,-1) -- (1.5,2) -- (5.5,-2) -- (6,-1.5);
\draw[dashed] (-1.5,-1.5) -- (-1,-2) -- (3, 2) -- (6,-1);
\draw[dashed] (-1.5,-.5) -- (0,-2) -- (4,2) -- (6,0);
\filldraw[fill=black] (1.75,0.75) circle[radius=0.6mm];
\filldraw[fill=black] (2.75,0.75) circle[radius=0.6mm];
\draw (1.75, 0.75) node [anchor=south] {$V$};
\draw (2.75, 0.75) node [anchor=south] {$W$};
\filldraw[fill=black] (-.25,1.25) circle[radius=.6mm];
\filldraw[fill=black] (.25,.75) circle[radius=.6mm];
\filldraw[fill=black] (.5,.5) circle[radius=.6mm];
\filldraw[fill=black] (0,0) circle[radius=.6mm];
\filldraw[fill=black] (0.25,-0.25)  circle[radius=.6mm];
\filldraw[fill=black]  (-0.25, -0.75) circle[radius=.6mm];
\filldraw[fill=black]  (0, -1) circle[radius=.6mm];
\filldraw[fill=black]  (.5, -1.5) circle[radius=.6mm];
\draw  (-0.25, -0.75) node [anchor=east] {$Q_{V1}\cong Q_{W1}$};
\draw (0,0) node [anchor=east] {$Q_{V2}\cong Q_{W2}$};
\draw (-.25,1.25) node[anchor=east] {$Q_{V3}$};
\draw (.25,.75) node[anchor=east] {$Q_{W3}$};
\draw (.5,-1.5) node[anchor=west] {$P_{W1}$};
\draw (0,-1) node[anchor=west] {$P_{V1}$};
\draw  (0.25,-0.25) node[anchor=west] {$P_{V2}\cong P_{W2}$};
\draw (.5,.4) node[anchor=west] {$P_{V3}\cong P_{W3}$};
\end{tikzpicture}\end{displaymath}
We then perform the following computations:
\begin{align*}
\langle [P_W]-[Q_W]\, ,\, [P_V] \rangle \\
= \langle [P_{W1}] + [P_{W2}] + [P_{W3}] - [Q_{W1}] - [Q_{W2}]  - [Q_{W3}] ,\, [P_{V1}] + [P_{V2}] + [P_{V3}]  \rangle \\ 
= -1 -1 -1= -3
\end{align*}
\begin{align*}
\langle [P_W]-[Q_W]\, ,\, -[Q_V] \rangle \\
= \langle [P_{W1}] + [P_{W2}] + [P_{W3}] - [Q_{W1}] - [Q_{W2}]  - [Q_{W3}] ,\,  - [Q_{V1}] - [Q_{V2}]  - [Q_{V3}] \rangle \\ 
= 1 + 1 + 0 =2
\end{align*}
And so we have
\begin{displaymath} \langle [P_W] - [Q_W]\, ,\, [P_V] - [Q_V] \rangle = -1. \end{displaymath}
\end{example}

\subsection{$\mathbf E$-Clusters}
In this section we define $\mathbf E$-clusters and show how to mutate elements in a way similar to the usual notion of mutation.

\begin{definition}\label{def:E-cluster}
Let $T$ be a collection of (isomorphism classes of) indecomposable objects in $\CAR$.
We say $T$ is \ul{$\mathbf E$-compatible} if for any pair $V, W\in T$ the $g$-vectors $[P_V]-[Q_V]$ and $[P_W]-[Q_W]$ are $\mathbf E$-compatible.

If, for any $U\notin T$, there exists a $V\in T$ such that $[P_U]-[Q_U]$ and $[P_V]-[Q_V]$ are not $\mathbf E$-compatible are call $T$ an \ul{$\mathbf E$-cluster}.
\end{definition}

\begin{example}\label{xmp:projective cluster}
Let $\mathcal P$ be the set of (isomorphism classes of) indecomposables $P$ whose $g$-vector is of the form $[P]$.
Then $\mathcal P$ is $\mathbf E$-compatible.

Suppose $V$ is ay other indecomposable with $g$-vector $[P_V]-[Q_V]$. 
Then $[P_V]-[Q_V]$ is not $\mathbf E$-compatible with $[Q_V]$.
Therefore $\mathcal P$ is an $\mathbf E$-cluster.
\end{example}

\begin{lemma}\label{lem:compatible rectangle}
Let $V \to U_1\oplus U_2\to W\to$ be a distinguished triangle in $\CAR$ where $V$, $U_1$, $U_2$, and $W$ are indecomposable.
Suppose further that one may take representatives of each isomorphism class in degree 0 and obtain a (almost complete) rectangle entirely in the AR-space of $\repAR$ as a subspace of the AR-space of $\DbAR$, where $V$ is the left corner and $W$ is the right corner.
Then for any indecomposable $X$ in $\CAR$ if $\{X,V,W\}$ is $\mathbf E$-compatible so is $\{X,U_1,U_2\}$.
Furthermore, if $\{X,V,U_1,U_2\}$ or $\{X,W,U_1,U_2\}$ is $\mathbf E$-compatible so is $\{X,W\}$ or $\{X,V\}$, respectively.
\end{lemma}
\begin{proof}
We will instead prove that if $\{X,U_1,U_2\}$ is not $\mathbf E$-compatible then $\{X,V,W\}$ is not $\mathbf E$-compatible, which is equivalent.
Thus we may assume, without loss of generality, that $\{X,U_1\}$ is not $\mathbf E$-compatible.
We need only to prove that $\{X,V\}$ or $\{X,W\}$ is not $\mathbf E$-compatible.

Suppose one of $\{X,V\}$ or $\{X,W\}$ is $\mathbf E$-compatible.
We shall assume $\{X,V\}$ is $\mathbf E$-compatible as the other assumption is symmetric.
By Proposition \ref{prop:incompatible rectangles}, since $\{X,U_1\}$ is not $\mathbf E$-compatible there is a (almost complete) rectangle in the AR-space of $\repAR$ whose left and right corners are $X$ and $U_1$ (possibly not respectively).

Since $\{X,V\}$ is $\mathbf E$-compatible $X$ must be to the left of $U_1$ or else there would be a (almost complete) rectangle with left corner $V$ and right corner $X$ in the AR-space of $\repAR$.
But then there is a (almost complete) rectangle with left corner $X$ and right corner $W$ and so $\{X,W\}$ is not $\mathbf E$-compatible.

Thus, if $\{X,U_1\}$ or $\{X,U_2\}$ is not $\mathbf E$-compatible then at least one of $\{X,V\}$ or $\{X,W\}$ is not $\mathbf E$-compatible.
Therefore, if $\{X,U_1,U_2\}$ is not $\mathbf E$-compatible then $\{X,V,W\}$ is not $\mathbf E$-compatible.
The furthermore in the lemma follows from a similar argument using AR-space geometry.
\end{proof}

\begin{definition}\label{def:mutable}
Let $T$ be an $\mathbf E$-cluster and $V\in T$.
If there exists $W$ such that $\{V,W\}$ is not $\mathbf E$-compatible but $(T\setminus \{V\})\cup\{W\}$ is $\mathbf E$-compatible we say $V$ is \ul{$\mathbf E$-mutable}.
\end{definition}

\begin{remark}\label{rem:mutable}
Note that we have not required that $(T\setminus \{V\})\cup\{W\}$ be an $\mathbf E$-cluster.
We only require that if $V$ is replaced with $W$ then the new set is $\mathbf E$-compatible.
We will prove later that this means $(T\setminus \{V\})\cup\{W\}$ is indeed an $\mathbf E$-cluster.
\end{remark}

\begin{proposition}\label{prop:rectangle inclusion}
Let $T$ be an $\mathbf E$-cluster and $V\in T$ be $\mathbf E$-mutable with choice $W$.
Then one of the following is the distinguished triangle associated to the (almost complete) rectangle in Proposition \ref{prop:incompatible rectangles} and whichever of $U_1$ and $U_2$ are nonzero are in $T$.
\begin{displaymath}\xymatrix@R=4ex{
\text{(1)} & V \ar[r] & U_1\oplus U_2 \ar[r] & W \ar[r] & V \\
\text{(2)} & W \ar[r] & U_1\oplus U_2 \ar[r] & V\ar[r] & W.
}\end{displaymath}
\end{proposition}
\begin{proof}
We will prove the statement for (1) as (2) is similar.
We know $\{V,U_1,U_2\}$ and $\{W,U_1,U_2\}$ are $\mathbf E$-compatible by Proposition \ref{prop:incompatible rectangles}.
We also know that for all $X\in T$ both $\{X,V\}$ and $\{X,W\}$ are $\mathbf E$-compatible.
Then, by Lemma \ref{lem:compatible rectangle}, for all $X\in T$ we know $\{X,U_1,U_2\}$ is $\mathbf E$-compatible.
Since $T$ is an $\mathbf E$-cluster, this means $U_1,U_2\in T$.
\end{proof}

\begin{lemma}\label{lem:weird compatible rectangle thing}
Let $V$, $W$, and $W'$ be indecomposables in $\CAR$ such that $\{V,W\}$ and $\{V,W'\}$ are not $\mathbf E$-compatible.
Let $U_1$, $U_2$, $U'_1$, and $U'_2$ be the indecomposables from the distinguished triangles in Proposition \ref{prop:incompatible rectangles}.
If $W\not\cong W'$ then at least one of $\{W, U'_1, U'_2\}$, $\{W', U_1,U_2\}$, or $\{U_1,U_2,U'_1,U'_2\}$ is not $\mathbf E$-compatible.
\end{lemma}
\begin{proof}
There are two cases: (1) when $\{W,W'\}$ is not $\mathbf E$-compatible and (2) when $\{W,W'\}$ is $\mathbf E$-compatible.
By Proposition \ref{prop:incompatible rectangles} this is equivalent to: (1) when there is a rectangle or almost complete rectangle in the AR-space of $\repAR$ whose left and right corners are $W$ and $W'$ and (2) when there is no such (almost complete) rectangle.
By symmetry we will assume that, in the AR-space of $\repAR$, $W$ is never to the right of $W'$.

We assume (1) first.
Then by our symmetry assumption $W$ is the left corner and $W'$ is the right corner.
We already know there is a rectangle or almost complete rectangle with left and right corners $W$ and $V$ and similarly for $W'$ and $V$.
There are then three possible places for $V$ horizontally in the AR-space of $\repAR$: (i) to the left of $W$, (ii) between $W$ and $W'$, and (iii) to the right of $W'$.
We see that (i) and (iii) are similar so we'll just focus on (ii) and (iii).
We have the following schematics in the AR-space of $\repAR$:
\begin{displaymath}\begin{tikzpicture}[scale=2]
\foreach \x in {0,1,2,4,5,6}
	\filldraw[fill=black] (\x,0) circle [radius=.3mm];

\foreach \x in {-1,1}
{
	\foreach \y in {0,4}
	{
		\filldraw[fill=black] (0.5+\y, \x*0.5) circle [radius=.3mm];
		\filldraw[fill=black] (1+\y, \x*1) circle [radius=.3mm];
		\filldraw[fill=black] (1.5+\y, \x*0.5) circle [radius=.3mm];
	}
}
\foreach \y in {0,4}
{
	\draw (0+\y,0) -- (1+\y,1) -- (2+\y,0) -- (1+\y,-1) -- (0+\y,0);
	\draw (0.5+\y,0.5) -- (1.5+\y,-0.5);
	\draw (0.5+\y,-0.5) -- (1.5+\y,0.5);
}
\draw (0,0) node[anchor=north] {$W$};
\draw (1,0) node[anchor=north] {$V$};
\draw (2,0) node[anchor=north] {$W'$};
\draw (0.5,0.5) node[anchor=south] {$U_2$};
\draw (0.5,-0.5) node[anchor=north] {$U_1$};
\draw (1.5,0.5) node[anchor=south] {$U'_2$};
\draw (1.5,-0.5) node[anchor=north] {$U'_1$};
\draw (1,1) node[anchor=south] {$X_2$};
\draw (1,-1) node[anchor=north] {$X_1$};
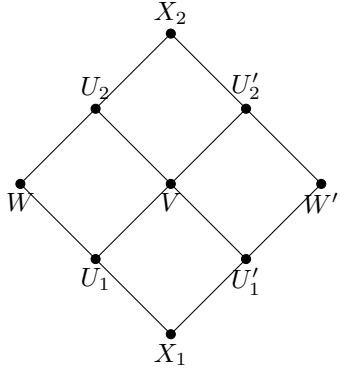
\draw (1,-1.5) node[anchor=north] {Case (1)(ii)};

\draw (4,0) node[anchor=north] {$W$};
\draw (5,0) node[anchor=north] {$W'$};
\draw (6,0) node[anchor=north] {$V$};
\draw (5,1) node[anchor=south] {$U_2$};
\draw (5,-1) node[anchor=north] {$U_1$};
\draw (5.5,0.5) node[anchor=south] {$U'_2$};
\draw (5.5,-0.5) node[anchor=north] {$U'_1$};
\draw (4.5,0.5) node[anchor=south] {$X_2$};
\draw (4.5,-0.5) node[anchor=north] {$X_1$};
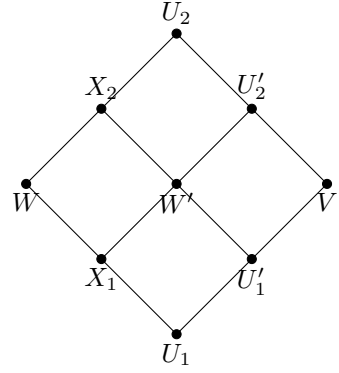
\draw (5,-1.5) node[anchor=north] {Case (1)(iii)};
\end{tikzpicture}\end{displaymath}
In the two schematics, at least one of $X_1, X_2$, least one of $U_1,U_2$, and both of $U'_1,U'_2$ must be nonzero.
In particular, if one of $X_1,X_2$ is 0 we must be in case (1)(ii) so both $U_1$ and $U_2$ are nonzero and similarly for case (1)(iii).

In case (1)(ii), choose $X_i$ to be one of $X_1$ or $X_2$ and nonzero.
Let $j\in\{1,2\}$ such that $\{i,j\}=\{1,2\}$.
Then $X_i$ and $U'_j$ are the top and bottom corners of a rectangle in the AR-space of $\repAR$ whose left corner is $U_i$ and right corner is $W'$.
By Proposition \ref{prop:incompatible rectangles} we see thus means $\{W',U_i\}$ is not $\mathbf E$-compatible and so $\{W',U_1,U_2\}$ is not $\mathbf E$-compatible.

In case (1)(iii) choose $U_i$ to be one of $U_1$ or $U_2$ and nonzero and $j\in\{1,2\}$ such that $\{i,j\}=\{1,2\}$.
We have a rectangle in the AR-space of $\repAR$ with left corner $W$, top and bottom corners $U_i$ and $X_j$, and right corner $U'_i$
By Proposition \ref{prop:incompatible rectangles} again we have $\{W,U'_1,U'_2\}$ is not $\mathbf E$-compatible.

Now we assume (2).
This also comes with subcases.
Either (i) $W$ and $W'$ are on the `same side' of $V$ in the AR-space of $\repAR$ or (ii) $W$ and $W'$ are on `opposite sides' of $V$.
In case (2)(i) this means $\MM^b W$ and $\MM^b W'$ both lie in the $H_V$ as described in \cite[Lemma 2.5.2]{Rock}.
By the same lemma, case (2)(ii) means one of $\MM^b W$ and $\MM^n W'$ lies in $H_V$ and the other in $H_{V[-1]}$.
Since $\{W,W'\}$ is $\mathbf E$-compatible these are equivalent to either (i) one of $W$ or $W'$ being `above' the other in the AR-space of $\repAR$ or (ii) $W$ and $W'$ are `too far apart' to be the left and right corners of a rectangle or almost complete rectangle.
In case (2)(ii), if we draw a rectangle with left and right corners $\MM^b W$ and $\MM^b W'$ in $\R^2$ one of the top or bottom corners will be in the image of the AR-space of $\repAR$ under $\MM$.
(Otherwise, $W$ and $W'$ would not both lie in the AR-space of $\repAR$ as a subspace of the AR-space of $\DbAR$.)

We have the following schematics (where the horizontal dashed line is the lower boundary of the AR-space of $\DbAR$) for cases (2)(i) and (2)(ii):
\begin{displaymath}\begin{tikzpicture}[scale=2]
\foreach \x in {0,1,2,4,5,6}
	\filldraw[fill=black] (\x,0) circle [radius=.3mm];

\foreach \x in {-1,1}
{
	\foreach \y in {0,4}
	{
		\filldraw[fill=black] (0.5+\y, \x*0.5) circle [radius=.3mm];
		\filldraw[fill=black] (1+\y, \x*1) circle [radius=.3mm];
		\filldraw[fill=black] (1.5+\y, \x*0.5) circle [radius=.3mm];
	}
}
\foreach \y in {0,4}
{
	\draw (0+\y,0) -- (1+\y,1) -- (2+\y,0) -- (1+\y,-1) -- (0+\y,0);
	\draw (0.5+\y,0.5) -- (1.5+\y,-0.5);
	\draw (0.5+\y,-0.5) -- (1.5+\y,0.5);
}

\draw[dashed] (3.75,-0.75) -- (6.25,-0.75);
\draw (0.5,0.5) node[anchor=south] {$W$};
\draw (0.5,-0.5) node[anchor=north] {$W'$};
\draw (2,0) node[anchor=north] {$V$};
\draw (1,1) node[anchor=south] {$U_2$};
\draw (1.5,-0.5) node[anchor=north] {$U_1$};
\draw (1,-1) node[anchor=north] {$U'_1$};
\draw (1.5,0.5) node[anchor=south] {$U'_2$};
\draw (0,0) node[anchor=north] {$X_1$};
\draw (1,0) node[anchor=north] {$X_2$};

\draw (4,0) node[anchor=north] {$W$};
\draw (5,0) node[anchor=north] {$V$};
\draw (6,0) node[anchor=north] {$W'$};
\draw (4.5,0.5) node[anchor=south] {$U_2$};
\draw (5.5,0.5) node[anchor=south] {$U'_2$};
\draw (4.5,-0.5) node[anchor=north] {$U_1$};
\draw (5.5,-0.5) node[anchor=north] {$U'_1$};
\draw (5,1) node[anchor=south] {$X$};
\draw (5,-1) node[anchor=north] {$0$};

\draw (1,-1.5) node[anchor=north] {Case (2)(i)};
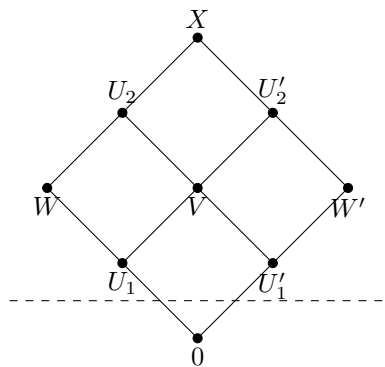
\draw (5,-1.5) node[anchor=north] {Case (2)(ii)};
\end{tikzpicture}\end{displaymath}
At least one of $U_2$ and $U'_1$ must be nonzero in case (2)(i).
Up to reversing roles, we see that there is a rectangle with left corner $W$, top and bottom corners $U_2$ and $X_2$, and right corner $U'_2$.
Thus $\{W,U'_1,U'_2\}$ is not $\mathbf E$-compatible.
In case (2)(ii) we have argued that $X$ must be nonzero so we have a rectangle with left corner $U_2$, top and bottom corners $X$ and $V$, and right corner $U'_2$.
This means $\{U_1,U_2,U'_1,U'_2\}$ is not $\mathbf E$-compatible because $\{U_2,U'_2\}$ is not $\mathbf E$-compatible.
Therefore the proposition holds.
\end{proof}

\begin{theorem}\label{thm:mutation theorem}
Let $T$ be an $\mathbf E$-cluster and $V\in T$ $\mathbf E$-mutable with choice $W$.
Then $(T\setminus \{V\})\cup\{W\}$ is an $\mathbf E$-cluster and any other choice $W'$ for $V$ is isomorphic to $W$.
\end{theorem}
\begin{proof}
First we prove the choice of $W$ is unique up to isomorphism.
By Proposition \ref{prop:rectangle inclusion} we know there are two dinstinguished triangles with indecomposables $U_1$, $U_2$, $U'_1$, and $U'_2$, all of which are in $T$.
By Lemma \ref{lem:weird compatible rectangle thing} we know that if $W\not\cong W'$ then one of $\{W, U'_1, U'_2\}$, $\{W', U_1,U_2\}$, or $\{U_1,U_2,U'_1,U'_2\}$ is not $\mathbf E$-compatible.
Therefore since both $W$ and $W'$ are choices for $V$ we must have $W\cong W'$.

Now let $X$ be an indecomposable in $\CAR$ such that $((T\setminus \{V\})\cup\{W\}) \cup \{X\}$ is $\mathbf E$-compatible.
Then $X$ is $\mathbf E$-compatible with $U_1$ and $U_2$ from the distinguished triangle.
By Lemma \ref{lem:compatible rectangle}, since $\{X,W,U_1,U_2\}$ is $\mathbf E$-compatible so is $\{X,V\}$.
Therefore $X\in T$.
\end{proof}

We may now use the theorem and state the following definition.
\begin{definition}\label{def:mutation}
Let $T$ be an $\mathbf E$-cluster and $V\in T$ be $\mathbf E$-mutable.
We call the indecomposable $W$ such that $(T\setminus\{V\})\cup\{W\}$ is an $\mathbf E$-cluster (Theorem \ref{thm:mutation theorem}) the \ul{$\mathbf E$-replacement} for $V$.

Let $\mu:T \to (T\setminus\{V\})\cup\{W\}$ be the bijection that sends $X$ to $X$ if $X\not\cong V$ and sends $V$ to $W$.
We call $\mu$ an \ul{$\mathbf E$-mutation} and say we have \ul{$\mathbf E$-mutated} $V$ to $W$.
\end{definition}

\section{Relation to Previous Construction}\label{sec:relationship}
This section is dedicated to providing a rigorous connection to the previous construction of the continuous cluster category in \cite{IgusaTodorov1}.

\subsection{Localizations}\label{sec:general localizations}
In this section we create a calculus of fractions in order to construct a triangulated localization of $\DbAR$ and $\CAR$.
We do this using a null system.

\begin{definition}\label{def:degenerate}
Let $V$ be an indecomposable in $\DbAR$.
If $\MM V=(x,\pm\frac{\pi}{2})$ or $\MM^b V=(x,\pm\frac{\pi}{2})$, respectively, for some $x$ we say $V$ is \ul{degenerate}.
Let $V$ now be an indecomposable in $\CAR$, which comes from an indecomposable object $V'$ in $\DbAR$.
If $V'$ is degenerate we say $V$ is also degenerate.

For simplicity we also say the 0 object in each of these categories is degenerate.
Let $V\cong\bigoplus V_i$ be an object in $\repAR$, $\DbAR$, or $\CAR$ where each $V_i$ is indecomposable.
If each $V_i$ is degenerate we say $V$ is degenerate.
\end{definition}

\begin{proposition}\label{prop:degenerate maps}
Let $V$ and $W$ be degenerate indecomposable objects in $\repAR$, $\DbAR$, or $\CAR$, and $f:V\to W$ a morphism.
Then $f$ is either 0 or an isomorphism.
\end{proposition}
\begin{proof}
When $V$ and $W$ are in $\repAR$ or $\DbAR$ this follows from \cite[Section 4.2 and Lemma 5.2.9]{Rock}, respectively.
Now suppose $V$ and $W$ are in $\CAR$ and $f$ is not 0.

Since $\Hom_{\CAR}(V,W)\cong\bigoplus_\Z \Hom_{\DbAR}(V,W[n])$ some shift of $W$ must be isomorphic to $V$ in $\DbAR$.
Then in $\CAR$ we see $V\cong W$.
Since $\Hom(V,V)\cong k$ and $f$ is not 0 $f$ must be an isomorphism.
\end{proof}

\begin{definition}\label{def:null system}
Let $\mathcal D$ be a triangulated category and $\mathcal N$ a full subcategory of $\mathcal D$ such that
\begin{enumerate}
\item the 0 object is in $\mathcal N$,
\item if $X$ is an object in $\mathcal N$ and $Y\cong X$ in $\mathcal D$ then $Y$ is an object in $\mathcal N$,
\item an object $X$ is in $\mathcal N$ if and only if $X[1]$ is in $\mathcal N$, and
\item if $X\to Y\to Z\to X[1]$ is a distinguished triangle in $\mathcal D$ and $X$ and $Z$ are both objects in $\mathcal N$ then so is $Y$.
\end{enumerate}
We call $\mathcal N$ a \ul{null system} in $\mathcal D$.
\end{definition}

\begin{remark}
We note that Definition \ref{def:null system} is the traditional definition used.
However, we also note that (3) and (4) imply (1) and (2).
\end{remark}

We recall the following fact from \cite{KashiwaraSchapira} as a proposition.
\begin{proposition}\label{prop:null is localizing}
Let $\mathcal D$ be a triangulated category and $\mathcal N$ a null system in $\mathcal D$.
Then the class of morphisms
\begin{displaymath}
\mathcal NQ:= \{ X\stackrel{f}{\to} Y : \exists \text{ distinguished triangle }X\stackrel{f}{\to} Y\to Z\to, Z\text{ in }\mathcal N\}
\end{displaymath}
admits a left and right calculus of fractions in $\mathcal D$.
\end{proposition}

\begin{proposition}\label{prop:calculus of fractions}
Let $\mathcal D$ be $\DbAR$ or $\CAR$.
Let $\mathcal N$ be the full, wide subcategory of $\mathcal D$ whose objects are the degenerate objects in $\mathcal D$.
Then $\mathcal N$ is a null system and so $\mathcal N Q$ admits a left and right calculus of fractions.
\end{proposition}
\begin{proof}
Items (1)--(3) in Definition \ref{def:null system} are clear by Definition \ref{def:degenerate}.
Item (4) follows from Proposition \ref{prop:degenerate maps}.
The second half of the conclusion follows from Proposition \ref{prop:null is localizing}.
\end{proof}

We now recall the following well known categorical fact.
\begin{proposition}\label{prop:triangulated localization}
Let $\mathcal D$ be a triangulated category and $\mathcal M$ a class of morphisms that admits a left and right calculus of fractions.
Then there exists a localization $\mathcal D\to \mathcal D[\mathcal M^{-1}]$ such that $\mathcal D[\mathcal M^{-1}]$ is a triangulated category whose distinguished triangles are exactly the images of distinguished triangles in $\mathcal D$.
\end{proposition}

\begin{proposition}
Let $\mathcal D$ be $\DbAR$ or $\CAR$ and $\mathcal N$ the full, wide subcategory of degenerate objects in $\mathcal D$.
Then $\mathcal D[\mathcal NQ^{-1}]$ is a triangulated category whose distinguished triangles are images of distinguished triangles in $\mathcal D$.
\end{proposition}
\begin{proof}
This follows from Propositions \ref{prop:calculus of fractions} and \ref{prop:triangulated localization}.
\end{proof}

\begin{lemma}\label{lem:localization collapse}
Let $\mathcal D$ be either $\DbAR$ or $\CAR$.
Let $V$ and $W$ be not degenerate indecomposables in $\mathcal D$.
\begin{enumerate}
\item If $\mathcal D= \DbAR$ then $V\cong W$ in $\mathcal D[\mathcal NQ^{-1}]$ if and only if $\MM^b V=\MM^b W$.
\item If $\mathcal D=\CAR$ then $V\cong W$ in $\mathcal D[\mathcal NQ^{-1}]$ if and only if there exists $n\in\Z$ such that $\MM^b V = \MM^b (W[n])$.
\end{enumerate}
\end{lemma}
\begin{proof}
In $\CAR$, $W\cong W[n]$ for all $n\in\Z$.
Thus, by replacing $W$ with $W[n]$, and choosing the appropriate $n$, we can prove (2) for $\MM^b V= \MM^b W$.

First suppose $\MM^b V =\MM^b W$ in $\DbAR$.
Without loss of generality, first suppose $\Hom_{\DbAR}(V,W)\cong k$.
By Theorem \ref{thm:triangles are rectangles} any nonzero indecomposable summands of $U$ in the distinguished triangle $W\to U\to V[1]\to $ in $\DbAR$ are degenerate.
Then $U$ is degenerate and so $V\cong W$ in $\DbAR[\mathcal NQ^{-1}]$.
Furthermore if $V$ and $W$ are instead in $\CAR$ then $U$ is still degenerate and so $V\cong W$ in $\CAR[\mathcal NQ^{-1}]$.

Now suppose $V\cong W$ in $\DbAR[\mathcal NQ^{-1}]$.
Then $U$ in the distinguished triangle $W\to U\to V[1]\to $ in $\DbAR$ is degenerate.
We use Theorem \ref{thm:triangles are rectangles} again and see $\MM^b V=\MM^b W$.
Now suppose $V\cong W$ in $\CAR[\mathcal NQ^{-1}]$; so $U$ is degenerate in the distinguished triangle $W\to U\to V\to $ in $\CAR$.
As an orbit category by almost-shift there are choices of lifts $W$, $U$, and $V[1]$ in $\DbAR$ that yield the triangle in $\CAR$.
Again we apply Theorem \ref{thm:triangles are rectangles} and see $\MM^b V=\MM^b W$.
\end{proof}

\begin{lemma}\label{lem:determined by position 1}
Let $\mathcal D$ be either $\DbAR$ or $\CAR$ and $V$ and $W$ be indecomposables in $\mathcal D$.
Let $V_1$ be an indecomposable in $\mathcal D$ such that $\MM^b V=\MM^b V_1$ and $V_1$ has position 1.
Then
\begin{displaymath}
\Hom_{\mathcal D[\mathcal NQ^{-1}]}(V,W)=0 \text{ if and ony if } \Hom_{\mathcal D}(V_1,W)=0.
\end{displaymath}
\end{lemma}
\begin{proof}
Suppose $\Hom_{\mathcal D}(V_1,W)=0$.
Consider a roof
\begin{displaymath}\underline{f}=\xymatrix{V & U\ar[l]_-f \ar[r]^-g & W}\end{displaymath}
in $\mathcal D[\mathcal NQ^{-1}]$.
Since $f\in \mathcal NQ$ we know (shifting if necessary in $\CAR$) that $\MM^b V=\MM^b U$.
Then $\MM^b V_1=\MM^b U=\MM^b V$ and so $\Hom_{\mathcal D}(V_1,U)\cong k\cong \Hom_{\mathcal D}(V_1,V)$.
Let $s:V_1\to U$ be a nontrivial morphism.
We then have the following commutative diagram in $\mathcal D$:
\begin{displaymath}\xymatrix{
V \ar@{=}[dd] & U \ar[l]_-f \ar[r]^-g & W \ar@{=}[dd] \\
& V_1 \ar[u]^-s \ar@{=}[d] & \\
V & V_1 \ar[l]^-{f\circ s} \ar[r]_-{g\circ s} & W.
}\end{displaymath}
Since $f\circ s\in\mathcal NQ$ we see that these two roofs are equivalent in $\mathcal D[\mathcal NQ^{-1}]$.
Denote the bottom roof by $\underline{f}'$.
If $g=0$ then $\underline{f}$ was 0 all along.
If $g\neq 0$ we have $g\circ s=0$ but $f\circ s\in \mathcal NQ$ and so $\underline{f}'=0$.
Thus $\underline{f}$ must be 0 and so $\Hom_{\mathcal D[\mathcal NQ^{-1}]}(V,W)=0$.

Now suppose $\Hom_{\mathcal D}(V_1,W)\neq 0$.
Choose nonzero morphisms $f:V_1\to V$ and $g:V_1\to W$.
Then we have the following roof
\begin{displaymath} \underline{f} = \xymatrix {V & V_1 \ar[l]_-f \ar[r]^-g & W. }\end{displaymath}
For contradiction, suppose $\underline{f}=0$.
Then there is a roof
\begin{displaymath} \underline{h}=\xymatrix{ V & U \ar[l]_-h \ar[r]^-{g'} & W} \end{displaymath}
where $g'=0$ and the following commutative diagram in $\mathcal D$:
\begin{displaymath}\xymatrix{
V \ar@{=}[dd] & V_1 \ar[l]_-f \ar[r]^-g & W \ar@{=}[dd] \\ & \tilde{U} \ar[u]^-s \ar[d]_-t & \\
V & U \ar[l]_-h \ar[r]^-{g'} & W.
}\end{displaymath}
However, this means (up to shifting in $\CAR$) $\MM^b \tilde{U}=\MM^b V_1$ and so $\tilde{U}\cong  V_1$, a contradiction as  then the right side of the diagram would not commute.
Therefore there exists a nonzero $\underline{f}\in \Hom_{\mathcal D[\mathcal NQ^{-1}]}(V,W)$.
\end{proof}

\begin{proposition}\label{prop:still small hom}
Let $\mathcal D$ be either $\DbAR$ or $\CAR$.
Let $V$ and $W$ be indecomposable objects in $\mathcal D[\mathcal N Q^{-1}]$.
Then $\Hom_{\mathcal D[\mathcal N Q^{-1}]}(V,W)\cong k$ or $\Hom_{\mathcal D[\mathcal N Q^{-1}]}(V,W)=0$.
\end{proposition}
\begin{proof}
Choose $\mathcal D=\DbAR$ or $\mathcal D=\CAR$.
Let $V$ and $W$ be indecomposables in $\mathcal D[\mathcal N Q^{-1}]$ such that $\Hom_{\mathcal D[\mathcal N Q^{-1}]}(V,W)\neq 0$.

Let $V_1$ be an indecomposable $\mathcal D$ such that $\MM^b V_1=\MM^b V$ and the position of $V_1$ is 1.
By Lemmas \ref{lem:localization collapse} and \ref{lem:determined by position 1} we know $V\cong_{\mathcal D[\mathcal NQ^{-1}]} V_1$ and $\Hom_{\mathcal D}(V_1,W)\neq 0$.
We will show
\begin{displaymath} \Hom_{\mathcal D[\mathcal NQ^{-1}]}(V,W) \cong  \Hom_{\mathcal D}(V_1,W) \end{displaymath}
by defining two maps
\begin{align*}
\Phi:\Hom_{\mathcal D[\mathcal NQ^{-1}]}(V,W) &\to \Hom_{\mathcal D}(V_1,W) \\
\Psi:\Hom_{\mathcal D}(V_1,W) &\to \Hom_{\mathcal D[\mathcal NQ^{-1}]}(V,W).
\end{align*}

Fix a nonzero morphism $f':V_1\to V$ in $\mathcal D$.
As we saw in the proof of Lemma \ref{lem:determined by position 1} ever nonzero morphism $V\to W$ in $\mathcal D[\mathcal NQ^{-1}]$ is equivalent to a roof whose middle term is $V_1$.
Let
\begin{displaymath} \underline{f} = \xymatrix{ V & V_1 \ar[l]_-f \ar[r]^-g & W} \end{displaymath}
in $\mathcal D[\mathcal NQ^{-1}]$.
Since $\Hom_{\mathcal D}(V_1,V)\cong k$ there is a unique $s:V_1\to V_1$ in $\mathcal D$ such that $f\circ s= f'$.
Similarly, there is a unique $g':V_1\to W$ in $\mathcal D$ such that $g'=g\circ s$.
So we set $\Phi(\underline{f})=g'$.

Let $g':V_1\to W$ be a morphism in $\mathcal D$.
Then there is a roof
\begin{displaymath} \underline{f}'=\xymatrix{ V & V_1\ar[l]_-{f'}\ar[r]^-{g'} & W}\end{displaymath}
in $\mathcal D[\mathcal NQ^{-1}]$.
So we set $\Psi(g')=\underline{f}'$.

Note that $\Psi\Phi(\underline{f})=\underline{f}'$ but the following diagram commutes in $\mathcal D$:
\begin{displaymath}\xymatrix{
V \ar@{=}[dd] & V_1 \ar[l]_-f \ar[r]^-g & W\ar@{=}[dd] \\ & V_1 \ar[u]^-s \ar@{=}[d] & \\
V & V_1\ar[l]_-{f'}\ar[r]^-{g'} & W
}\end{displaymath}
Thus $\underline{f}'=\underline{f}$ in $\mathcal D[\mathcal NQ^{-1}]$ and so $\Psi\Phi(\underline{f})=\underline{f}$.
We see $\Phi\Psi(g') = g'$ since in this case the $s$ from our definition of $\Phi( \Psi(g) )$ is the identity.
Therefore $\Phi = \Psi^{-1}$ and $\Psi = \Phi^{-1}$.
Finally, it is straightforward to check that $\Phi$ and $\Psi$ preserve addition and send 0 to 0.
Thus we have the desired isomorphism and the proposition holds.
\end{proof}

\subsection{Triangulated Equivalences}\label{sec:triangulated relationship}
Here we show the localization of the derived and new continuous cluster categories are triangulated equivalent to the previous derived and continuous cluster categories, respectively.
The localization of the new continuous cluster category is \emph{not} equivalent to the previous continuous cluster category in a way that is $\mathbf E$-compatible with mutation. See Section \ref{sec:cluster difference}.

We recall the construction of the continuous derived category $\mathcal D_r$ from \cite{IgusaTodorov1} as modified in \cite{GarciaIgusa}. The idea is that $\mathcal D_\pi$ is the limit as $n\to\infty$ of the bounded derived category of $\rep_k(Q_n)$ where $Q_n$ is the quiver of type $A_n$ with straight orientation. This will be a topological category with a continuous triangulation. The original definition constructed $\mathcal D_r$ for any positive real number $r$ as the stable category of a Frobenius category. Here we take the approach given in \cite{GarciaIgusa} which does not involve construction of an auxiliary category.

\begin{definition}\label{def:D_pi}
The object space of $\mathcal D_r$ is a subset of the plane:
\begin{displaymath}
	\mathcal O b(\mathcal D_r)=\{(x,y)\in\R^2\,|\, |y-x|<r\}.
\end{displaymath}
Equivalently, $x-r<y<x+r$.
We take the discrete topology on the field $k$ and let
\begin{displaymath}
	\Hom_{\mathcal D_r}(X,Y)=\begin{cases} k & \text{if } x_1\le x_2< y_1+r \text{ and } y_1\le y_2<x_1+r\\
    0& \text{otherwise}
    \end{cases}
\end{displaymath}
Thus, for $X=(x,y)$, the support of $\Hom_{\mathcal D_r}(X,-)$ is the half-open rectangle
\begin{displaymath}
	[x, y+r)\times [y,x+r)
\end{displaymath}
and nonzero morphisms are specified by scalars in $k$. When $Y$ converges to a limit point of this half-open interval, morphisms converge to zero. Then $\mathcal D_r$ is a topological $k$-category in the sense that the object and morphism sets are topological spaces and the structure maps of the category (source, target, composition, identity, scalar multiplication) are continuous. 

Following \cite[Section 4.3]{GarciaIgusa}, the distinguished triangles in $\mathcal D_r$ are constructed out of a family of distinguished triangles called \ul{universal virtual triangles}. For each object $X=(x,y)$ these is a family of distinguished triangles:
\begin{equation}\label{eq: virtual triangle}
			X\xrightarrow{\binom11} I_1^\varepsilon X\oplus I_2^\varepsilon X\xrightarrow{(-1,1)} T^\varepsilon X \xrightarrow{ 1} T X
\end{equation}
for sufficiently small $\varepsilon>0$ where $I_1^\varepsilon X=(x,x+r-\varepsilon)$, $I_2^\varepsilon X=(y+r-\varepsilon,y)$ and $T^\varepsilon X=(y+r-\varepsilon,x+r-\varepsilon)$. Morphisms are given by the indicated scalars. If $X$ has several components, we take the direct sum of the virtual triangles \eqref{eq: virtual triangle} over all components of $X$. Note that, as $\varepsilon\to0$, $T^\varepsilon X$ converges to $TX$ and $T^\varepsilon X \xrightarrow{ 1} T X$ converges to the identity morphism on $TX$. The objects $I_1^\varepsilon X$, $I_2^\varepsilon X$ converge to $0$ as $\varepsilon\to0$.

The distinguished triangles in $\mathcal D_r$ are given as follows. For any morphism $f:X\to Y$ in $\mathcal D_r$, we defined the distinguished triangle $X\to_f Y\to_g Z\to_h TX$ to be the limit as $\varepsilon\to0$ of the following pushout diagram.
\begin{displaymath}
\xymatrixrowsep{20pt}\xymatrixcolsep{35pt}
\xymatrix{
X\ar[d]_f\ar[r]^(.3){\binom11} &
	I_1^\varepsilon X\oplus I_2^\varepsilon X\ar[d]\ar[r]^(.6){(-1,1)} &
	T^\varepsilon X\ar[d]\\
Y \ar[r]^g& 
	Z^\varepsilon \ar[r]^h&
	T^\varepsilon X
	}
\end{displaymath}
As $\varepsilon\to0$, the object $Z^\varepsilon$ converges to an object $Z$. The morphisms $g,h$ also stabilize and the limit is well-defined. (See \cite{GarciaIgusa} for details.)
\end{definition}

\setcounter{figure}{\value{lemma}}

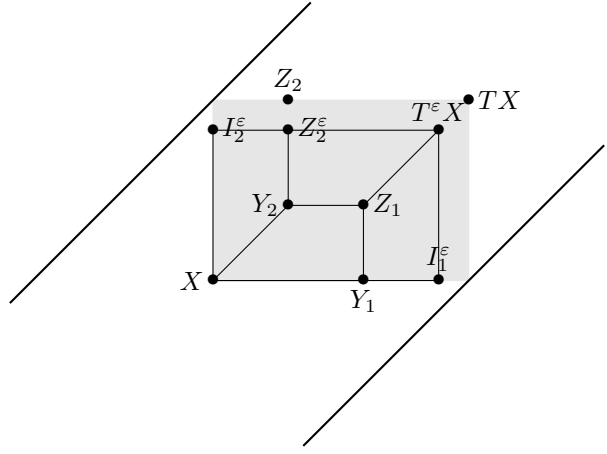
\begin{figure}[htbp]
\begin{center}
\begin{tikzpicture}
\coordinate (X) at (-1,0);
\coordinate (Y1) at (1,0);
\coordinate (Y2) at (0,1);
\coordinate (Z1) at (1,1);
\coordinate (Z2) at (0,2);
\coordinate (Z3) at (0,2.4);
\coordinate (Z) at (2.4,2.4);
\coordinate (I1) at (2,0);
\coordinate (I11) at (2.4,0);
\coordinate (I2) at (-1,2);
\coordinate (I22) at (-1,2.4);
\coordinate (S) at (2,2);
\coordinate (T) at (2.4,2.4);
\draw[fill,color=gray!20!white] (X) rectangle (Z);
\draw (X) node{$\bullet$} node[left]{$X$};
\draw (S) node{$\bullet$} node[above]{$T^\varepsilon X$};
\draw (T) node{$\bullet$} node[right]{$TX$};
\draw (X)--(Y2)--(Z2) (I1)--(X)--(Y1)--(Z1)--(Y2) (X)--(I2)--(S)--(I1) (Z1)--(S);
\draw (Y1) node{$\bullet$} node[below]{$Y_1$};
\draw (Y2) node{$\bullet$} node[left]{$Y_2$};
\draw (Z1) node{$\bullet$} node[right]{$Z_1$};
\draw (Z2) node{$\bullet$} node[right]{$Z_2^\varepsilon$};
\draw (Z3) node{$\bullet$} node[above]{$Z_2$};
\draw (I1) node{$\bullet$} node[above]{$I_1^\varepsilon$};
\draw (I2) node{$\bullet$} node[right]{$I_2^\varepsilon$};
\begin{scope}[xshift=12mm,yshift=8mm]
\draw[thick] (-1,-3)--(3,1);
\end{scope}
\begin{scope}[xshift=-2mm,yshift=2mm]
\draw[thick] (-3.5,-.5)--(0.5,3.5);
\end{scope}
\end{tikzpicture}
\caption{In this example, $Y=Y_1\oplus Y_2$ where $Y_1$ has the same $y$-coordinate as $X$. Then $Z^\varepsilon=I_1^\varepsilon\oplus Z_1\oplus Z_2^\varepsilon$. As $\varepsilon\to 0$, $I_1^\varepsilon$ moves to the right until it becomes $0$, $Z_2^\varepsilon$ moves up to $Z_2$, $Z_1$ stays where it is. Also, $T^\varepsilon X$ goes to $ TX$. Thus the distinguished triangle is $X\to Y_1\oplus Y_2\to Z_1\oplus Z_2\to TX$. The support of $\Hom_{\mathcal D_r}(X,-)$ is shaded.}
\label{Figure99}
\end{center}
\end{figure}

\setcounter{lemma}{\value{figure}}

\begin{definition}\label{def:new to old functor}
We define a functor $G: \DbAR[\mathcal NQ^{-1}]\to \mathcal D_{\pi}$.
We will use the representative objects as in Definition \ref{def:representative}.
Noting Lemma \ref{lem:localization collapse} we will assume our representative object in each isomorphism class has position 1 in the AR-space of $\DbAR$.

Let $V$ be an indecomposable representative object in $\DbAR[\mathcal NQ^{-1}]$.
Let $(x,y)=\MM^b V$.
We define $G V$ to be $M(x-y,x+y)$ in $\mathcal D_{\pi}$.
It is easy to check that $|x-y-x-y|<\pi$.
Since $\Hom_{\DbAR[\mathcal NQ^{-1}]}(V,W)= k$ (Proposition \ref{prop:still small hom}) for two indecomposable representatives $V$ and $W$, we send a morphism $f\in k$ to $f\in k=\Hom_{\mathcal D_{\pi}}(G V, G W)$.
Since both categories are Krull-Schmidt the rest of $G$ is defined by extending bilinearly.
\end{definition}

\begin{theorem}\label{thm:triangulated equivalent}
Assume $A_\R$ has finitely-many sinks and sources.
Then $G$ in Definition \ref{def:new to old functor} is a triangulated equivalence.
\end{theorem}
\begin{proof}
One quickly verifies that $G$ induces a bijection on isomorphism classes of objects and bijections on Hom spaces.
It remains to show that cones in distinguished triangles are taken to cones in distinguished triangles.

Let $V \to U \to W\to $ be a distinguished triangle in $\DbAR[\mathcal NQ^{-1}]$ such that $V$, $U$, and $W$ are all nonzero and distinct and $V$ and $W$ are indecomposable.
Then this comes from a triangle $\tilde{V}\to \tilde{U}\to \tilde{W}\to $ in $\DbAR$ and so $U=U_1\oplus U_2$ where $U_1$ and $U_2$ are indecomposable and at most one is 0.
Furthermore, $\tilde{V}$ and $\tilde{W}$ are indecomposable.
Then by Theorem \ref{thm:triangles are rectangles} there is a rectangle or almost complete rectangle in the AR-space of $\DbAR$ whose sides have slopes $\pm(1,1)$ and whose left and right corner are $\tilde{V}$ and $\tilde{W}$, respectively.
Thus $\tilde{U}$ has at most two indecomposable summands.

Since $U$ is not 0, $\MM^b \tilde{V}, \MM^b\tilde{W}$, and $\MM^b \tilde{U}$ form the corners of a rectangle in $\R^2$ whose left and right corners are $\MM^b \tilde{V}$ and $\MM^b\tilde{W}$, respectively.
Since $U\neq 0$ no more than one indecomposable summand of $\tilde{U}$ may be sent to $\R\times(-\frac{\pi}{2},\frac{\pi}{2})$ by $\MM^b$.

We give the coordinates of the image of $\MM^b$ of the lifts of each indecomposable and compute $G$ on each of our indecomposables below.
First, a few notes.
If one of $\tilde{U}_1$ or $\tilde{U}_2$ is 0 then there is no $\MM^b$ of that indecomposable.
If one of $U_1$ or $U_2$ is 0 then $G$ of that indecomposable will be 0.
However, we will compute $\MM^b$ and $G$ for both possibilities in each case for when those situations arise.
\begin{align*}
{\MM}^b \tilde{V}& =(x,y) & GV &= M(x-y,x+y) \\
{\MM}^b \tilde{U}_1& =(x+\alpha, y-\alpha) &GU_1 &= M(x-y+2\alpha,x+y) \\
{\MM}^b \tilde{U}_2& =(x+\beta, y+\beta) &GU_2 &= M(x-y,x+y+2\beta) \\
{\MM}^b \tilde{W}&= (x+\alpha+\beta,y-\alpha+\beta) &GW &= M(x-y+2\alpha,x+y+2\beta).
\end{align*}
By the description in \cite{IgusaTodorov1} the four indecomposables in the image of $G$ also form a distinguished triangle.
In particular, if $\alpha$ or $\beta$ are 0 then the distinguished triangles in $\DbAR[\mathcal NQ^{-1}]$ and $\mathcal D_\pi$ are split/trivial.
With the same techniques used to prove Theorems \ref{thm:geometric cones} and \ref{thm:derived classification} we see that $G$ takes cones in distinguished triangles to cones in distinguished triangles.
\end{proof}

\begin{theorem}\label{thm:localized cluster category}
Assume $A_\R$ has finitely-many sinks and sources.
Then there is a triangulated equivalence $H:\CAR[\mathcal NQ^{-1}]\to \mathcal C_\pi$.
\end{theorem}
\begin{proof}
Let $\mathcal C$ be the orbit category of $\DbAR[\mathcal NQ^{-1}]$ via doubling and almost-shift as in \cite{IgusaTodorov1}.
Since $\DbAR[\mathcal NQ^{-1}]$ is triangulated equivalent to $\mathcal D_\pi$, by Theorem \ref{thm:triangulated equivalent}, we see there must be a triangulated equivalence $H_2:\mathcal C\to \mathcal C_\pi$.
We will define a triangulated equivalence $H_1:\CAR[\mathcal NQ^{-1}] \to \mathcal C$.
Afterwards we let $H = H_2 \circ H_1$, completing the proof.

Since $\mathcal C$ is an orbit category we choose our fundamental domain.
We choose those indecomposables $V$ that come from an indecomposable $\tilde{V}$ in $\DbAR$ such that $(\alpha,\beta)=\MM^b V$ satisfy
\begin{align*}
-\frac{\pi}{2} < &\beta < \frac{\pi}{2} \\
\beta \leq &\alpha < \pi - \beta.
\end{align*}
This is precisely the image of the the 0th degree indecomposables in $\DbAR$, excluding the injective indecomposables from $\repAR$.

Recall $\CAR$ has the same objects as $\DbAR$ but different isomorphism classes and similarly for $\mathcal C$ and $\DbAR[\mathcal NQ^{-1}]$, respectively.
For each indecomposable $V$ in $\CAR[\mathcal NQ^{-1}]$ there exists a $\tilde{V}$ in $\DbAR$ in degree 0 such that after taking the orbit, $\tilde{V}$ is sent to $V$ in the localization of $\CAR$.
We define $H_1 V$ to be the indecomposable in $\mathcal C$ that comes from an indecomposable $\hat{V}$ in $\DbAR[\mathcal NQ^{-1}]$ that also comes from $\tilde{V}$ in $\DbAR$.

Let $V$ and $W$ be indecomposables in $\CAR[\mathcal NQ^{-1}]$.
We show $\Hom_{\CAR[\mathcal NQ^{-1}]}(V,W)\cong k$ if and only if $\Hom_{\mathcal C}(H_1 V, H_1 W)\cong k$ and similarly for 0 hom spaces.
First, there are $\tilde{V}$ and $\tilde{W}$ in degree 0 in $\DbAR$ that are sent to $V$ and $W$ after taking the orbit and localization.
Then either $\Hom_{\DbAR}(\tilde{V},\tilde{W})\cong k$ or $\Hom_{\DbAR}(\tilde{V},\tilde{W}[1])\cong k$.
Let $\tilde{V}_1$ be an indecomposable in $\DbAR$ such that $\MM^b \tilde{V}_1=\MM^b \tilde{V}$ and $\tilde{V}_1$ has position 1.
If $\Hom_{\DbAR}(\tilde{V}_1,\tilde{W})$ and $\Hom_{\DbAR}(\tilde{V}_1,\tilde{W}[1])$ were both 0 then $\Hom_{\CAR[\mathcal NQ^{-1}]}(V,W)$ would be 0.
Since this is not the case, $\Hom_{\DbAR}(\tilde{V}_1,\tilde{W})\cong k$ or $\Hom_{\DbAR}(\tilde{V}_1,\tilde{W}[1])\cong k$.
Then $\Hom_{\DbAR[\mathcal NQ^{-1}]}(\hat{V},\hat{W})\cong k$ or $\Hom_{\DbAR[\mathcal NQ^{-1}]}(\hat{V},\hat{W}[1])\cong k$.
In either case $\Hom_{\mathcal C}(H_1 V,H_1 W)\cong k$.
In the case that $\Hom_{\CAR[\mathcal NQ^{-1}]}(V,W)=0$ we have 
\begin{displaymath}\Hom_{\DbAR}(\tilde{V}_1,\tilde{W})=0\Hom_{\DbAR}(\tilde{V}_1,\tilde{W}[1])\end{displaymath} and so $\Hom_{\mathcal C}(H_1 V,H_1 W)=0$ as well.

Now, as in Definition \ref{def:the functor} for our triangulated equivalence of derived categories for different continuous quivers of type $A$, we can choose representatives from each isomorphism class of indecomposables and fix isomorphisms between indecomposables and their respective representatives.
With such a construction we set $H_1 (\Hom_{\CAR[\mathcal NQ^{-1}]}(V,W)) := \Hom_{\mathcal C}(H_1 V,H_1 W)$ for each pair of representatives $V$ and $W$ in $\CAR[\mathcal NQ^{-1}]$.
This gives us an equivalence of categories but must still check triangles.

However, each distinguished triangle $U\to V\to W\to$ in $\CAR[\mathcal NQ^{-1}]$ comes from a triangle $\tilde{U}\to \tilde{V}\to\tilde{W}\to$ in the doubling of $\DbAR$.
But after taking localization and then the orbit, the image of $\tilde{U}\to \tilde{V}\to\tilde{W}\to$ is a distinguished triangle in $\mathcal C$.
This is, by definition, precisely the image of $U\to V\to W\to $ under $H_1$.
Thus, $H_1$ is a triangulated equivalence and so is $H=H_2\circ H_1$.
\end{proof}

\subsection{Comparing the Constructions}\label{sec:cluster difference}
In this section we highlight the differences and similarities in between the previous and new constructions.

The cluster structure introduced by Igusa and Todorov in \cite{IgusaTodorov1} requires that the clusters be discrete.
This was required so that every object in the cluster be mutable.
The new theory given in Section \ref{sec:new continuous cluster category} does not come with this restriction.
Accordingly, not all objects in a cluster are $\mathbf E$-mutable in the new theory.
We refer back to Example \ref{xmp:projective cluster} for the description of the cluster in the following example.

\begin{example}\label{xmp:mutate projective cluster}
Choose a particular continuous quiver $A_\R$ of type $A$.
Consider again the $\mathbf E$-cluster $\mathcal P$ in Example \ref{xmp:projective cluster}.
Let $a\in \R$ such that $a$ is neither a source nor a sink.
By Theorem \ref{thm:projs} there are exactly two projectives at $a$ in $\repAR$.

Let $P$ be the object in $\CAR$ that comes from $P_a$ and $Q$ the object that comes from $P_{(a}$ or $P_{a)}$, whichever exists in $\repAR$.
Then we have a distinguished triangle $Q\to P\to V\to $ in $\CAR$ where $V$ is degenerate.
In particular, there is no distinguished triangle $P'\to W\to V\to$ in $\CAR$ where $P'$ is in $\mathcal P$ and not isomorphic to $Q$.
Thus $(\mathcal P \setminus \{Q\})\cup\{V\}$ is $\mathbf E$-compatible and so $Q$ is $\mathbf E$-mutable.

However, for any object $V$ such that $\{V,P\}$ is not $\mathbf E$-compatible we have $\{V,P_{a+\e}\}$ is not $\mathbf E$-compatible for $0<\e<<1$.
Thus $P$ is not $\mathbf E$-mutable.
\end{example}

We would like to mutate all the projectives from $\repAR$ into all the injectives from $\repAR$.
However, this example appears to present a problem.
The next paper in this series will address this with a continuous generalization of mutation.

\section{Embeddings of Cluster Theories}\label{sec:embeddings of cluster theories}
In this section we demonstrate how to embed existing cluster theories (Definition \ref{def:cluster theory}) in the literature into the new continuous cluster theory in a way that is compatible with mutation.
It should be noted that we do not make an attempt at embedding the cluster categories themselves, as this can lead to unanticipated problems (see Section \ref{sec:problems}).
Thus we create new machinery in Section \ref{sec:cluster theories} to rigorously describe what we mean to embed one cluster theory within another.
Other relationships between between cluster theories is outside the scope of this thesis but may be of interest in the future.

The goal of embedding cluster theories is the following.
We hope, and anticipate, a continuous cluster \emph{algebra} at some point in the future.
If we can embed the existing type $A$ cluster theories into the new $\mathbf E$-cluster theory then this will provide the intuition, and perhaps some machinery, useful for the embedding of the relevant algebras.
The $A_n$ type cluster algebras already exist; the cluster category constructions in \cite{BMRRT, CalderoChapatonSchiffler2006} came after.
In the reverse order, an $A_\infty$ type cluster structure was introduced by Holm and J{\o}rgensen in \cite{HolmJorgensen} and the cluster algebra type came later, introduced by Ndoun\'e in \cite{Ndoune}.
A continuous version of a cluster algebra related to this work or the work in \cite{IgusaTodorov1} is not known to the authors.

\subsection{Cluster Theories: $\mathscr T_{\mathbf P}(\mathcal C)$}\label{sec:cluster theories}
This subsection is dedicated to providing a framework in which to talk about embedding cluster theories without requiring a functor between cluster categories. 
It should be noted that cluster theories are not cluster structures as they are often defined (for examples, \cite{CalderoChapatonSchiffler2006, BMRRT, HolmJorgensen}).
While cluster structures require each indecomposable in a cluster be mutable and mutation be given by approximations, cluster theories do not make such a requirement.
Instead, we require that \emph{if} an indecomposable is mutable then there is a unique choice for a replacement.
In practice this should be related to some homological or other algebraic property but we do not explicitly make this requirement.
The reader should recall that a pairwise compability condition on indecomposable objects takes unordered pairs of indecomposables and determines their compatibility, thus allowing larger sets to be called compatible if every subset of 2 elements is pairwise compatible.

\begin{definition}\label{def:cluster theory}
Let $\mathcal C$ be a skeletally small Krull-Schmidt additive category in which there exists a pairwise compatibility condition $\mathbf P$ on (isomorphism classes of) indecomposable objects.
Suppose also that for each (isomorphism class of) indecomposable $X$ in a maximally $\mathbf P$-compatible set $T$ there exists none or one (isomorphism class of) indecomposable $Y$ such that $\{X,Y\}$ is not $\mathbf P$-compatible but $(T\setminus\{X\})\cup\{Y\}$ is maximally $\mathbf P$-compatible. Then
\begin{itemize}
\item We call the maximally $\mathbf P$-compatible sets \ul{$\mathbf P$-clusters}.
\item We call a function of the form $\mu:T\to (T\setminus\{X\})\cup\{Y\}$ such that $\mu Z=Z$ when $Z\neq X$ and $\mu X=Y$ a \ul{$\mathbf P$-mutation} or \ul{$\mathbf P$-mutation at $X$}.
\item If there exists a $\mathbf P$-mutation $\mu:T\to (T\setminus\{X\})\cup\{Y\}$ we say $X\in T$ is \ul{$\mathbf P$-mutable}.
\item The subcategory $\mathscr T_{\mathbf P}(\mathcal C)$ of $\mathcal S\text{et}$ whose objects are $\mathbf P$-clusters and whose morphisms are generated by $\mathbf P$-mutations (and identity functions) is called \ul{the $\mathbf P$-cluster theory of $\mathcal C$}.
\item The functor $I_{\mathbf P,\mathcal C}:\mathscr T_{\mathbf P}(\mathcal C)\to \mathcal S\text{et}$ is the inclusion of the subcategory.
\end{itemize}
\end{definition}

From now on, when we say ``a Krull-Schmidt category'' we mean ``a skeletally small Krull-Schmidt additive category.''

\begin{remark}\label{rmk:pairwise induces theory}
Let $\mathcal C$ be a Krull-Schmidt category and $\mathbf P$ a pairwise compatibility condition on the indecomposable objects in $\mathcal C$.
If the $\mathbf P$-cluster theory of $\mathcal C$ exists then it is completely determined by $\mathbf P$.
Thus we say that $\mathbf P$ \ul{induces} the $\mathbf P$-cluster theory of $\mathcal C$.
\end{remark}

\begin{remark}\label{rmk:maximality by zorn}
Let $\mathcal C$ be a Krull-Schmidt category and $\mathbf P$ a pairwise compatibility condition on $\Ind(\mathcal C)$.
Using Zorn's lemma we note that there exist maximally $\mathbf P$-compatible sets of indecomposables of $\mathcal C$.
\end{remark}

\begin{definition}\label{def:tilting cluster theory}
Let $\mathcal C$ be a Krull-Schmidt category and $\mathbf P$ a pairwise compatibility condition such that $\mathbf P$ induces the $\mathbf P$-cluster theory of $\mathcal C$.
If for every $\mathbf P$-cluster $T$ and $X\in T$ there is a $\mathbf P$-mutation at $X$ then we call $\mathscr T_{\mathbf P}(\mathcal C)$ the \ul{tilting $\mathbf P$-cluster theory}.
\end{definition}

\begin{remark}
The reader may be familiar with frozen elements of a cluster and notice that sometimes an indecomposable $X$ in a $\mathbf P$-cluster $T$ may not be $\mathbf P$-mutable.
However, we do not call $X$ frozen.
This is because in the next paper of this series we will introduce a continuous generalization of mutation that allows some $X$ which are not $\mathbf P$-mutable to \emph{become} $\mathbf P$-mutable.
The word frozen, then, should be reserved for indecomposables we have intentionally frozen or those that may \emph{never} be made $\mathbf P$-mutable.
\end{remark}

\begin{proposition}
Let $\mathcal C$ be a Krull-Schmidt category and $\mathbf P$ a pairwise compatibility condition such that $\mathbf P$ induces the $\mathbf P$-cluster theory of $\mathcal C$.
Let $T$ be a $\mathbf P$-cluster and $X\in T$ such that there exists a $\mathbf P$-mutation $T\to (T\setminus\{X\})\cup\{Y\}$.
Then there exists a $\mathbf P$-mutation $T'=(T\setminus\{X\})\cup\{Y\}\to (T'\setminus\{Y\})\cup\{X\}$.
\end{proposition}
\begin{proof}
We know $\{Y,X\}$ is not $\mathbf P$-compatible but $(T'\setminus\{Y\})\cup\{X\}$ is maximally $\mathbf P$-compatible.
\end{proof}

\begin{example}\label{xmp:BMRRT model}
Our first example of a cluster theory is the cluster structure defined in \cite{BMRRT} for cluster categories $\mathcal C(Q)$ of dynkin quivers $Q$.
Consider clusters in such cluster categories like those described in \cite{BMRRT} as maximally rigid sets of indecomposables instead of the subcategories generated by those indecomposables.

We define our pairwise compatibility condition $\mathbf R$ to be rigidity.
Then the $\mathbf R$-clusters are maximally rigid sets of indecomposables in $\mathcal C(Q)$ and $\mathbf R$-mutations are traditional cluster tilting in $\mathcal C(Q)$.
This yields the tilting $\mathbf R$-cluster theory of $\mathcal C(Q)$.
\end{example}

\begin{example}\label{xmp:IRT model}
The Euler form on $K_0^{\text{split}}(\CAR)$ in Definition \ref{def:euler form} is by definition a pairwise compatibility condition that we have already called $\mathbf E$-compatibility.
We have shown in Theorem \ref{thm:mutation theorem} that $\mathbf E$ induces the $\mathbf E$-cluster theory of $\CAR$.
In the $\mathbf E$-cluster $T$ in Example \ref{xmp:mutate projective cluster} we see that there exist indecomposables $P$ which are not $\mathbf E$-mutable.
Thus the $\mathbf E$-cluster theory is not tilting.
\end{example}

\begin{example}\label{xmp:CCS model}
Our final example for now is the triangulations of the $(n+3)$-gon model introduced in \cite{CalderoChapatonSchiffler2006}.
In \cite{CalderoChapatonSchiffler2006}, the authors describe the $A_n$ cluster structure as triangulations of the $(n+3)$-gon.
This arises in a category $\mathcal C(A_n)$ whose indecomposable objects are diagonals of the $(n+3)$-gon and rigidity is given by two diagonals not crossing.

We let $\mathbf N_n$ be the pairwise compatibility condition of not crossing.
We let $\mathbf N_n$-clusters be maximal sets of noncrossing diagonals and let $\mathbf N_n$-mutations be the exchanging of one diagonal for another to produce a different triangulation.
This is the tilting $\mathbf N_n$-cluster theory of $\mathcal C(A_n)$.
\end{example}

\begin{definition}\label{def:cluster embedding}
Let $\mathcal C$ and $\mathcal D$ be two Krull-Schmidt categories with respective pairwise compatibility conditions $\mathbf P$ and $\mathbf Q$.
Suppose these compatibility conditions induce the $\mathbf P$-cluster theory and $\mathbf Q$-cluster theory of $\mathcal C$ and $\mathcal D$, respectively.

Suppose there exists a functor $F:\mathscr T_{\mathbf P}(\mathcal C) \to \mathscr T_{\mathbf Q}(\mathcal D)$ such that $F$ is an injection on objects and an injection from $\mathbf P$-mutations to $\mathbf Q$-mutations.
Suppose also there is a natural transformation $\eta: I_{\mathbf P,\mathcal C} \to  I_{\mathbf Q, \mathcal D}\circ F$ whose morphisms $\eta_T: I_{\mathbf P,\mathcal C}(T) \to I_{\mathbf Q, \mathcal D}\circ F(T)$ are all injections.
Then we call $(F,\eta):\mathscr T_{\mathbf P}(\mathcal C) \to \mathscr T_{\mathbf Q}(\mathcal D)$ an \ul{embedding of cluster theories}.
\end{definition}

\begin{remark}\label{5rmk:minimal requirements}
Observing Definition \ref{def:cluster embedding} we see that if we can produce the functions involved in the natural transformation then we obtain the embedding of cluster theories
\end{remark}

We will not immediately provide an example of Definition \ref{def:cluster embedding} as Sections \ref{sec:An embedding} and \ref{sec:Ainfty embedding} are devoted to such examples.

We conclude this subsection with the definition of an $\mathbf E$-compatible set $T_\infty$ for straight descending $A_R$ we will use in Sections \ref{sec:An embedding} and \ref{sec:Ainfty embedding}.

\begin{definition}\label{def:Tinfty}
Let $\{a_i\}_{i\in\Z}$ be a collection of real numbers such that
\begin{itemize} \item $a_i < a_{i+1}$ for all $i\in \Z$ and \item $\lim_{i\to-\infty} a_i, \lim_{i\to+\infty} a_i \in \Z$. \end{itemize}
Let $a_{-\infty} = \lim_{i\to-\infty} a_i$ and $a_{+\infty} = \lim_{i\to + \infty} a_i$.
For each $i,j,\ell\in\Z$ such that $l\geq 0$ and $0\leq j \leq 2^\ell$ define
\begin{displaymath} a_{i,j,\ell} := a_i + \left(\frac{j}{2^\ell}\right)(a_{i+1} - a_i). \end{displaymath}
For each $a_i$, we define the following $\mathbf E$-compatible set:
\begin{align*}
T_{a_i} :=& \left\{ M_{(a_{i,j,\ell},\,  a_{i,j+1,\ell})} :j,\ell\in\Z,\,  \ell \geq 0,\, 0\leq j < 2^\ell\right\} \\
&\cup \left\{ M_{\{x\}} : x\in (a_i,a_i+1),\, x \neq a_{i,j,l} ,\,\, j,\ell\in\Z,\,  \ell \geq 0,\, 0\leq j < 2^\ell\right\}, 
\end{align*}
Note that for any $a_i$ and $a_j$ then $T_{a_i}\cup T_{a_j}$ is $\mathbf E$-compatible.
Now, for each $i\in\Z$ such that $i< a_{-\infty}$ or $i\geq a_{+\infty}$ define a similar type of $\mathbf E$-compatible set:
\begin{align*}
T_i :=& \left\{ M_{(i + j/2^\ell, i+ (j+1)/2^\ell)} : j,\ell\in\Z,\, \ell\geq 0,\, 0\leq j < 2^\ell\right\} \cup \{P_{i+1}\}\\
& \cup \left\{ M_{\{x\}} : x\in (i,i+1),\, x \neq i + j/2^\ell,\, j,\ell\in\Z,\,  \ell \geq 0,\, 0\leq j < 2^\ell \right\}
\end{align*}
The $\mathbf E$-compatible set we want is
\begin{align*}
T_\infty := &\left( \bigcup_{i\in\Z} T_{a_i} \right) \cup \left(\bigcup_{i< a_{-\infty} \text{ or } i \geq a_{+\infty}} T_i \right) \\
& \cup \{M_{(a_{-\infty},a_{+\infty})}, P_{+\infty}\} \cup \left\{P_i : i \leq a_{-\infty} \text{ or } i \geq a_{+\infty} \right\}.
\end{align*}
\end{definition}

\subsection{Embeddings $\mathscr T_{\mathbf N_n} (\mathcal C(A_n)) \to \mathscr T_{\mathbf E}(\CAR)$} \label{sec:An embedding}
In this section we demonstrate how to embed type $A_n$ cluster theories into the $\mathbf E$-cluster theory of $\CAR$.
We will assume that $A_\R$ has the straight descending orientation.

We will use the $\mathbf N_n$-cluster theory in Example \ref{xmp:CCS model}.
Choose $n$ and label the vertices of the $(n+3)$-gon counterclockwise.
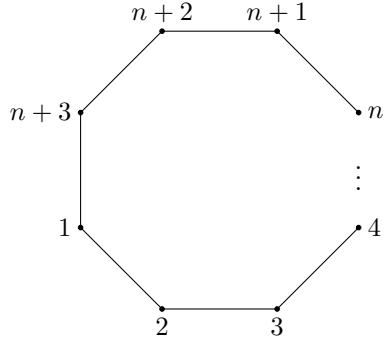
\begin{displaymath}\begin{tikzpicture}
 \draw (1.848,.765) -- (.765,1.848) -- (-.765,1.848) -- (-1.848, .765) -- (-1.848,-.765) -- (-.765,-1.848) -- (.765,-1.848) -- (1.848,-.765);
 \draw (1.848,0) node {$\vdots$};
 \draw (1.848,.765) node [anchor=west] {$n$};
 \draw (.765,1.848) node [anchor=south] {$n+1$};
 \draw (-.765,1.848) node[anchor=south] {$n+2$};
 \draw (-1.848, .765) node[anchor=east] {$n+3$};
 \draw (-1.848,-.765) node[anchor=east] {$1$};
 \draw (-.765,-1.848) node[anchor=north] {$2$};
 \draw (.765,-1.848) node[anchor=north] {$3$};
 \draw (1.848,-.765) node[anchor=west] {$4$};
 \filldraw (1.848,.765) circle[radius=.3mm];
 \filldraw (.765,1.848) circle[radius=.3mm];
 \filldraw (-.765,1.848) circle[radius=.3mm];
 \filldraw (-1.848, .765) circle[radius=.3mm];
 \filldraw (-1.848,-.765) circle[radius=.3mm];
 \filldraw (-.765,-1.848) circle[radius=.3mm];
 \filldraw (.765,-1.848) circle[radius=.3mm];
 \filldraw (1.848,-.765) circle[radius=.3mm];
\end{tikzpicture}\end{displaymath}

We will label a diagonal in the $(n+3)$-gon by $i\diag j$, where $i<j$.
An $\mathbf N_n$-cluster is a maximal collection of noncrossing diagonals; this is also called a triangulation of the $(n+3)$-gon.
A pair of diagonals $i\diag j$ and $i'\diag j'$ cross if and only if $i<i'<j<j'$ or $i'<i<j'<j$.

Recall the notation $M_{|a,b|}$ from Defintion \ref{note:interval indecomposable} and Notation \ref{note:intervals}.

\begin{definition}\label{def:Tn}
Assume $A_\R$ has the straight orientation. 
Again let $a_i$ (for all $i\in\Z$), $a_{-\infty}$, and $a_{+\infty}$ be as in Definition \ref{def:Tinfty}.
Let 
\begin{displaymath} 
T_n := T_{\infty} \cup \{M_{(a_i,a_1)}: i < 0\} \cup \{M_{(a_1,a_j)}: j \geq n+3\} \cup \{M_{(a_{-\infty},a_1)}, M_{(a_1,a_{+\infty})}\}
\end{displaymath} 
\end{definition}
It is straightforward to check that $T_n$ is $\mathbf E$-compatible.

\begin{definition}\label{def:An arcs}
Let $T_{\mathbf N_n}$ be an $\mathbf N_n$-cluster in $\mathscr T_{\mathbf N_n}(\mathcal C(A_n))$ as described.
We will construct an $\mathbf E$-cluster $T_{\mathbf E}$ in $\mathscr T_{\mathbf E}(\CAR)$ based on $T_{\mathbf N_n}$.
Thus, define $M_{i\diag j} := M_{(a_i,a_j)}$.
\end{definition}

\begin{proposition}\label{prop:noncrossings A_n}
A set of diagonals $\{i\diag j,i'\diag j'\}$ is $\mathbf N_n$-compatible if and only if the set $\{M_{i\diag j},M_{i'\diag j'}\}$ is $\mathbf{E}$-compatible.
\end{proposition}
\begin{proof}
Suppose $i\diag j$ is not $\mathbf N_n$-compatible with $i'\diag j'$.
Then, up to symmetry, $i<i'<j<j'$.
We then know there exists a rectangle in the AR-space of $\repAR$ whose left corner is $M_{i\diag j}$, top corner is $M_{(a_i,a_{j'})}$, bottom corner is $M_{(a_{i'}, a_j)}$, and right corner is $M_{i'\diag j'}$.
Thus $M_{i\diag j}$ and $M_{i'\diag j'}$ are not $\mathbf E$-compatible.

If we start with $M_{i\diag j}$ and $M_{i'\diag j'}$ are not $\mathbf E$-compatible we get the rectangle in the AR-space of $\repAR$ again which implies (up to symmetry) that $a_i < a_{i'} < a_j < a_{j'}$ and so $i<i'<j<j'$.
Therefore $i\diag j$ and $i'\diag j'$ are not $\mathbf N_n$-compatible so the proposition holds.
\end{proof}

\begin{definition}\label{def:image of TNn}
Given an $\mathbf N_n$-cluster $T_{\mathbf N_n}$, let
\begin{displaymath} T_{\mathbf E_n} = T_n \cup \left\{M_{i\diag j}: i\diag j \in T_{\mathbf N_n}\right\}. \end{displaymath}
\end{definition}

With Proposition \ref{prop:noncrossings A_n} it is straightforward to check $T_{\mathbf E_n}$ is $\mathbf E$-compatible.

\begin{proposition}\label{prop:TE is an E-cluster}
The $\mathbf E$-compatible set $T_{\mathbf E_n}$ is an $\mathbf E$-cluster.
\end{proposition}
\begin{proof}
Choose some indecomposable $M_{|c,d|}$ in $\CAR$ such that $\{M_{|c,d|}\}\cup T_{\mathbf E}$ is $\mathbf E$-compatible.
We will show that $M_{|c,d|}\in T_{\mathbf E}$.
Recall the $|$s in our notation mean we not assuming whether or not $c$ or $d$ is in the interval $|c,d|$.
We will check the various possible values of $c$ to complete the proof; note that $c < +\infty$.

Suppose $c =-\infty$.
Then either $d\leq a_{-\infty}$ or $d\geq a_{+\infty}$.
If $d=a_{\pm\infty}$ then $M_{|c,d|}$ must be the corresponding open projective at $a_{\pm\infty}$.
If $d < a_{-\infty}$ or $d>a_{+\infty}$ then $d\in\Z$ or $d=+\infty$ as $M_{(i,i+1)}\in T_i$ from Definition \ref{def:Tinfty}.
Thus in all these cases, $M_{|c,d|}$ must be $P_i$ for some $i\in\Z$ outside $(a_{-\infty},a_{+\infty})$ or $i=+\infty$.
If $-\infty < c < a_{-\infty}$ or $a_{+\infty}\leq c<+\infty$ then either $c=i+j/2^\ell$ and $d=i+(j+1)/2^\ell$ for some $i$, $\ell\geq 0$, $0\leq j < 2^\ell$ or $c=d$ and $M_{|c,d|}=M_{\{c\}}$.

Suppose $c=a_{-\infty}$ 
Then $d = a_1$ or $d =a_{+\infty}$. 
Thus, $M_{|c,d|}$ must be $M_{(a_{-\infty},a_1)}$ or $M_{(a_{-\infty},a_{+\infty})}$.
Suppose $a_{-\infty} < c < a_{+\infty}$ and $c\neq a_i$ for any $i$.
If $c=a_{i,j,\ell}$ for some $i\in\Z$, $\ell \geq 0$, and $0<j<(2^\ell)$ then, up to adjusting $\ell$, $d=a_{i,j+1,\ell}$.
If $c$ is not of this form then $d=c$ and $M_{|c,d|}=M_{\{c\}}$.

If $c=a_i$ for some $i\notin \{1,\ldots, n+2\}$ then either $d=a_{i+1}$, $d=a_1$, or $d=+\infty$.
Thus $M_{|c,d|}$ is one of $M_{(a_i,a_{i+1})}|$, $M_{(a_i,a_1)}$, or $I_{(a_1}$. 

So now we check $c\in \{a_i\}_{i=1}^{n+2}$.
First assume $M_{|c,d|}\neq M_{i\diag j}$ for any $i\diag j$. 
If $c=a_1$ then either $d=a_{n+3}$, $d=a_{+\infty}$, or $d\in (a_1,a_2]$.
If $c=a_i$ for $1<i<n+3$ then $d\in (a_i,a_{i+1}]$. 
In any of these cases, $M_{|c,d|}=M_{(a_1,a_{n+3})}$, $M_{|c,d|}=M_{(a_i, a_{i,j,l})}$, $M_{|c,d|}=M_{(a_1,a_{+\infty})}$, or $M_{|c,d|}=M_{(a_i,a_{i+1})}$.

Now the only possibility left to check is $c\in\{a_i\}_{i=1}^{n+2}$ and $M_{|c,d|}=M_{i\diag j}$ for some $i\diag j$, not necessarily in $T_{\mathbf N_n}$.
For contradiction, assume $i\diag j\notin T_{\mathbf N_n}$.
Then there is a $i'\diag j'\in T_{\mathbf N_n}$ such that $\{i\diag j, i'\diag j'\}$ is not $\mathbf N_n$-compatible.
By symmetry suppose $i<i'<j<j'$.
However, $M_{i\diag j} \to M_{(a_{i'}, a_j)}\oplus M_{(a_i, a_{j'})} \to M_{i'\diag j'}$ is an extension in $\repAR$ and so there is a rectangle in the AR-space of $\repAR$ whose left and right corners are $M_{i\diag j}$ and $M_{i'\diag j'}$, respectively (Theorem \ref{thm:extensions are rectangles}).
Then $M_{i\diag j}$ and $M_{i'\diag j'}$ are not $\mathbf E$-compatible by Proposition \ref{prop:incompatible rectangles}.
Since $M_{i'\diag j'}\in T_{\mathbf E}$, we have a contradiction.
\end{proof}

\begin{lemma}\label{lem:mutation is preserved}
Consider an $\mathbf N_n$-cluster $T_{\mathbf N_n}$ as above and the induced $\mathbf E$-cluster $T_{\mathbf E_n}$.
Suppose $T_{\mathbf N_n} \to (T_{\mathbf N_n} \setminus \{i\diag j\} )\cup\{i'\diag j'\}$ is an $\mathbf N_n$-mutation.
Then $T_{\mathbf E_n} \to (T_{\mathbf E_n}\setminus \{M_{i\diag j} \})\cup\{M_{i'\diag j'}\}$ is an $\mathbf E$-mutation.
\end{lemma}
\begin{proof}
It suffices to show that $\{M_{i\diag j}, M_{i'\diag j'}\}$ is not $\mathbf E$-compatible and $(T_{\mathbf E_n}\setminus \{M_{i\diag j} \})\cup\{M_{i'\diag j'}\}$ is $\mathbf E$-compatible.
By the end of the proof of Proposition \ref{prop:TE is an E-cluster}, we see that $\{M_{i\diag j}, M_{i'\diag j'}\}$ is not $\mathbf E$-compatible.
Let $T'_{\mathbf N_n} = (T_{\mathbf N_n} \setminus \{i\diag j\} )\cup\{i'\diag j'\}$.
Note that $(T_{\mathbf E_n}\setminus \{M_{i\diag j} \})\cup\{M_{i'\diag j'}\} = T'_{\mathbf E_n}$.
By Proposition \ref{prop:TE is an E-cluster}, $T'_{\mathbf E_n}$ is an $\mathbf E$-cluster.
Therefore $T_{\mathbf E_n} \to (T_{\mathbf E_n}\setminus \{M_{i\diag j} \})\cup\{M_{i'\diag j'}\}$ is an $\mathbf E$-mutation.
\end{proof}

\begin{theorem}\label{thm:finite embedding}
There exists an embedding of cluster theories $(F,\eta):\mathscr T_{\mathbf N_n}(\mathcal C(A_n))\to \mathscr T_{\mathbf E_n}(\CAR)$.
\end{theorem}
\begin{proof}
By Lemma \ref{lem:mutation is preserved} we see that defining $F(T_{\mathbf N_n}):= T_{\mathbf E_n}$ and sending $\mathbf N_n$-mutations to the corresponding $\mathbf E$-mutations described in the lemma yields a functor.
Since $M_{i\diag j}\not\cong M_{i'\diag j'}$ if $i\neq i'$ or $j\neq j'$ we see that if $T_{\mathbf N_n}\neq T'_{\mathbf N_n}$ then $T_{\mathbf E_n}\neq T'_{\mathbf E_n}$.
Furthermore, $F$ is injective on clusters and generating morphisms between clusters.

Let $\eta_{T_{\mathbf N_n}}: T_{\mathbf N_n}\to T_{\mathbf E_n}$ be given by $\eta_{T_{\mathbf N_n}} (i\diag j) = M_{i\diag j}$.
We see this is an injection and as described in Lemma \ref{lem:mutation is preserved} these injections commute with mutations.
Therefore $(F,\eta)$ is an embedding of cluster theories.
\end{proof}

\subsection{Embedding $\mathscr T_{\mathbf N_{\infty}} (\mathcal C(A_\infty)) \to \mathscr T_{\mathbf E}(\CAR)$}\label{sec:Ainfty embedding}

In this section we demonstrate how to embed type $A_\infty$ cluster theory into the $\mathbf E$-cluster theory of $\CAR$.
Consider again $A_\R$ with the straight descending orientation.
In particular, we address the structure introduced by Holm and J{\o}rgensen in \cite{HolmJorgensen} using triangulations of the infinity-gon.

We will discuss how to embed the closely related structure on the completed infinity-gon (introduced by Baur and Graz in \cite{BaurGraz}) in the next paper of this series.

\begin{definition}(From \cite{HolmJorgensen})\label{def:infinity gon}
The infinity-gon is has its vertices indexed by $\Z$ and no vertex at infinity.
An arc is a pair of integers $(i,j)$ such that $i<j$ and $i-j\geq 2$.
Two arcs $(i,j)$ and $(i',j')$ are defined to cross if and only if $i<i'<j<j'$ or $i'<i<j'<j$.
We will call this pairwise compatibility condition $\mathbf N_\infty$.
Note the similarity to the definitions in Section \ref{sec:An embedding}.
Thus we will write our arcs as $i\diag j$.

The authors show there is a triangulated category whose indecomposables are the diagonals that are compatible if and only if they do not cross.
We will denote the cluster category in which the arcs exist by $\mathcal C(A_\infty)$.
\end{definition}

We continue to let $\{a_i\}$ be the sequence from Definition \ref{def:Tinfty} and define $M_{i\diag j} := M_{(a_i,a_j)}$.
The following proposition is proved in precisely the same fashion as Proposition \ref{prop:noncrossings A_n}.

\begin{proposition}\label{prop:noncrossings A_infty}
Two arcs $i\diag j$ and $i'\diag j'$ are $\mathbf N_\infty$-compatible if and only if $M_{i\diag j}$ and $M_{i'\diag j'}$ are $\mathbf E$-compatible.
\end{proposition}

\begin{definition}\label{def:image of TEinfty}
Let $T_{\mathbf N_\infty}$ be an $\mathbf N_\infty$-cluster.
Define $T_{\mathbf E_\infty}^\circ$ to be
\begin{displaymath}
T_{\mathbf E_\infty}^\circ = T_{\infty} \cup \{M_{i\diag j}: i\diag j \in T_{\mathbf N_\infty}\}.
\end{displaymath}
\end{definition}
Similar to Section \ref{sec:An embedding} we see that with Proposition \ref{prop:noncrossings A_infty} it is straightforward to check that $T_{\mathbf E_\infty}$ is $\mathbf E$-compatible.

It is not true that $T_{\mathbf E_\infty}^\circ$ is always an $\mathbf E$-cluster.
That is, the clusters considered in \cite{HolmJorgensen} form a \emph{proper} subcategory of $\mathscr T_{\mathbf N_\infty}(\mathcal C(A_\infty))$.
The $\mathbf N_\infty$-cluster theory includes what the authors in \cite{HolmJorgensen} called ``weak clusters.''
Put another way: the $\mathbf N_\infty$-cluster theory of $\mathcal C(A_\infty)$ is not tilting.

\begin{example}
Consider the $\mathbf N_\infty$-cluster
\begin{displaymath} T_{\mathbf N_\infty} = \{ i\diag 0: i < -1\} \cup \{1\diag j : j>2\} \end{displaymath}
as in \cite[Sketch 3, p.279]{HolmJorgensen}.
This is maximally $\mathbf N_{\infty}$-compatible but $T_{\mathbf E_\infty}^\circ$ is not maximally $\mathbf E$-compatible.
However, one may check that
\begin{displaymath}
T_{\mathbf E_\infty} := T_{\mathbf E_\infty} \cup \{M_{(a_{-\infty},a_0)}, M_{(a_{-\infty},a_1)}, M_{(a_1,a_{+\infty})}\}
\end{displaymath}
is maximally $\mathbf E$-compatible.
\end{example}

One issue with $T_{\mathbf N_\infty}$ is addressed by the authors in \cite{HolmJorgensen}: $\add\, T_{\mathbf N_\infty}$ is not functorially finite in $\mathcal C(A_\infty)$.
There is not truly a problem with embedding \emph{too many} cluster-like objects.
Once the embedding has been established, one may take the subgroupoid of $\mathscr T_{\mathbf N_\infty}(\mathcal C(A_\infty))$ consisting of only those $\mathbf N_\infty$-clusters that are part of the cluster structure in \cite{HolmJorgensen}.
Thus there is still an embedding into the $\mathbf E$-cluster theory of $\CAR$.

To create the embedding of cluster theories $\mathscr T_{\mathbf N_\infty}(\mathcal C(A_\infty))\to \mathscr T_{\mathbf E}(\CAR)$ we need the following definitions adapted from \cite[Definition 3.2]{HolmJorgensen}.
\begin{definition}
Let $T$ be an $\mathbf N_\infty$-compatible set of arcs.
\begin{itemize}
\item If, for all $n\in\Z$, there are only finitely many arcs in the set $\{i\diag j \in T : i=n \text{ or }j=n\}$ we say $T$ is \ul{locally finite}.
\item If there exists $n\in \Z$ such that $\{i\diag j \in T: j=n\}$ is infinite we call this set of arcs a left-fountain.
\item If there exists $n\in\Z$ such that $\{i\diag j \in T: i = n\}$ is infinite we call this set of arcs a right-fountain.
\item If there exists $n$ that has both a left- and right-fountain we say $\{i\diag j\in T: i=n\text{ or }j=n\}$ is a fountain.
\end{itemize}
\end{definition}
The authors note in \cite[Lemma 3.3]{HolmJorgensen} that if a left- or right-fountain exists in a $\mathbf N_\infty$-cluster then it must be unique.
Just before the lemma the authors note that a left-fountain exists if and only if a right-fountain exists, crediting Collin Bleak.
This will become quite important.
It is also prudent to note that if there is a left-fountain at $m$ and a right-fountain at $n$ then $m\leq n$.

\begin{definition}\label{def:TEinfty}
We now define $T_{\mathbf E_\infty}$ given $T_{\mathbf N_\infty}$.
\begin{itemize}
\item If $T_{\mathbf N_\infty}$ is locally finite then \begin{displaymath} T_{\mathbf E_\infty} = T_{\mathbf E_\infty}^\circ.\end{displaymath}
\item If $T_{\mathbf N_\infty}$ has a left- or right-fountain it has the other. 
Let $m$ be the vertex with the left-fountain and $n$ the vertex with the right-fountain; note that it is possible $m=n$.
Set
\begin{displaymath} T_{\mathbf E_\infty} = T_{\mathbf E_\infty}^\circ \cup \{M_{(a_{-\infty},a_m)}, M_{(a_{-\infty},a_n)}, M_{(a_n,a_{+\infty})}\}. \end{displaymath}
\end{itemize}
\end{definition}

\begin{proposition}\label{prop:local finite means no skip}
Let $T_{\mathbf N_\infty}$ be an $\mathbf N_\infty$ cluster.
Suppose there exists $\ell\in\Z$ such that for all $i\diag j\in T_{\mathbf N_\infty}$, $\ell \leq i $ or $j\leq \ell$.
Then there exists a left- and right-fountain in $T_{\mathbf N_\infty}$.
\end{proposition}
\begin{proof}
For contradiction, suppose $T_{\mathbf N_\infty}$ is locally finite.
Let 
\begin{align*}
i_\ell&=\min_i \{i\geq \ell\in \Z : \exists\, i\diag j\in T_{\mathbf N_\infty}\} \\
j_\ell&=\max_j \{j\leq \ell\in\Z : \exists\, i\diag j\in T_{\mathbf N_\infty} \}.
\end{align*}
By the maximality of $T_{\mathbf N_\infty}$ we see that $0\leq i_\ell - j_\ell\leq 1$ and $\ell\in\{i_\ell, j_\ell\}$.
Since we have assumed $T_{\mathbf N_\infty}$ is locally finite, let
\begin{align*}
j_0 &= \max_j \{j\in\Z: i_\ell\diag j\in T\} \\
i_0 &= \min_i \{i\in\Z : i\diag j_\ell \in T\}.
\end{align*}
We will show $i_0\diag j_0\in T_{\mathbf N_\infty}$, contradicting our assumption about $\ell$.

For contradiction, suppose there exists $i \diag j\in T_{\mathbf N_\infty}$ such that $i_0 < i < j_0 < j$.
Since $i_0<i<j_0$, we must have $j_\ell \leq i \leq i_\ell$.
But then $i=i_\ell$ by our definition of $i_\ell$.
However $j_0< j$, contradiction our definition of $j_0$.
Thus there cannot be such a $i\diag j\in T_{\mathbf N_\infty}$.
Similarly, there can be no $i'\diag j'\in T_{\mathbf N_\infty}$ such that $i'<i_0<j'<j_0$.

This means $\{i_0\diag j_0 \}\cup T_{\mathbf N_\infty}$ is $\mathbf N_\infty$-compatible.
Since $T_{\mathbf N_\infty}$ is an $\mathbf N_\infty$-cluster we have $i_0\diag j_0\in T_{\mathbf N_\infty}$.
This contradicts our assumption about $\ell$ since $i_0<\ell < j_0$.
Therefore $T_{\mathbf N_\infty}$ is not locally finite; i.e.\ there exists a left- and right-fountain in $T_{\mathbf N_\infty}$.
\end{proof}

\begin{proposition}\label{prop:TEinfty is an E-cluster}
Let $T_{\mathbf N_\infty}$ be an $\mathbf N_\infty$-cluster.
Then $T_{\mathbf E_\infty}$ is an $\mathbf E$-cluster.
\end{proposition}
\begin{proof}
Recall $T_{\mathbf E_\infty}$ is an $\mathbf E$-compatible set.
Now suppose $M_{|c,d|}$ is an indecomposable in $\CAR$ and $T_{\mathbf E_\infty}\cup\{M_{|c,d|}\}$ is $\mathbf E$-compatible.
We will first assume $M_{|c,d|}$ is not of the form $M_{i\diag j}$ for any pair of integers $i<j$ where $j-i\geq 2$.
Similar to Proposition \ref{prop:TE is an E-cluster} we will check various possibilities for $c$.
The argument when $c<a_{-\infty}$ or $c\geq a_{+\infty}$ is identical to Proposition \ref{prop:TE is an E-cluster}.

Suppose $c=a_{-\infty}$.
Suppose $T_{\mathbf N_\infty}$ has a left- and right-fountain at $m$ and $n$, respectively.
Then $d =a_m$, $d=a_n$, or $d=a_{+\infty}$ since there is a left-fountain at $m$ and a right-fountain at $n$.
Note that if $d=a_{+\infty}$ then either $c=a_{-\infty}$ or $c=a_n$.
If $T_{\mathbf N_\infty}$ is locally finite then $c=a_{-\infty}$ if and only if $d=a_{+\infty}$.

We now suppose neither $c$ nor $d$ is in $\{a_{-\infty},a_{+\infty}\}$.
If $a_i < c < a_{i+1}$ for some $i\in\Z$ then $M_{|c,d|}=M_{\{c\}}$ or $M_{|c,d|}=M_{(a_{i,j,\ell}, a_{i,j+1,\ell})}$ for $0\leq \ell$ and $0\leq j < 2^\ell$ similar to the proof of Proposition \ref{prop:TE is an E-cluster}.
If $c = a_i$ for some $i\in \Z$ then, since we are still assuming $M_{|c,d|}\neq M_{i'\diag j'}$ for any $i'<i'+1<j'$, $M_{|c,d|}=M_{(a_i,a_{i,j,\ell})}$ for some $0\leq \ell$ and $0\leq j \leq 2^\ell$.

Now we finally suppose $M_{|c,d|}=M_{i\diag j}$ for some $i<i+1<j$.
For contradiction suppose $i\diag j\notin T_{\mathbf N_\infty}$.
Since $T_{\mathbf N_\infty}$ is an $\mathbf N_\infty$-cluster we know there exists $i'\diag j'\in T_{\mathbf N_\infty}$ such that $i<i'<j<j'$ or $i'<i<j'<j$.
But then by Proposition \ref{prop:noncrossings A_infty} $\{M_{i\diag j},M_{i'\diag j'}\}$ is not $\mathbf E$-compatible, a contradiction since $M_{i'\diag j'}\in T_{\mathbf E_\infty}$.
Therefore, in all possibilities, $M_{|c,d|}\in T_{\mathbf E_\infty}$ already and so $T_{\mathbf E_\infty}$ is an $\mathbf E$-cluster.
\end{proof}

\begin{lemma}\label{lem:mutation is preserved A_infty}
Consider an $\mathbf N_\infty$-cluster $T_{\mathbf N}$ and the induced $\mathbf E$-cluster $T_{\mathbf E_\infty}$.
Suppose $T_{\mathbf N_\infty} \to (T_{\mathbf N_\infty} \setminus \{i\diag j\} )\cup\{i'\diag j'\}$ is an $\mathbf N_\infty$-mutation.
Then $T_{\mathbf E_\infty} \to (T_{\mathbf E_\infty}\setminus \{M_{i\diag j} \})\cup\{M_{i'\diag j'}\}$ is an $\mathbf E$-mutation.
\end{lemma}
\begin{proof}
As with Lemma \ref{lem:mutation is preserved} it suffices to show $\{M_{i\diag j} ,M_{i'\diag j'}\}$ is not $\mathbf E$-compatible but $(T_{\mathbf E_\infty}\setminus \{M_{i\diag j} \})\cup\{M_{i'\diag j'}\}$ is $\mathbf E$-compatible.
Since $\{i\diag j, i'\diag j'\}$ is not $\mathbf N_\infty$-compatible we know $\{M_{i\diag j} ,M_{i'\diag j'}\}$ is not $\mathbf E$-compatible by Proposition \ref{prop:noncrossings A_infty}.
Since $T'_{\mathbf N_\infty} = (T_{\mathbf N_\infty} \setminus \{i\diag j\} )\cup\{i'\diag j'\}$ is an $\mathbf N_\infty$-cluster and one mutation cannot introduce or remove a left- or right-fountain we see $T'_{\mathbf E_\infty} = (T_{\mathbf E_\infty}\setminus \{M_{i\diag j} \})\cup\{M_{i'\diag j'}\}$ and $T'_{\mathbf E_\infty}$ is an $\mathbf E$-cluster by Proposition \ref{prop:TEinfty is an E-cluster}.
\end{proof}

\begin{theorem}\label{thm:infinite embedding}
There exists an embedding of cluster theories $(F,\eta):\mathscr T_{\mathbf N_\infty}(\mathcal C(A_\infty))\to \mathscr T_{\mathbf E}(\CAR)$.
\end{theorem}
\begin{proof}
By Lemma \ref{lem:mutation is preserved A_infty} we see that defining $F(T_{\mathbf N_\infty}):= T_{\mathbf E_\infty}$ and sending $\mathbf N_\infty$-mutations to the corresponding $\mathbf E$-mutations yields a functor $\mathscr T_{\mathbf N_\infty}(\mathcal C(A_\infty))\to \mathscr T_{\mathbf E}(\CAR)$.
As with Theorem \ref{thm:finite embedding} $F$ is an injection on clusters and mutations.

Let $\eta_{T_{\mathbf N_\infty}}:T_{\mathbf N_\infty}\to T_{\mathbf E_\infty}$ be defined by $\eta_{T_{\mathbf N_\infty}} (i\diag j)  := M_{i\diag j}$.
This is an injection by definition and by Lemma \ref{lem:mutation is preserved A_infty} the $\eta$'s commute with mutation.
Therefore $(F,\eta)$ is an embedding of cluster theories.
\end{proof}

\subsection{Embedding $\mathscr T_{\mathbf N_{\R}} (\mathcal C_{\pi}) \to \mathscr T_{\mathbf E}(\CAR)$}\label{sec:continuous embedding}
In this section we demonstrate how to embed the previous continuous cluster theory into the new continuous cluster theory.
Let $A_\R$ again have the straight orientation.

In \cite{IgusaTodorov1} the continuous cluster category $\mathcal C_\pi$ is the orbit category of the doubling of $\mathcal D_\pi$ category (Definition \ref{def:D_pi}) via almost-shift.
Two indecomposables $V$ and $W$ in $\mathcal C_\pi$ are defined to be compatible if and only if
\begin{displaymath}
\dim (\Ext(V,W)\oplus \Ext(W,V)) \leq 1.
\end{displaymath}
We will denote this compatibility condition by $\mathbf N_\R$.

Equivalently, $V$ and $W$ are not compatible in $\mathcal C_\pi$ if there exists $n\in\Z$ such that there is a rectangle contained in $\mathcal D_\pi$ with lower left corner and upper right corner equal to $V$ and $W[n]$ and at most one of the other two corners on the boundary.
I.e., there may be up to one point missing from the rectangle and it must be one of the corners not equal to $V$ of $W$.

Let $V$ and $W$ be indecomposables in $\mathcal C_\pi$, which come from indecomposables in $\mathcal D_\pi$.
Using the functor $G$ (Section \ref{sec:triangulated relationship}, Definition \ref{def:new to old functor}), we have a guide for where indecomposables from $\mathcal C_\pi$ should approximately be sent.
Let $P_a$ be a projective indecomposable from $\repAR$, where $a\neq+\infty$, as an indecomposable in degree 0 in $\DbAR$.
Then $\MM^b P_a = (\tan^{-1}a, \tan^{-1}a)$.

We will take our fundamental domain of $\mathcal C_\pi$ to be those indecomposables between the lines given by $M(0,y)$ and $M(x,\pi)$, including the $M(0,y)$ indecomposables and excluding the $M(x,\pi)$ indecomposables.
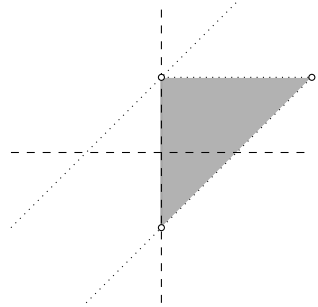
\begin{displaymath}\begin{tikzpicture}
\filldraw[fill=white!70!black, draw=white] (0,-1) -- (0,1) -- (2,1) -- (0,-1);
 \draw[white!65!black, thick] (0,-1) -- (0,1);
\draw[dashed] (-2,0) -- (2,0);
\draw[dashed] (0,-2) -- (0,2);
\draw[dotted] (-1,-2) -- (2,1);
\draw[dotted] (-2,-1) -- (1,2);
\draw[dotted] (0,1) -- (2,1);
\filldraw[fill=white] (0,-1) circle [radius=.4mm];
\filldraw[fill=white] (0,1) circle [radius=.4mm];
\filldraw[fill=white] (2,1) circle [radius=.4mm];
\end{tikzpicture}\end{displaymath}
Will send send each of the indecomposables in this fundamental domain to an indecomposable in $\CAR$.
Recall that for an interval $|a,b|$ we denote by $M_{|a,b|}$ the indecomposable in $\repAR$ (and its image in $\DbAR$ and $\CAR$) corresponding to the interval.
To avoid confusion in notation, we let ${_{(x,y)}M}$ denote the indecompoable in $\CAR$ that we obtain from $M(x,y)$ in $\mathcal C_\pi$.
Each of our ${_{(x,y)}M}$ indecomposables will be representatives chosen from the 0th degree in $\DbAR$.

The line segment (without its endpoints) from $(-\frac{\pi}{2},-\frac{\pi}{2})$ to $(\frac{\pi}{2},\frac{\pi}{2})$ in $\R\times[-\frac{\pi}{2},\frac{\pi}{2}]$ will be the image of the indecomposables in $\mathcal D_\pi$ of the form $M(0,y)$.
It is also the image of the projective indecomposables from $\repAR$ in the 0th degree, with the exception of $P_{+\infty}$, in $\DbAR$.
The dotted line bordering the fundamental domain in the picture are the indecomposables in $\mathcal D_\pi$ of the form $M(x,\pi)$.
These \emph{would} be sent to the line segment from $(\frac{\pi}{2},\frac{\pi}{2})$ to $(\pi,-\frac{\pi}{2})$.
This is precisely the image of the injectives from $\repAR$ in the 0th degree in $\DbAR$ under $\MM^b$.
As we've shown the rest of the the shaded triangle will then correspond to indecomposables in degree 0 in $\DbAR$ that (i) are not degenerate and (ii) are from neither projectives nor injectives in $\repAR$.

\begin{definition}\label{def:alpha beta a b}
For each $M(x,y)$ in the fundamental domain, set 
\begin{align*}
\alpha_{x,y} & =\frac{y+x}{2} & \beta_{x,y} &=\frac{y-x}{2}.
\end{align*}
It is straightforward to see that $-\frac{\pi}{2} < \beta_{x,y} < \frac{\pi}{2}$ and $\beta_{x,y} \leq \alpha_{x,y} < \pi-\beta_{x,y}$.

We define ${_{(x,y)}M}$ to be the indecomposable $M_{(a,b)}$ in $\CAR$ whose image under $\MM^b$ is $(\alpha_{x,y},\beta_{x,y})$.
We set
\begin{align*}
a_{x,y} &= \tan \left(\frac{\alpha_{x,y}-\beta_{x,y}-\pi}{2}\right) & b_{x,y} &= \tan \left(\frac{\alpha_{x,y}+\beta_{x,y}}{2}\right).
\end{align*}
Note it is possible that $a_{x,y}=-\infty$, in which case we have $\MM^b P_{b_{x,y}}$.
Then we set
\begin{displaymath}
{_{(x,y)}M} = M_{(a_{x,y}, b_{x,y})}.
\end{displaymath}
\end{definition}

We will use the following definition a few times.
\begin{definition}\label{def:A B f}
Let $\mathcal CA$, $\mathcal CB$, and $\mathcal CC$ be the sets below:
\begin{align*}
\mathcal CA &=  \{(x,y)\in \R^2 : |x-y|<\pi, x \geq 0, y < \pi \} \\
\mathcal CB &= \left\{(\alpha,\beta)\in \R^2: -\frac{\pi}{2} < \beta < \frac{\pi}{2} \text{ and } \beta \leq \alpha < \pi - \beta \right\} \\
\mathcal CC&= \{(a,b)\in (\R\cup\{-\infty\})\times \R: -\infty \leq a < b < +\infty\}.
\end{align*}
Using Definition \ref{def:alpha beta a b}, let
\begin{align*}
\mathfrak g: \mathcal CA &\to \mathcal CB \\
 (x,y) &\mapsto (\alpha_{x,y},\beta_{x,y}) \\
\mathfrak h: \mathcal CB &\to \mathcal CC \\
(\alpha_{x,y},\beta_{x,y}) &\mapsto (a_{x,y},b_{x,y}).
\end{align*}
Let $\mathfrak f:\mathcal CA \to \mathcal CC$ be the composite $\mathfrak h \circ \mathfrak g$.
For $i\in\{1,2\}$ define $\mathfrak g_i$, $\mathfrak h_i$, and $\mathfrak f_i$ to be the projection onto the $i$th coordinate.
\end{definition}

\begin{proposition}\label{prop:f is a bijection}
The function $\mathfrak f$ in Definition \ref{def:A B f} is a bijection.
\end{proposition}
\begin{proof}
We first show $\mathfrak f$ is well-defined.
The set $\mathcal CA$ is precisely the set corresponding to the fundamental domain of $\mathcal C_\pi$ we have chosen.
We know $\alpha_{x,y}$ and $\beta_{x,y}$ are defined in terms of $x$ and $y$.
Since $-\frac{\pi}{2} < \beta_{x,y} < \frac{\pi}{2}$ and $ \beta_{x,y} \leq \alpha_{x,y} < \pi - \beta_{x,y}$ we see that 
\begin{displaymath} -\pi \leq \alpha_{x,y} - \beta_{x,y} - \pi < \pi. \end{displaymath}
so $-\infty \leq a < +\infty$.
We also see that 
\begin{displaymath} \alpha_{x,y} - \beta_{x,y} - \pi < \alpha_{x,y} + \beta_{x,y}   <\pi. \end{displaymath}
Thus, $-\infty \leq a_{x,y} < b_{x,y} < +\infty$ and so $f(x,y)\in B$.

Now suppose $(x,y)\neq (x',y')$.
Then $(\alpha_{x,y},\beta_{x,y}) \neq (\alpha_{x,y}',\beta_{x,y}')$ and so $(a_{x,y},b_{x,y})\neq (a_{x,y}',b_{x,y}')$.
Thus, $\mathfrak f$ is injective.

Let $(a,b)\in \mathcal CC$.
Set
\begin{align*}
\alpha := \tan^{-1}b+\tan^{-1}a + \frac{\pi}{2} \\
\beta := \tan^{-1}b-\tan^{-1}a -\frac{\pi}{2}.
\end{align*}
We immediately see that $-\frac{\pi}{2} < \beta < \frac{\pi}{2}$.
It is straightforward to see that $\beta \leq \alpha < \pi - \beta$.
Then we define
\begin{align*}
x := \alpha - \beta \\
y := \alpha + \beta.
\end{align*}
We see $x \geq 0$, $y < \pi$, and $|x-y|<\pi$.
Thus $\mathfrak f(x,y) = (a,b)$ and so $\mathfrak f$ is surjective.
\end{proof}

\begin{lemma}\label{lem:compatible compatibility}
Let $M(x,y)$ and $M(x',y')$ be indecomposables in the fundamental domain of $\mathcal C_\pi$.
The set $\{M(x,y), M(x',y')\}$ is $\mathbf N_\R$-compatible if and only if $\{{_{(x,y)}M},{_{(x',y')}M}\}$ is $\mathbf E$-compatible.
\end{lemma}
\begin{proof}
First suppose $\{M(x,y), M(x',y')\}$ is not $\mathbf N_\R$-compatible.
Then $M(x',y)$ and $M(x,y')$ are indecomposables in the fundamental domain of $\mathcal C_\pi$.
This means ${_{(x',y)}M}$ and ${_{(x,y')}M}$ are well-defined indecomposables in $\CAR$ that, with ${_{(x,y)}M}$ and ${_{(x',y')}M}$, form a rectangle in the AR-space of $\DbAR$ which is entirely contained in the AR-space of $\repAR$.
Thus, $\{{_{(x,y)}M},{_{(x',y')}M}\}$ is not $\mathbf E$-compatible.

Now suppose $\{{_{(x,y)}M},{_{(x',y')}M}\}$ is not $\mathbf E$-compatible.
We reverse the argument and see that $M(x',y)$ and $M(x,y')$ are indecomposables in the fundamental domain of $\mathcal C_\pi$.
Therefore, $\{M(x,y), M(x',y')\}$ is not $\mathbf N_\R$-compatible.
\end{proof}

\begin{definition}\label{def:image of TNR}
Let $T_{\mathbf N_\R}$ be an $\mathbf N_\R$-cluster.
We define
\begin{align*}
T_{\mathbf E_\R}^\circ =& \{P_{+\infty}\}\cup \left\{ {_{(x,y)}M} = M_{\mathfrak f(x,y)} : M(x,y) \in T_{\mathbf N_\R} \right\} \\
& \cup \{ M_{\{z\}} :z \in\R,\, \not\exists M(x,y)\in T_{\mathbf N_\R} \text{ such that } (\mathfrak f_1(x,y)= z \text{ or } \mathfrak f_2(x,y)=z)\}
\end{align*}
\end{definition}
For each ${_{(x,y)}M}=M_{(a,b)}\in T_{\mathbf E_\R}^\circ$ we will define the set $\tau(a,b)$ that we will use to construct $T_{\mathbf E_\R}$.
\begin{definition}\label{def:satisfactory}
At $a$ we can check the following conditions.
\begin{enumerate}
\item There exist $M_{(c,a)}\in T_{\mathbf E_\R}^\circ$.
\item There exist $M_{(a,b')}\in T_{\mathbf E_\R}^\circ$ where $b' > b$.
\item $a=-\infty$.
\end{enumerate}
If $a$ satisfies any of (1), (2), or (3) we say it is \ul{satisfactory}.
We check similar conditions to (1) and (2) for $b$ and use the same definition of satisfactory.
Now we define $\tau(a,b)$.
\begin{itemize}
\item If both $a$ and $b$ are satisfactory let $\tau(a,b)=\emptyset$. 
\item If $a$ is satisfactory but $b$ is not let $\tau(a,b)=\{M_{(a,b]}\}$.
\item If $b$ is satisfactory but $a$ is not let $\tau(a,b)=\{M_{[a,b)}\}$. 
\item If neither $a$ nor $b$ are satisfactory let $\tau(a,b) = \{M_{[a,b]}, M_{[a,b)}\}$. 
\end{itemize}

We now define $T_{\mathbf E_\R}$ in one of two ways. 
Let $\mathcal P =\{ P_{b)}: P_{b)}\in T_{\mathbf E_\R}^\circ, b < +\infty \}$ with total order given by $P_{b)} \leq P_{b')}$ if and only if $b \leq b'$.
\begin{itemize}
\item If $\mathcal P$ is empty or has no maximal element then define
\begin{displaymath} T_{\mathbf E_\R} := \left( \bigcup_{M_{(a,b)}\in T_{\mathbf E_\R}^\circ} \tau(a,b) \right)\cup  T_{\mathbf E_\R}^\circ. \end{displaymath} 
\item If $\mathcal P$ is nonempty with a maximal element $P_{b)}$ then define 
\begin{displaymath} T_{\mathbf E_\R} := \left( \bigcup_{P_{b)}\neq M_{(a,b)}\in T_{\mathbf E_\R}^\circ} \tau(a,b) \right)\cup \{I_{(b}\} \cup T_{\mathbf E_\R}^\circ. \end{displaymath} 
\end{itemize}
\end{definition}

\begin{proposition}\label{prop:old cluster new cluster}
Let $T_{\mathbf N_\R}$ be an $\mathbf N_\R$-cluster.
Then $T_{\mathbf E_\R}$ is an $\mathbf E$-cluster
\end{proposition}
\begin{proof}
Using Lemma \ref{lem:compatible compatibility} it is straightforward to check that $T_{\mathbf E_\R}$ is $\mathbf E$-compatible.
Let $M_{|c,d|}$ be an indecomposable in $\CAR$ such that $T_{\mathbf E_\R}\cup\{M_{|c,d|}\}$ is $\mathbf E$-compatible.
We will show $M_{|c,d|}\in T_{\mathbf E_\R}$.

As before we can check various values for $c$ but in a different fashion than before.
First suppose there exists $M(x,y)\in T_{\mathbf N_\R}$ such that $\mathfrak f_1(x,y)=c$.
Then there is no $M(x',y')\in T_{\mathbf N_\R}$ such that $\mathfrak f(x,y) = (a',b')$ and $a' < d < b'$.
Thus, there exist $M(x',y')\in T_{\mathbf N_\R}$ such that $d=\mathfrak f_1(x',y')$ or $d=\mathfrak f_2(x',y')$.
Then $M_{\{c\}}$ and $M_{\{d\}}$ are not in $T_{\mathbf E_\R}$.

Now we have $\{M_{(c,d)}\}\cup T_{\mathbf E_\R}$ is $\mathbf E$-compatible.
By Proposition \ref{prop:f is a bijection} there is a $M(x'',y'')$ that is compatible with $T_{\mathbf N_\R}$ such that $\mathfrak f(x'',y'')=(c,d)$.
However, since $T_{\mathbf N_\R}$ is an $\mathbf N_\R$-cluster $M(x'',y'')\in T_{\mathbf N_\R}$ and so $M_{(c,d)}\in T_{\mathbf E_\R}$.
If $c\in |c,d|$ or $c=-\infty$ then there is no $M_{(a,c)}\in T_{\mathbf E_\R}$ and if $d\in |c,d|$ there is no $M_{(d,b)}\in T_{\mathbf E_\R}$.
Thus, either $|c,d|=(c,d)$ or $M_{|c,d|}\in \tau(c,d)$.
In either case $M_{|c,d|}\in T_{\mathbf E_\R}$.

Now suppose there is no $M(x,y)\in T_{\mathbf N_\R}$ such that $\mathfrak f_1(x,y)=c$.
For contradiction suppose $d> c$.
Then $\{M(\mathfrak f^{-1}(c,d))\} \cup T_{\mathbf N_\R}$ is not $\mathbf N_\R$-compatible and so by Lemma \ref{lem:compatible compatibility} there is $M(x',y')\in T_{\mathbf N_\R}$ such tha (i) $\mathfrak f(x',y')=(a,b)$ and (ii) $a<c<b<d$ or $c<a<d<b$.
Thus, $\{M_{|c,d|}\}\cup T_{\mathbf E_\R}$ is not $\mathbf E$-compatible, a contradiction.
Thus $d=c$ and $M_{|c,d|}=M_{\{c\}}$.
Then we know $d=c$ cannot be $\mathfrak f_2(x',y')$ for some $M(x',y')\in T_{\mathbf N_\R}$ or else $\{M_{\{c\}}\}\cup T_{\mathbf E_\R}$ would not be $\mathbf E$-compatible.
Therefore $M_{\{c\}}$ is already in $T_{\mathbf E_\R}^\circ \subset T_{\mathbf E_\R}$.
\end{proof}

\begin{lemma}\label{lem:preserves mutation AR}
Let $T_{\mathbf N_\R}$ be an $\mathbf N_\R$-cluster and $T_{\mathbf E_\R}$ the induced $\mathbf E$-cluster.
If $T_{\mathbf N_\R} \to (T_{\mathbf N_\R}\setminus \{M(x,y)\})\cup \{M(x',y')\}$ is an $\mathbf N_\R$-mutation then $T_{\mathbf E_\R} \to (T_{\mathbf E_\R}\setminus \{{_{(x,y)}M}\})\cup \{{_{(x',y')}M}\}$ is an $\mathbf E$-mutation.
\end{lemma}
\begin{proof}
By Lemma \ref{lem:compatible compatibility} we know $\{{_{(x,y)}M},{_{(x',y')}M}\}$ is not $\mathbf E$-compatible.
Thus it remains to show $(T_{\mathbf E_\R}\setminus \{{_{(x,y)}M}\})\cup \{{_{(x',y')}M}\}$ is an $\mathbf E$-cluster.
Without loss of generality we will assume $y' > y$ which implies $x' > x$, since $\{M(x,y),M(x',y')\}$ is not $\mathbf N_\R$-compatible.
This also means $x'-\pi \leq y$, $x-\pi < x'-\pi$, $x' \leq \pi + y$, and $\pi + y < \pi + y'$.

Let $M(w,z)\in T_{\mathbf N_\R}$ such that $(w,z)\neq (x,y)$.
Then $\{M(w,z), M(x,y)\}$ and\\ $\{M(w,z), M(x',y')\}$ are both $\mathbf N_\R$-compatible.
We will check various possibilities for $z$ starting with the highest possible values and and working down.
\begin{itemize}
\item If $y' < z < \pi$ then $w \leq x$ or $\pi + y' < w$.
\item If $z=y'$ then $w \leq x$ or $\pi + y \leq w < \pi + y'$.
\item If $y <z<y'$ then $\pi + y\leq w$.
\item If $z=y$ then $x'\leq w < \pi + y$.
\item If $x'-\pi < z < y$ then $x' \leq w$.
\item If $z=x'-\pi$ then $x\leq w<x'$.
\item If $x- \pi < z < x'-\pi$ then $x \leq w$.
\item If $z=x-\pi$ then $w<x$.
\item If $z < x-\pi$ then $w \leq x$.
\end{itemize}
In each of these cases $\{M(w,z),M(x',y),M(x,y')\}$ is $\mathbf N_\R$-compatible.
Thus,\\$M(x',y),M(x,y')\in T_{\mathbf N_\R}$.
We can further work with $M(w,z)$.
\begin{itemize}
\item We have $x \leq w \leq x'$ if and only if $x-\pi < z \leq x'-\pi$.
\item We have $\pi + y \leq w \leq \pi + y'$ if and only if $y  < z \leq y'$.
\end{itemize}
Thus $\{M(w,z), M(x,x'), M(y,y')\}$ is also $\mathbf N_\R$-compatible and so $M(x,x'),M(y,y')\in T_{\mathbf N_\R}$.

Let $(a,b)=\mathfrak f(x,y)$ and $(a',b')=\mathfrak f(x',y')$.
Then
\begin{align*}
(a',b) &= \mathfrak f(x',y) & (a,a') &= \mathfrak f(x,x') \\
(a,b') &= \mathfrak f(x,y') & (b,b') &= \mathfrak f(y,y').
\end{align*}
We see that for $\tau(a,b)$ both $a$ and $b$ are satisfactory.
Furthermore, for $\tau(a,a')$ and $\tau(b,b')$ all of $a$, $a'$, $b$, and $b'$ are satisfactory.
Thus $(T_{\mathbf E_\R}\setminus \{{_{(x,y)}M}\})\cup \{{_{(x',y')}M}\}$ is $\mathbf E$-compatible.
Furthermore, let $T'_{\mathbf N_\R} = (T_{\mathbf N_\R}\setminus \{M(x,y)\})\cup\{M(x',y')\}$.
Then 
\begin{displaymath}
T'_{\mathbf E_\R}  = (T_{\mathbf E_\R}\setminus \{{_{(x,y)}M}\})\cup \{{_{(x',y')}M}\}.
\end{displaymath}
Therefore $T'_{\mathbf E_\R}$ is an $\mathbf E$-cluster by Proposition \ref{prop:old cluster new cluster} and the lemma holds.
\end{proof}

\begin{theorem}\label{thm:continuous embedding}
There exists an embedding of cluster theories $(F,\eta):\mathscr T_{\mathbf N_\R}(\mathcal C_\pi)\to$ $\mathscr T_{\mathbf E}(\CAR)$.
\end{theorem}
\begin{proof}
By Proposition \ref{prop:old cluster new cluster} and Lemma \ref{lem:preserves mutation AR} we see that defining $F(T_{\mathbf N_\R}):= T_{\mathbf E_\R}$ and sending $\mathbf N_\R$-mutations to the corresponding $\mathbf E$-mutations yields a functor $\mathscr T_{\mathbf N_\R}(\mathcal C_\pi)\to \mathscr T_{\mathbf E}(\CAR)$.
It is straightforward to check $F$ is an injection on clusters and mutations.

Let $\eta_{T_{\mathbf N_\R}}:T_{\mathbf N_\R}\to T_{\mathbf E_\R}$ be defined by $\eta_{T_{\mathbf N_\R}} (M(x,y))  := {_{(x,y)}M}$.
This is an injection by definition and by Lemma \ref{lem:preserves mutation AR} the $\eta$'s commute with mutation.
Therefore $(F,\eta)$ is an embedding of cluster theories.
\end{proof}

\begin{remark}
We remark here that all clusters in \cite{IgusaTodorov1} are $\mathbf N_\R$-clusters.
However, there are $\mathbf N_\R$-clusters that are not part of the cluster structure in \cite{IgusaTodorov1}.
For example, the vertical line $\{M(0,y): -\pi < y < \pi\}$ is an $\mathbf N_\R$-cluster but not part of the cluster structure in $\mathcal C_\pi$.

However, as with $A_\infty$, this is not truly an issue.
We have an injection on objects and so taking a subgroupoid of $\mathscr T_{\mathbf N_\R}(\mathcal C_\pi)$ that only contains the clusters in the cluster structure in $\mathcal C_\pi$ still embeds into $\mathscr T_{\mathbf E}(\mathcal C_\pi)$ while preserving mutation.
This is done in the next paper in this series.
\end{remark}

\subsection{Issues with Functors Between Cluster Categories}\label{sec:problems}
In this short subsection we describe issues to any simple or straightforward embedding $\mathcal C_\pi\to\CAR$ that is somehow compatible with mutation.
The first issue is the following: any embedding of $\mathcal C_\pi \to \CAR$ cannot originate as a functor $\mathcal D_\pi\to \DbAR$.

Since they are triangulated equivalent, we will consider $\DbAR[\mathcal NQ^{-1}]$ instead of $\mathcal D_\pi$.
As an example, consider an embedding that sends all the indcemposables in $\mathcal D_\pi$ to the indecomposable in position 1 in the inverse image of the localization $\DbAR \to \DbAR[\mathcal NQ^{-1}]$.
This type of embedding will preserve the triangulated structure but cause problems with compatibility after taking the orbit category.
In $\mathcal C_\pi$, $M(x,y)$ and $M(\pi+y,y')$ are compatible if $y < y'$. 
However, using an embedding $\DbAR[\mathcal NQ^{-1}]\to \DbAR$ and then taking the orbit sends $M(x,y)$ and $M(\pi+y,y')$ to $M_{[a,b)}$ and $M_{[b, b')}$. 
These are not $\mathbf E$-compatible.
Thus any such embedding would have to embed to positions 2 and 3, using position 3 in even degrees and position 2 in odd degrees to preserve the triangulated structure.

This creates a new problem.
Consider $M(x,y)$ and $M(x,y')$ where $y < y'$.
Suppose $M(x,y)$ is sent to degree 0 and $M(x,y')$ to degree 1 in $\DbAR$.
But now the slope from $M_{(a,b)}$ to $M_{(c,a)}[1]$ is greater than $(1,1)$ in the AR-space of $\DbAR$ and so $\Hom(M_{(a,b)}, M_{(c,a)}[1])=0$.
While this doesn't necessarily prevent such a functor from being triangulated, or even preserving mutation in some way after taking orbits, this is no longer an embedding.

So we are left with a functor $\mathcal C_\pi\to\CAR$.
As we've seen, we need to send the fundamental domain to position 3 as we did in Section \ref{sec:continuous embedding}.
But then we have the same problem as we've just described with Hom support, and so we don't have an embedding.
Since $\Ext (V,W)=\Hom(W,V)$ in both $\mathcal C_\pi$ and $\CAR$ and compatibility in $\mathcal C_\pi$ is determined by $\Ext$ this is fundamentally a problem.
Therefore, the authors resorted to cluster theories in order to create an embedding that makes sense.

\end{document}